%% file: paper_initial_submission_to_SISC.tex
\documentclass[review,onefignum,onetabnum]{siamart190516}
\renewcommand{\theequation}{\thesection.\arabic{equation}}
\renewcommand{\thefigure}{\thesection.\arabic{figure}}
\renewcommand{\thetable}{\thesection.\arabic{table}}
       
\usepackage{amsmath} 
\usepackage{cleveref}
\usepackage{amsfonts}
\usepackage{amssymb}
\usepackage{multirow}
\usepackage{multicol}
\usepackage{stmaryrd}
\usepackage{subcaption}
\usepackage{makecell}
\usepackage{cite}
\usepackage{boldline}
\usepackage{graphics}
\usepackage{epstopdf}
\usepackage{arydshln}
\graphicspath{{./pic/}}
\usepackage{mathtools}
\usepackage{caption}
\usepackage{subcaption}
\usepackage{hyperref}  
\usepackage{tikz}

\input{ex_shared}

\newcommand{\dd}{\mathrm{d}}
\newcommand{\DG}{\mathrm{DG}}

\newcommand\dxDGa[1]{\nabla^{\mathrm{DG}}_{#1}\cdot}
\newcommand{\dxDGk}{\nabla^{\mathrm{DG}}_k\cdot}
\newcommand{\dxDGkm}{\nabla^{\mathrm{DG}}_{k-1}\cdot}

\newcommand{\hf}{\frac{1}{2}}
\newcommand{\ext}{\mathrm{ext}}
\newcommand{\inte}{\mathrm{int}}
\newcommand{\nm}[1]{\left\|#1\right\|}
\newcommand{\quand}{\quad \text{and} \quad}

\newcommand{\cfl}{\lambda}

\newcommand{\omt}{\overline{t}}
\newsiamremark{exmp}{Example}

\usepackage{comment}
\usepackage[belowskip=-2pt]{caption}

\newsiamremark{remark}{Remark}

\begin{document}

\maketitle 

\begin{abstract}
 In this paper, we present a novel class of high-order Runge--Kutta (RK) discontinuous Galerkin (DG) schemes for hyperbolic conservation laws. The new method extends beyond the traditional method of lines framework and utilizes stage-dependent polynomial spaces for the spatial discretization operators. To be more specific, two different DG operators, associated with $\mathcal{P}^k$ and $\mathcal{P}^{k-1}$ piecewise polynomial spaces, are used at different RK stages. The resulting method is referred to as the sdRKDG method. It features fewer floating-point operations and may achieve larger time step sizes. For problems without sonic points, we observe optimal convergence for all the sdRKDG schemes; and for problems with sonic points, we observe that a subset of the sdRKDG schemes remains optimal. We have also conducted von Neumann analysis for the stability and error of the sdRKDG schemes for the linear advection equation in one dimension. Numerical tests, for problems including two-dimensional Euler equations for gas dynamics, are provided to demonstrate the performance of the new method. \end{abstract}

\section{Introduction}

The development of the Runge--Kutta (RK) discontinuous Galerkin (DG) method was pioneered by Cockburn et al. in 
\cite{rkdg1,rkdg2,rkdg3,rkdg4,rkdg5,cockburn2001runge} for solving hyperbolic conservation laws. It possesses desirable properties, such as preservation of local conservation, good $hp$ adaptivity, and flexibility for fitting complex geometry, etc. The method has been widely recognized and become one of the mainstream numerical methods for computational fluid dynamics in recent decades.  

Originally, the construction of the RKDG method follows the common practice of the method of lines (MOL). The idea is to first apply the DG spatial discretization to convert the partial differential equations (PDEs) into an ordinary differential equation (ODE) system, and then use a high-order RK method for the time discretization. In the MOL approach, the same DG operator is used at all stages of the RK discretization, which is conceptually simple and straightforward to implement. 

The idea of MOL has guided the development and analysis of the DG schemes and led to a more profound understanding of the method. Firstly, in the MOL approach, one naturally expects that the fully discrete RKDG solutions converge to those of the semi-discrete DG scheme when the time step size approaches zero. Therefore, the analysis of the semi-discrete scheme, such as \cite{jiang1994cell,shu2009discontinuous,meng2016optimal,chen2017entropy,chan2018discretely,liu2020optimal}, may give indications on the performance of the fully discrete RKDG schemes with small time steps. Secondly, for the fully discrete cases with suitable RK methods, the MOL approach usually allows reducing the complicated high-order schemes into the same minimum module, such as the forward Euler steps and a single DG operator, which simplifies the design and analysis of the RKDG method \cite{rkdg2,zhang2010maximum,meng2016optimal,zhang2011third,xu2020superconvergence}. 

Despite various positive outcomes from MOL, it may be natural to ask: is MOL always a golden principle for the design of the RKDG schemes, and can we go beyond MOL and allow different spatial operators at different stages? On the one hand, recent works on RKDG schemes reveal that the RK time integration may not always preserve the properties of the semi-discrete MOL scheme. For example, the $L^2$ and entropy stability of the semi-discrete DG method may be violated in the fully discrete settings \cite{sun2019strong,xu20192,ranocha2020relaxation}. On the other hand, with theoretical tools developed \cite{sun2017stability,xu2024sdf} in recent years, people become capable to analyze fully discrete schemes with different spatial operators. 

There have been a few recent works on RKDG schemes beyond MOL, allowing different DG approximations for the spatial derivative at different temporal stages.  
In \cite{xu2024sdf}, Xu et al. studied the RKDG method with stage-dependent fluxes. Their work focuses on the linear hyperbolic equations in one dimension and considers stage operators involving different numerical fluxes. Rigorous stability and error analysis were presented in their manuscript. In our recent work \cite{chen2023runge}, we propose a compact RKDG method by hybridizing two different spatial operators in an explicit RK scheme. We reduce the stencil size over the RKDG method by utilizing local projected derivatives for inner RK stages and using the standard DG operators only at the final stage of each time step. The resulting method is more compact and features reduced data communications. This method also inherits many advantages of the standard RKDG method such as local conservation and high-order accuracy. From these works, it can be seen that replacing the DG operator with unstable or less accurate discretization at certain stages may not affect the overall stability and accuracy of RKDG schemes. 

Motivated by these recent works, we further explore the possibilities beyond the traditional MOL and propose an improved framework for the RKDG method. We construct a novel RKDG method by hybridizing  DG discretizations with different piecewise polynomial spaces within each time step, and therefore name it by the RKDG method with stage-dependent polynomial spaces (sdRKDG method). Our method is still based on the RK schemes, which can be either in the Butcher form or the Shu--Osher form (also known as the strong-stability-preserving (SSP) form \cite{gottlieb2001strong,gottlieb2011strong}). We mix two different DG operators into the RK stages, one uses the piecewise polynomial space of degree $k$, and the other uses the piecewise polynomial space of degree $k-1$. To obtain the $(k+1)$th order convergence rate,  we adopt a $(k+1)$th order RK method for time integration.   When the same DG operator is used at all stages, the sdRKDG method retrieves the standard RKDG method. Therefore, the sdRKDG method can be considered as a generalization of the standard RKDG method. 

Below, we summarize the properties of the new sdRKDG method:
\begin{itemize}
    \item (Stability and time step sizes.)
    The sdRKDG method may achieve improved Courant–Friedrichs–Lewy (CFL) conditions. We perform von Neumann analysis to show that, for several sdRKDG schemes in this paper, the new method allows larger time step sizes (30\% to 70\% larger) compared to standard RKDG schemes, which leads to enhanced efficiency and robustness. 
    \item (Accuracy with and without sonic points.) The von Neumann analysis shows that for the linear advection equation, the sdRKDG method achieves the same optimal order of accuracy as that of the standard RKDG scheme. The optimal convergence is also observed for nonlinear problems without sonic points. However, for problems with sonic points, some of the sdRKDG schemes may suffer accuracy reduction. See \cref{ex:1dburgers}.
    \item (Implementation and floating-point operations.) Since the sdRKDG method can be implemented by replacing selected $\mathcal{P}^k$ DG operator with the $\mathcal{P}^{k-1}$ DG operator, the sdRKDG code can be easily modified from an RKDG code in the modal implementation. Due to the use of a cheaper $\mathcal{P}^{k-1}$ DG operator, the sdRKDG method features fewer floating-points operations compared to the standard RKDG method. 
    \item (Inherited properties associated with cell averages.) It can be seen that, in each temporal stage, the sdRKDG method uses the same formula to update the cell average values as the standard RKDG method. Thus, the sdRKDG schemes also preserve local conservations. Moreover, with SSP-RK time integration and appropriate limiters, the sdRKDG method can also preserve the total-variation-diminishing (TVD) in the means (TVDM) property and the maximum principle.     
\end{itemize}

Before concluding this section, we would like to point to related literature outside of the RKDG framework on blending different approximations into a fully discrete RK scheme. Here we only provide an incomplete list of references due to the space limit. In the design of Lax--Wendroff DG methods \cite{guo2015new,sun2017stability}, it was discussed to use DG operators with both upwind and downwind fluxes in the approximations of the high-order spatial derivatives in the Cauchy--Kowalewski procedure. In the modified equation approach for the fully discrete schemes for the wave equation (see, for example,  \cite{henshaw2006high}, and references therein), it is aware that one can use high-order spatial operators for low-order terms in the Taylor expansion and use low-order operators for high-order terms. In \cite{tsai2014two},  Tsai et al. proposed the modified method of lines, in which they investigated the two-derivative RK method and proposed to use a different spatial operator to approximate the second-order derivative in the discretization. 
In recent works on the mixed-precision RK methods for ODE systems \cite{grant2022perturbed,croci2022mixed}, the error effect of mixing single-precision and double-precision operators at different stages of a RK scheme has been studied. 

The rest of the paper is organized as follows.  In \Cref{sec:construction}, we revisit the standard RKDG method and introduce the proposed sdRKDG method. In \Cref{sec:prop}, we discuss theoretical properties  and conduct von Neumann analysis of the sdRKDG schemes. Numerical tests are provided in  \Cref{sec:Numerical results} and the conclusions are drawn in \Cref{sec:conclusions}.

\section{Numerical schemes}
\label{sec:construction}

In this section, we start by briefly reviewing the RKDG method and then describe in detail the construction of the sdRKDG method. For ease of notation, we focus on scalar conservation laws, but the method can be extended to systems of conservation laws straightforwardly. 

\subsection{RKDG schemes}
Consider the hyperbolic conservation laws
\begin{equation}
	\partial_t u +\nabla \cdot f(u) = 0, \qquad u(x,0) = u_0(x).  
	\label{eq:conservation law}
\end{equation}
We denote by $\mathcal{T}_h = \{K\}$ the partition of the spatial domain in $d$ dimension. $h_K$ is the diameter of $K$ and $h = \max_{K\in \mathcal{T}_h} h_K$. We denote by $\partial K$ the boundary of $K$. For each edge $e \in \partial K$, $\nu_{e,K}$ is the unit outward normal vector along $e$ with respect to $K$. For a numerical approximation of the solution, we introduce the space $
	\mathcal{V}_h^\ell = \{v_h:v_h|_{K}\in \mathcal{P}^\ell(K),\, \forall K \in \mathcal{T}_h\}$,
where $\mathcal{P}^\ell(K)$ is the linear space of polynomials of degree up to $\ell$ on the cell $K$. The standard semi-discrete DG method for solving \eqref{eq:conservation law} is defined as follows: find $u_h \in \mathcal{V}_h^k$ such that on each $K\in \mathcal{T}_h$, 
\begin{equation}\label{eq:DGsemiweak}
	\int_K\left(u_{h}\right)_{t} v_h \dd x-\int_K f\left(u_{h}\right) \cdot \nabla v_h \dd x+\sum_{e\in \partial K} \int_{e} \widehat{f\cdot \nu_{e,K}} v_h \dd l   = 0, \quad \forall v_h\in \mathcal{V}_h^k,
\end{equation}
where $\widehat{f\cdot \nu_{e,K}}$ is the so-called numerical flux. In this paper, as an example, we usually take the Lax--Friedrichs type fluxes
\begin{equation*}
	\widehat{f\cdot \nu_{e,K}}  
        = \widehat{f\cdot \nu_{e,K}} (u_h) 
        = \hf\left(f(u_h^\inte)\cdot\nu_{e,K} + f(u_h^\ext)\cdot\nu_{e,K} - {\alpha_{e,K}}\left(u_h^\ext - u_h^\inte\right)\right), 
\end{equation*}
with $\alpha_{e,K} = \max \left|\partial_u f\cdot \nu_{e,K}\right|$. Here $u_h^\inte$ and $u_h^\ext$ are limits of $u_h$ along $e$ from the interior and exterior of the cell $K$, respectively. 

We introduce the discrete operator $\dxDGk f:\mathcal{V}_h^k\to \mathcal{V}_h^k$, defined by
\begin{equation}
	\int_K \dxDGk f(u_h) v_h \dd x = -\int_K f(u_h)\cdot \nabla v_h \dd x + \sum_{e\in \partial K} \int_e \widehat{f\cdot \nu_{e,K}} v_h \dd l, \quad \forall v_h \in \mathcal{V}_h^k.\label{eq:DG operator}
\end{equation}
Therefore the semi-discrete DG scheme \eqref{eq:DGsemiweak} can be rewritten in the strong form
\begin{equation}\label{eq:semiDG}
	\partial_t u_h + \dxDGk f(u_h) = 0.
\end{equation}

To discretize the ODE system \eqref{eq:semiDG} in time, we apply an explicit RK method associated with the Butcher tableau
\begin{equation*}
	\begin{array}{c|c}
		c&A\\\hline
		&b\\
	\end{array},\quad A = (a_{ij})_{s\times s}, \quad b = (b_1,\cdots, b_s), \quad c = (c_1,\cdots, c_s)^\intercal.
\end{equation*}
Here $A$ is a lower triangular matrix, namely, $a_{ij} = 0$ if $i\leq j$. The corresponding RKDG scheme is given by 
\begin{align*}
		u_h^{(i)} =&\, u_h^n -  \Delta t\sum_{j = 1}^{i-1}  a_{ij} \dxDGk f\left(u_h^{(j)}\right), \quad  i = 1, 2, \cdots, s,\\
		u_h^{n+1} =&\, u_h^n - \Delta t \sum_{i = 1}^s b_i \dxDGk  f\left(u_h^{(i)} \right).
	\end{align*}
Note we have $u_h^{(1)} = u_h^n$. 

 \begin{remark}[RK methods in Shu--Osher form]\label{rmk:shu-osher}
		We can also apply an $s$-stage RK method in Shu--Osher representation as in \cite{gottlieb2011strong} to discretize \eqref{eq:semiDG}  in time
	\begin{subequations} \label{rk:ssp way1}
		\begin{align}
			u_h^{(0)} & =u_h^n, \\
			u_h^{(i)} & =\sum_{j=0}^{i-1}\left(\alpha_{i j} u_h^{(j)}-\Delta t \beta_{i j} \dxDGk f\left(u_h^{(j)}\right)\right), \quad \alpha_{ij}\geq 0, \quad i=1, \ldots, s, \\
			u_h^{n+1} & =u_h^{(s)}.
		\end{align}		
	\end{subequations}
An RK method in Shu--Osher form is also known as an SSP-RK method when $\beta_{ij}\geq 0$. Notice that every RK method in the form of \eqref{rk:ssp way1} can be easily converted in a
unique way into the Butcher form \cite{gottlieb2011strong}. An RK method written in the
Butcher form can also be rewritten into the form \eqref{rk:ssp way1}, however, this conversion is
not unique in general.
\end{remark}

\subsection{sdRKDG schemes} 

Similar to \eqref{eq:DG operator}, we introduce a discrete operator denoted by  $\dxDGkm f: \mathcal{V}_h^k \to \mathcal{V}_h^{k-1}$ such that 
\begin{equation}\label{eq:dxl}
	\int_K \dxDGkm f\left(u_{h}\right) v_h \dd x=-\int_K f\left(u_{h}\right)\cdot \nabla v_h \dd x+\sum_{e\in \partial K} \int_e \widehat{f\cdot \nu_{e,K}} v_h \dd l, \quad \forall v_h \in \mathcal{V}_h^{k-1}.
\end{equation}
The difference between the standard DG operator $\dxDGk f$ in \eqref{eq:DG operator} and the new operator $\dxDGkm f$ in \eqref{eq:dxl} is that the polynomial order of the function space for the output changes from $k$ to $k-1$. Equivalently, one can see that 
\begin{equation}\label{eq:dxpl}
\dxDGkm f\left(u_{h}\right)=\Pi_{k-1}\, \dxDGk  f\left(u_{h}\right),
\end{equation}
where $\Pi_{k-1}$ is the standard $L^2$ projection to $\mathcal{V}_{h}^{k-1}$.

With $\dxDGkm f$ defined above, we propose our new sdRKDG method. Our main idea is to hybridize two different spatial operators $\dxDGkm f$ and $\dxDGk f$ at different temporal stages of an RKDG scheme. In the Butcher form, our sdRKDG scheme can be written as the following
\begin{subequations}\label{eq:sdRKDG}
	\begin{align}
		u_h^{(i)} =& u_h^n -  \Delta t\sum_{j = 1}^{i-1} a_{ij} \nabla^\DG_{d_{ij}}\cdot f\left(u_h^{(j)}\right), \quad  d_{ij}\in \{k-1,k\}, \quad  i = 1,2,\cdots, s, \label{eq:sdRKDG-in}
  \\
		u_h^{n+1} =& u_h^n - \Delta t\sum_{i = 1}^{s} b_i \nabla^\DG_{e_i}\cdot f\left(u_h^{(i)}\right), \quad e_i\in \{k-1,k\} \text{ for }i<s,\quad  e_s=k.
  \label{eq:sdRKDG-last}
	\end{align}
\end{subequations}
Here we introduce $D = (d_{ij})$ and $e = (e_i)$ to denote the spatial discretization used in each stage. Note that, although $\dxDGkm f(u_h^{(j)})\in\mathcal{V}_h^{k-1}$, we always have $u_h^{(i)}\in\mathcal{V}_h^{k}$. When $a_{ij} = 0$ or $b_i = 0$, the corresponding values of $d_{ij}$ and $e_i$ are nonessential, and we will leave them as blanks. For simplicity, we refer to the scheme \eqref{eq:sdRKDG} with an enlarged tableau as the following 
 \begin{equation*}
	\begin{array}{c|c:c}
		c&A&D\\\hline
		&b&e\\
	\end{array},\quad A = (a_{ij})_{s\times s},   D = (d_{ij})_{s\times s},   b = (b_1,\cdots, b_s),   e = (e_1,\cdots, e_s).
\end{equation*}

In the above scheme, the spatial operators and their corresponding polynomial spaces are stage-dependent. For this reason, we name our method as the sdRKDG method. In contrast, in the standard RKDG scheme, the same $\dxDGk f$ is used at all places, corresponding to \eqref{eq:sdRKDG} with $d_{ij} = e_i = k$ for all $i,j$. In this sense, the sdRKDG method can be considered as a generalization of the standard RKDG method. 

The question is  how the stage-dependent polynomial spaces will affect the stability and accuracy of the fully discrete scheme. As in the standard RKDG method, we will set $p = k+1$, where $p$ is the order of the RK time discretization. We expect $(k+1)$th order convergence under the CFL condition $\Delta t \leq \lambda h$. This rate is referred to as the optimal rate, and the corresponding RKDG scheme is referred to as the RKDG$(k+1)$ scheme when there is no confusion. Based on our von Neumann analysis and numerical experiments, we classify our sdRKDG schemes into the following two groups. 
\begin{enumerate}
    \item Class A:  $e_i = k$ for all defined $e_i$. Namely, in the last stage \eqref{eq:sdRKDG-last}, all stage operators are $\dxDGk f$. This class of schemes exhibits optimal convergence in all tested cases. The standard RKDG schemes also belong to this class.  
    \item Class B:  $e_{i_0} = k-1$ for some $i_0<s$. Namely, in the last stage \eqref{eq:sdRKDG-last}, $\dxDGkm f$ is used for some of the stage operators. According to our numerical tests, this class of schemes admits optimal convergence for problems without sonic points, but degenerates to $(k+1/2)$th order rate in $L^2$ norm for problems with sonic points. Despite the suboptimal convergence, this class of sdRKDG schemes can typically accommodate much larger CFL numbers.
\end{enumerate}
Note that in both Class A and Class B, we need to ensure that $e_s = k$, as indicated in \eqref{eq:sdRKDG-last}, for optimal convergence in cases without sonic points.

For clarity, we provide a few examples of sdRKDG schemes based on commonly used RK schemes. Note that, in general, even with the same RK method, there can be different sdRKDG schemes arising from various combinations of $\dxDGk f$ and $\dxDGkm f$. There are plenty of other sdRKDG schemes besides what we have listed here. 

\begin{itemize}
    \item (Class A, second-order) Explicit midpoint rule:
    \[\begin{array}{c|cc:cc}
0 & 0 & 0 &&\\
1 / 2 & 1 / 2 & 0 & 0&\\
\hline & 0 & 1& &1
\end{array}\]\vspace{-0.3625cm}
\begin{subequations}\label{eq:midpoint}
	\begin{align}
		u_h^{(2)} =& u_h^n - \frac{\Delta t}{2} \dxDGa{0} f\left(u_h^n\right),\\
		u_h^{n+1} =& u_h^n - \Delta t \dxDGa{1} f\left(u_h^{(2)}\right).
	\end{align} 
\end{subequations}  
    \item (Class A, third-order) Heun's method: \[\begin{array}{c|ccc:ccc}
0 & 0 & 0 & 0 &&&\\
1 / 3 & 1 / 3 & 0 & 0 &1&&\\
2 / 3 & 0 & 2 / 3 & 0&&1& \\
\hline & 1 / 4 & 0 & 3 / 4&2&&2
\end{array}\]\vspace{-0.3625cm}
\begin{subequations}\label{eq:heun}
	\begin{align}
		&u_h^{(2)} = u_h^n - \frac{1}{3} \Delta t \dxDGa{1} f\left(u_h^n\right),\quad
		u_h^{(3)} = u_h^n - \frac{2}{3}\Delta t \dxDGa{1}f\left(u_h^{(2)}\right),\\
		&u_h^{n+1} = u_h^n - \Delta t \left(\frac{1}{4}\dxDGa{2} f\left(u_h^{n}\right) + \frac{3}{4}  \dxDGa{2} f\left(u_h^{(3)}\right)\right).
	\end{align}
\end{subequations}
\item (Class B, second-order) The second-order SSP-RK (SSP-RK2) method:
\[\begin{array}{c|cc:cc}
0 & 0 & 0&   &   \\
1 &1 & 0 & 0 &  \\
\hline & 1/2 & 1/2 & 0 & 1
\end{array}\]\vspace{-0.3625cm}
\begin{subequations}\label{eq:SSPRK2}
	\begin{align}
		& u_h^{(2)} = u_h^n -  \Delta t\dxDGa{0} f\left(u_h^n\right),\\
		& u_h^{n+1} = \frac{1}{2}u_h^n +\frac{1}{2}\left(u_h^{(2)}- \Delta t \dxDGa{1} f\left(u_h^{(2)}\right)\right).
	\end{align}
\end{subequations}

\item (Class B, third-order) The third-order SSP-RK (SSP-RK3) method:
\[\begin{array}{c|ccc:ccc}
0 & 0 & 0 & 0 &&&\\
1   & 1   & 0 & 0 &1&&\\
1/2& 1/4 &   1/4&0&1&2& \\
\hline & 1 / 6 & 1/6 & 2/3&1&2&2
\end{array}\]\vspace{-0.3625cm}
\begin{subequations}
\label{eq:SSPRK3}
	\begin{align}
		& u_h^{(2)} = u_h^n -   \Delta t\dxDGa{1}  f\left(u_h^n\right),\\
		& u_h^{(3)} = \frac{3}{4}u_h^n +\frac{1}{4}\left(u_h^{(2)}- \Delta t \dxDGa{2}  f\left(u_h^{(2)}\right)\right),\\
		& u_h^{n+1} = \frac{1}{3} u_h^n +\frac{2}{3} \left(u_h^{(3)}-  \Delta t\dxDGa{2}  f\left(u_h^{(3)}\right)  \right).
	\end{align}
\end{subequations}
\item (Class B, fourth-order) The classic fourth-order RK (RK4) method: (The detailed scheme is omitted to save space.)
\begin{equation}\label{eq:rk4}
\begin{array}{c|cccc:cccc}
0 & 0 & 0 & 0 & 0&&&& \\
1 / 2 & 1 / 2 & 0 & 0 & 0&2&&& \\
1 / 2 & 0 & 1 / 2 & 0 & 0 &&2&&\\
1 & 0 & 0 & 1 & 0 &&&2&\\
\hline & 1 / 6 & 1 / 3 & 1 / 3 & 1 / 6&2&2&2&3
\end{array}
\end{equation}
\end{itemize}

\begin{remark}[Implementation]\label{Implementation} 
The implementation of the sdRKDG method can be easily modified from the standard RKDG code. Due to the relationship \eqref{eq:dxpl}, the $\dxDGkm f$ operators can be implemented by avoiding the update of the $k$th-order polynomial coefficients in $\dxDGk f$ under a set of orthogonal basis. For example, for $\mathcal{P}^k$ elements in two dimensions, we need to loop from $1$ to $(k+2)(k+1)/2$ for polynomial coefficients of $\dxDGk f$, while we only need to loop from $1$ to $(k+1)k/2$ for polynomial coefficients of $\dxDGkm f$. One simply needs to change the range of the {\it for} loop in the modal implementation of the standard RKDG method. 
 \end{remark}
 \begin{remark}[Efficiency improvement]
 The efficiency improvement of the new sdRKDG method over the standard RKDG method can be from two aspects. Firstly, the sdRKDG method may allow a larger CFL number for time marching (see \cref{fourier analysis}), which results in fewer time step iterations. Secondly, there are fewer floating-point operations for $\dxDGkm f$ compared to $\dxDGk f$ as stated in \cref{Implementation}. In the $d$-dimensional case with $\mathcal{P}^k$ elements, the floating-point operations of $\dxDGkm f$ versus $\dxDGk f$ is roughly $\dim(\mathcal{P}^{k-1}): \dim(\mathcal{P}^{k})= k:d+k$. However, as for whether the saving in floating-point operations can lead to savings in CPU time, it may depend on the implementation and the vectorization of the compiler. For example, we have tested the scheme \eqref{eq:midpoint} with Intel Fortran Compiler Version 2021.6.0 in one dimension. In this scheme, we evaluate $\dxDGa{0}f$ and $\dxDGa{1}f$ instead of two $\dxDGa{1}f$. The theoretical estimate of the number of floating-point operations of \eqref{eq:midpoint} is $75\%$ of that of the standard RKDG2 scheme. Without the auto-vectorization, the ratio of the CPU time is around $77\%$. However, when the auto-vectorization is on, this ratio rises to around $93\%$. It may be left as a future research topic on how one can optimize from the compiler level to take advantage of the saving in the floating-point operations of the sdRKDG method. 

 \end{remark}
 
 \begin{remark}[Sonic points] Numerically, we observe that an sdRKDG scheme in Class B may only attain a suboptimal convergence rate for problems with sonic points. See \cref{ex:1dburgers}. Note that, unfortunately, all sdRKDG schemes based on SSP-RK formulations could fall into this category. Further investigation into the cause of the order degeneration and the appropriate fix will be postponed to our future works. 
\end{remark}

\section{Properties of the sdRKDG method}\label{sec:prop}
The sdRKDG method inherits several properties of the standard RKDG method, which are discussed in \Cref{sec:conservation} and \Cref{sec:tvdm}. Moreover, we have applied the standard von Neumann approach to analyze the stability and error of the sdRKDG schemes for the linear advection equation in \cref{fourier analysis}. 
\subsection{Local conservation and convergence}\label{sec:conservation}
As all stages in \eqref{eq:sdRKDG} are using conservative DG spatial discretizations, the fully discrete numerical scheme is still conservative. Therefore, similar to the standard RKDG method, we can prove a Lax–Wendroff type theorem under standard assumptions. 
\begin{theorem}\label{thm:lwthm}
	Consider the sdRKDG scheme \eqref{eq:sdRKDG} for the hyperbolic conservation laws \eqref{eq:conservation law} on quasi-uniform meshes with the following assumptions:
	\begin{enumerate}
		\item $f$ is Lipschitz continuous.
		\item The numerical flux $\widehat{f\cdot\nu_{e,K}}(u_h)$ has the following properties: 
		\begin{enumerate}
			\item Consistency: if $u_h = u_h^\inte = u_h^\ext$, then $\widehat{f\cdot\nu_{e,K}}(u_h)= {f(u_h)\cdot\nu_{e,K}}$.
			\item Lipschitz continuity: $|\widehat{f\cdot\nu_{e,K}}(u_h) -  \widehat{f\cdot\nu_{e,K}}(v_h)|\leq C\nm{u_h-v_h}_{L^\infty(\widehat{B}_K)}$, where $\widehat{B}_{K} = K\cup K^\ext$ is the union of $K$ and its neighboring cell $K^\ext$. 
		\end{enumerate}
		\item The CFL condition $\Delta t/h \leq \lambda_0$ holds, where $\lambda_0$ is a fixed constant. 
	\end{enumerate}
	If $u_h^{n}$ converges boundedly almost everywhere to a function $u$ as $\Delta t, h \to 0$, then the limit $u$ is a weak solution to the conservation laws \eqref{eq:conservation law}, namely
	\begin{equation*}
		\int_{\mathbb{R}^d}u_0 \phi\dd x + \int_{\mathbb{R}^d\times \mathbb{R}^+} u \phi_t \dd x \dd t+ \int_{\mathbb{R}^d\times \mathbb{R}^+}f(u) \cdot \nabla \phi \dd x\dd t= 0, \quad \forall  \phi\in C_0^\infty(\mathbb{R}^d\times\mathbb{R}^+).
	\end{equation*} 
 
\end{theorem}

A sketch of the proof is given in \Cref{sec:conservation-proof}. The proof essentially follows the arguments in \cite[Appendix A]{chen2023runge}, with some additional technicalities in \cref{lem:lip}.
\begin{remark}
    As a corollary, the proof of \cref{thm:lwthm} also shows that the standard RKDG method in Butcher form, if it converges boundedly, will coverge to the weak solution of the conservation laws. 
\end{remark}

\subsection{TVDM property and maximum principle} \label{sec:tvdm}
Note that the forward Euler steps $u_h^{(i)} - \gamma_i \Delta t\dxDGkm f(u_h^{(i)})$ and $u_h^{(i)} - \gamma_i \Delta t \dxDGk f(u_h^{(i)})$ are identical in updating the cell average values. As a result, with the minmod limiter in \cite{rkdg2}, one can achieve the TVDM property for both forward Euler schemes per Harten's lemma; and with the scaling limiter in \cite{zhang2010maximum}, one can enforce the maximum principle. Therefore, if the sdRKDG scheme can be written as a convex combination of these first-order forward Euler steps, or equivalently, it is based on an SSP-RK method such as  \eqref{eq:SSPRK2} and \eqref{eq:SSPRK3},  then this sdRKDG scheme will maintain the TVDM property and preserve the maximum principle after applying appropriate limiters.  
\begin{theorem}\label{thm:tvdm}
    Consider the sdRKDG scheme for the scalar hyperbolic conservation laws in one dimension with periodic or compactly supported boundary conditions. If the sdRKDG scheme can be formulated as a convex combination of $u_h^{(i)} - \gamma_i \Delta t\dxDGkm f(u_h^{(i)})$ and $u_h^{(i)} - \gamma_i \Delta t \dxDGk f(u_h^{(i)})$,  then after each forward Euler step:
    \begin{enumerate}
        \item by applying the minmod limiter in \cite{rkdg2}, the solution $u_h^n$ of the sdRKDG scheme is TVDM, namely, $\mathrm{TVM}(u_h^{n+1})\leq \mathrm{TVM}(u_h^n)$, 
    where 
    $\mathrm{TVM}(v_h) = \sum_j |\bar{v}_{h,j+1} - \bar{v}_{h,j}|$ and $\bar{v}_{h,j}$ is the cell average of $v_h$ on the $j$th mesh cell.
    \item by applying the scaling limiter in \cite{zhang2010maximum}, the solution $u_h^n$ of the sdRKDG scheme preserves the maximum principle under a time step constraint, namely, if $u_h^{n} \in [m,M]$ then $u_h^{n+1} \in [m,M]$.
    \end{enumerate} 
\end{theorem}

The proof of \cref{thm:tvdm} is identical to that of the standard SSP-RKDG scheme as that in \cite[Section 3.2.2]{shu2009discontinuous} and \cite[Section 2]{zhang2010maximum}, and is hence omitted. 
\begin{remark}
    In practice, we often use the total variation bounded (TVB) limiter instead of the TVD limiter. It recovers high-order accuracy at smooth extrema that is failed by the TVD limiter, while maintaining total variation boundedness in the means (TVBM). Such TVBM property can also be preserved by the above-mentioned sdRKDG schemes.
\end{remark}

\subsection{Von Neumann Analysis}
\label{fourier analysis}

In this section, we apply the von Neumann analysis, also known as the Fourier-type analysis \cite{zhong2013simple}, to study the linear stability and numerical error of the new sdRKDG method with the upwind flux. 
 
\subsubsection{General settings and methodologies} Consider the linear advection equation in one dimension with periodic boundary condition
\begin{equation}
	\begin{aligned}
		u_t+  u_x=&0, \quad x \in[0,2 \pi] \quand t>0.
	\end{aligned}
\label{eq:linear wave for fourier}
\end{equation}

Let $[0,2\pi]=\bigcup_{j = 0}^{N-1} I_j$, with $I_j=(x_j - h/2, x_j + h/2)$, $x_j = jh$ and $h = 2\pi/N$, be a uniform partition of the domain. As before,  the approximation space for the solution is $
\mathcal{V}_h^k=\left\{v_h:\left.v_h\right|_{I_j} \in \mathcal{P}^k\left(I_j\right),  0 \leq j \leq N-1\right\}$. The functions $v_h\in \mathcal{V}_h^k$ can be represented as 
\begin{equation*}
v_h(x)=\sum_{l=0}^k v_{l, j} \phi_{l, j}(x), \quad x \in I_j,
\end{equation*}
where $\{\phi_{l, j}(x)\}_{l=0}^k$ is a set of orthonormal basis of $\mathcal{P}^k(I_j)$ and $v_{l, j}$ are the coefficients. We denote by $\boldsymbol{v}_{j} = (v_{0,j},\cdots, v_{k,j})^\intercal$ and will use similar notations for $\boldsymbol{u}_j(t)$ and $\boldsymbol{u}_j^n$, which represent the coefficients of $u_h(t)$ and $u_h^n$ in the cell $I_j$. 

We first review the von Neumann analysis of a standard RKDG scheme. With the upwind flux, the semi-discrete DG scheme \eqref{eq:semiDG} can be written as 
\begin{equation}\label{eq:matrixeq}
\frac{\dd }{\dd t}\boldsymbol{u}_j=\frac{1}{h}\left(C_{-1}\boldsymbol{u}_{j-1} + C_0 \boldsymbol{u}_j \right),
\end{equation}
where $C_{-1}$ and $C_0$ are $(k+1) \times(k+1)$ constant matrices. Note \eqref{eq:matrixeq} can be formally seen as a semi-discrete finite difference scheme. Then one can apply the classical von Neumann analysis. To this end, we assume the ansatz with the wave number $\omega$
\begin{equation*}
\boldsymbol{u}_j(t)=\hat{\boldsymbol{u}}(t) \exp \left(\mathrm{i} \omega x_j\right), \quand  \mathrm{i} \text{ is the imaginary unit.}
\end{equation*}
Now \eqref{eq:matrixeq} gives a simple ODE system
\begin{equation*}
\frac{\dd }{\dd t}{\hat{\boldsymbol{u}}}=D_k {\hat{\boldsymbol{u}}},\quad \text{with}\quad D_k = \frac{1}{h}\left(C_{-1} e^{-\mathrm{i}\xi}+C_0\right)\quand \xi = wh. 
\end{equation*}
Applying the time discretization, the ansatz coefficients of the standard RKDG scheme admit the following relationship: 
	\begin{align*}
		\hat{\boldsymbol{u}}^{(i)} =& \hat{\boldsymbol{u}}^n -  \Delta t\sum_{j = 1}^{i-1} a_{ij} D_k \hat{\boldsymbol{u}}^{(i)}, \quad i = 1,\cdots, s,\\
		\hat{\boldsymbol{u}}^{n+1} =& \hat{\boldsymbol{u}}^n - \Delta t\sum_{i = 1}^{s} b_i D_k \hat{\boldsymbol{u}}^{(i)}.
	\end{align*}
We can write it in the short form $\hat{\boldsymbol{u}}^{n+1} = R_k\hat{\boldsymbol{u}}^{n}$. 

Next, we consider the von Neumann analysis for the sdRKDG scheme. In above, $D_k$ is a $(k+1)\times (k+1)$ matrix that corresponds to the operator $\dxDGk f$. For the sdRKDG method, we will also need to introduce a matrix $D_{k-1}$ for the operator $\dxDGkm f$. From \eqref{eq:dxpl}, it can be seen that 
\begin{equation*}
    D_{k-1} = \diag(\underbrace{1,\cdots,1}_{k},0) D_k.
\end{equation*}
With this new operator, we have 
	\begin{align*}
		\hat{\boldsymbol{u}}^{(i)} =& \hat{\boldsymbol{u}}^n -  \Delta t\sum_{j = 1}^{i-1} a_{ij} D_{d_{ij}} \hat{\boldsymbol{u}}^{(i)},\qquad  d_{ij} \in \{k-1,k\}, \quad i = 1,\cdots, s,\\
		\hat{\boldsymbol{u}}^{n+1} =& \hat{\boldsymbol{u}}^n - \Delta t\sum_{i = 1}^{s} b_i D_{e_i} \hat{\boldsymbol{u}}^{(i)},\qquad  e_{i} \in \{k-1,k\} \text{ for } i<s \quand e_s = k.
	\end{align*}
This relation can be written in the short form as $\hat{\boldsymbol{u}}^{n+1} = R\hat{\boldsymbol{u}}^{n}$. Note that $R = R(\lambda,\xi)$ depends on two parameters, $\lambda = \Delta t/h$ and $\xi = \omega h$. 

For stability analysis in \Cref{sec:cfl}, we want to find the CFL number $\lambda_0$ such that the spectral radius of $R$ satisfies $\rho(R(\lambda,\xi)) \leq 1$ for all $\xi$. In other words, 
\begin{equation}\label{eq:lambda0}
    \lambda_0 = \sup_{\lambda\geq 0} \{\lambda:\sup_{\xi} \rho(R(\tilde{\lambda},\xi)) \leq 1, \,\, \forall \tilde{\lambda}\leq \lambda\}.
\end{equation}

For error analysis in \Cref{sec:full discrete}, we describe  the procedures for the $\mathcal{P}^1$ DG method as an example and adopt the normalized Legendre basis, 
\begin{equation*}
    \phi_{0,j}(x)=\frac{1}{\sqrt{h}}\quand  \phi_{1,j}(x)=\sqrt{\frac{{3}}{h}}\frac{x-x_j}{h/2}.
\end{equation*}
We would like to use the interpolated values at $x_{j\pm 1/4}$ as the initial condition and evaluate the error at the final time $t^n = n\Delta t$ also at $x_{j\pm 1/4}$. Note when we take $j = 0,\cdots, N-1$, and this will return the error on a uniform grid. 

Let us consider \eqref{eq:linear wave for fourier} with the initial condition $u(x,0) = \exp(\mathrm{i}\omega x)$. Its exact solution is 
\begin{equation}\label{eq:ujp1/4}
\begin{pmatrix}
u\left(x_{j-\frac{1}{4}},t^n\right)   \\
u\left(x_{j+\frac{1}{4}},t^n\right)    
\end{pmatrix}
 = 
 \begin{pmatrix}
     e^{-\mathrm{i}\frac{\xi}{4}}\\
     e^{\mathrm{i}\frac{\xi}{4}}
 \end{pmatrix}e^{\mathrm{i}(\omega x_{j} - t^n)}.    
\end{equation}
To obtain the numerical solution, we introduce the Vandermonde matrix 
\begin{equation*}
    V = \begin{pmatrix}
    \phi_{0,j}\left(x_{j-\frac{1}{4}}\right) & \phi_{1,j}\left(x_{j-\frac{1}{4}}\right)\\
    \phi_{0,j}\left(x_{j+\frac{1}{4}}\right) & \phi_{1,j}\left(x_{j+\frac{1}{4}}\right)
\end{pmatrix}.
\end{equation*}
Recall that $\hat{\boldsymbol{u}}^n =  R^n \hat{\boldsymbol{u}}^0$. By transforming back and forth between the nodal values and the modal coefficients, we have 
\begin{equation}\label{eq:uhjp1/4}
\begin{aligned}
    \begin{pmatrix}
    u_{j-\frac{1}{4}}^n\\
    u_{j+\frac{1}{4}}^n
\end{pmatrix}
        = V\boldsymbol{u}_j^n = V \hat{\boldsymbol{u}}^n
         e^{\mathrm{i}\omega x_{j}}
    = V R^n \hat{\boldsymbol{u}}^0e^{\mathrm{i}\omega x_{j}}= V R^n V^{-1}
    \begin{pmatrix}
     e^{-\mathrm{i}\frac{\xi}{4}}\\
     e^{\mathrm{i}\frac{\xi}{4}}        
    \end{pmatrix}e^{\mathrm{i}\omega x_{j}}.
\end{aligned}
\end{equation}
Consider the eigen-decomposition $R = Q\Lambda Q^{-1}$ with $\Lambda$ being a diagonal matrix. Then $R^n = Q\Lambda^nQ^{-1}$ can be computed. Subtracting \eqref{eq:ujp1/4} from \eqref{eq:uhjp1/4}, it gives
\begin{equation*}
\begin{aligned}
\begin{pmatrix}
    \varepsilon_{j-\frac{1}{4}}\\
    \varepsilon_{j+\frac{1}{4}}
\end{pmatrix}=
    \begin{pmatrix}
    u_{j-\frac{1}{4}}^n-u\left(x_{j-\frac{1}{4}},t^n\right)\\
    u_{j+\frac{1}{4}}^n-u\left(x_{j+\frac{1}{4}},t^n\right)
\end{pmatrix}
    = \left(V Q\Lambda^n(VQ)^{-1} -e^{-\mathrm{i}t^n} I\right)
    \begin{pmatrix}
     e^{-\mathrm{i}\frac{\xi}{4}}\\
     e^{\mathrm{i}\frac{\xi}{4}}
    \end{pmatrix}e^{\mathrm{i}\omega x_{j}}.
\end{aligned}
\end{equation*}
Note that $\varepsilon_{j\pm 1/4} = \varepsilon_{\pm 1/4} e^{\mathrm{i}\omega (x_j-x_0)}$. Hence one can use 
\begin{equation}\label{eq:globerr} 
{\varepsilon}_{\star,1} :=\max_{0\leq j \leq N-1} \left|\varepsilon_{j\pm\frac{1}{4}}\right|=\max_{0\leq j \leq N-1} \left|\varepsilon_{\pm 1/4} e^{\mathrm{i}\omega (x_j-x_0)}\right| = \max \left\{\left|\varepsilon_{-\frac{1}{4}}\right|, \left|\varepsilon_{\frac{1}{4}}\right|\right\}
\end{equation}
as a prediction of the numerical error. For the $\mathcal{P}^2$ case, we can similarly compute $\varepsilon_{-1/3}$, $\varepsilon_0$, and $\varepsilon_{1/3}$, and use the following value
as a prediction of the numerical error 
\begin{equation*}
    \varepsilon_{\star,2} = \max_{0\leq j\leq N-1}\left\{ \left|\varepsilon_{-\frac{1}{3}}\right|, \left|\varepsilon_{0}\right|, \left|\varepsilon_{\frac{1}{3}}\right|\right\}.
\end{equation*}

\subsubsection{CFL conditions for stability}\label{sec:cfl} 

In this section, we present the CFL condition for linear stability of several sdRKDG schemes following the method in \eqref{eq:lambda0}. 

\emph{Generic second-order method.} We first consider schemes with a generic second-order RK time discretization with different combinations of the stage operators.  
\begin{subequations}
\begin{alignat}{2}
\label{genericRk2}
&\text{v1: }\begin{tabular}{c|cc:cc}
0 & 0 & 0&   &   \\
$\alpha$ & $\alpha$ & 0 & $0$ &  \\
\hline & $1-\frac{1}{2 \alpha}$ & $\frac{1}{2 \alpha}$ & $0$ & $1$
\end{tabular}
\qquad
&\text{v2: }
\begin{tabular}{c|cc:cc}
0 & 0 & 0&   &   \\
$\alpha$ & $\alpha$ & 0 & $0$ &  \\
\hline & $1-\frac{1}{2 \alpha}$ & $\frac{1}{2 \alpha}$ & $1$ & $1$
\end{tabular}
\\
&\text{v3: }
\begin{tabular}{c|cc:cc}
0 & 0 & 0&   &   \\
$\alpha$ & $\alpha$ & 0 & $1$ &  \\
\hline & $1-\frac{1}{2 \alpha}$ & $\frac{1}{2 \alpha}$ & $0$ & $1$
\end{tabular}
\qquad 
&\text{v4: }
\begin{tabular}{c|cc:cc}
0 & 0 & 0&   &   \\
$\alpha$ & $\alpha$ & 0 & $1$ &  \\
\hline & $1-\frac{1}{2 \alpha}$ & $\frac{1}{2 \alpha}$ & $1$ & $1$
\end{tabular}
\end{alignat}
\end{subequations}
Here $\alpha$ is a parameter. In particular, when $\alpha=1$, the RK scheme corresponds to the SSP-RK2 method; when $\alpha={1}/{2}$, the RK scheme corresponds to the explicit midpoint rule. Moreover, the methods v1 and v3 belong to Class B, the methods v2 and v4 belong to Class A, and the method v4 retrieves the standard RKDG method. 

On the left of \cref{fig:CFL}, we plot the maximum CFL number versus the parameter $\alpha$. Note that when $e = (1,1)$, namely for v2 and v4, the CFL number is a constant $0.333$ and is independent of $\alpha$  due to the same matrix $R$ regardless of $\alpha$. The scheme v1 admits a larger maximum CFL number compared to the standard RKDG method. For example, when $\alpha = 1$, the v1 sdRKDG scheme admits the CFL number $0.566$  and is around $70\%$ larger than that of the standard RKDG method. The maximum CFL numbers of v3 sdRKDG schemes are typically smaller than those of the standard RKDG schemes.  

\emph{Third-order method.} The third-order RK method is typically parameterized by more than one parameter, and there are $2^5 = 32$ different possible combinations of the stage operators. Here as an example, we consider a family of RK methods in \cite{sanderse2019constraint} and consider two different combinations of the stage operators:
\begin{equation*}
\begin{array}{c|ccc}
0 & 0 & 0 & 0 \\
\alpha & \alpha & 0 & 0\\
1 & 1+\frac{1-\alpha}{\alpha(3 \alpha-2)} & -\frac{1-\alpha}{\alpha(3 \alpha-2)} & 0 \\
\hline & \frac{1}{2}-\frac{1}{6 \alpha} & \frac{1}{6 \alpha(1-\alpha)} & \frac{2-3 \alpha}{6(1-\alpha)}
\end{array}\qquad 
\text{v1: } 
\begin{array}{c:ccc}
&   &   &   \\
& 2 &   &   \\
& 1 & 1 &  \\
 \hline
& 2 & 2 & 2
\end{array}
\qquad
\text{v2: }
\begin{array}{c:ccc}
&   &   &   \\
& 1 &   &   \\
& 1 & 1 &  \\
\hline
& 1 & 1 & 2
\end{array}    
\end{equation*}
The method v1 belongs to Class A, and v2 belongs to Class B.
The maximum CFL number versus the parameter $\alpha$ is plotted on the right of \cref{fig:CFL}. It can be seen that, on the region $-0.5\leq \alpha\leq 0.6$, the CFL number of the standard RKDG scheme is $0.209$ and is independent of $\alpha$, the CFL number of the v1 sdRKDG scheme attains its maximum $0.262$ when $\alpha = 0.15$, and  the CFL number of the v2 sdRKDG scheme attains $0.333$ when $\alpha = -0.5$. The improvement of the CFL number can be up to $30\%$ and $60\%$ for v1 and v2 sdRKDG schemes, respectively.  
\begin{figure}[!ht] 
	\centering
 \begin{subfigure}[t]{.45\textwidth}
			\centering
			 \includegraphics[trim=0cm 0cm 0cm 0cm,width=0.8\textwidth]{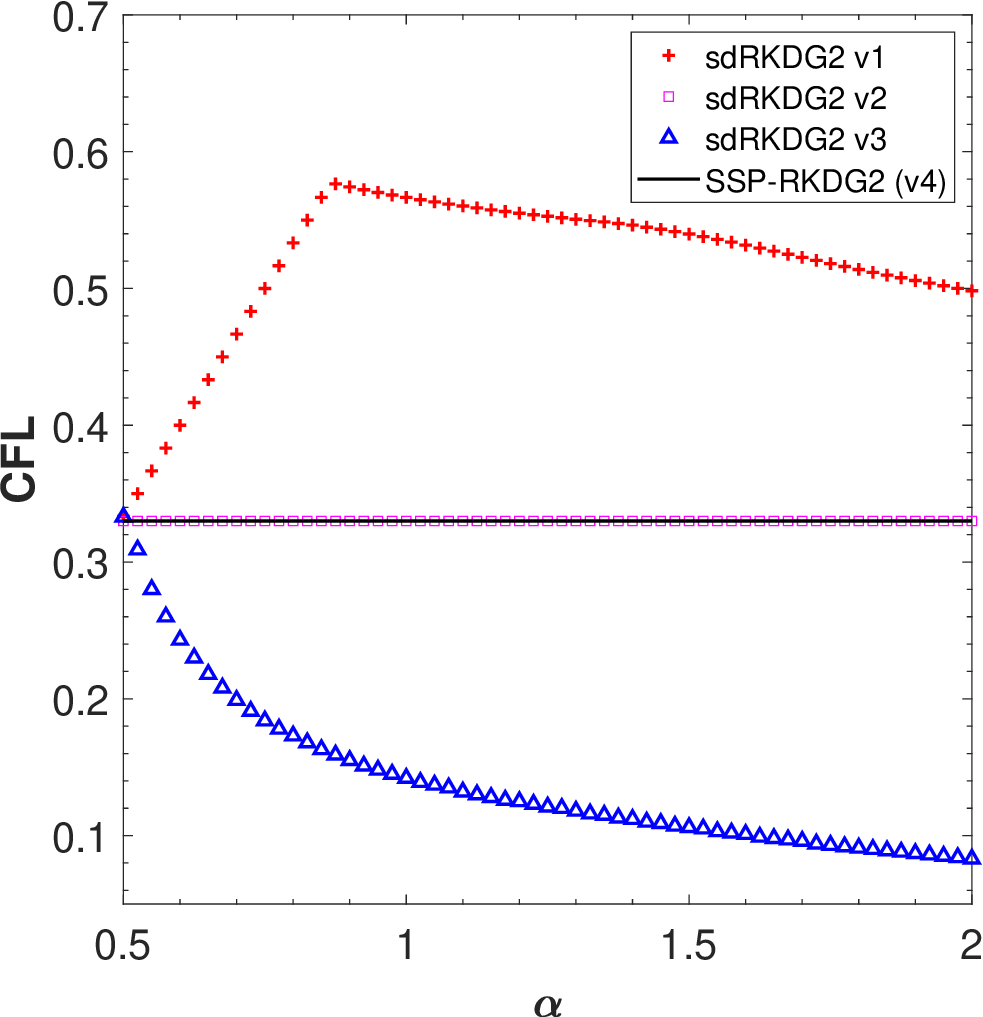}	 
		\end{subfigure}%
		\hspace{0mm}
		\begin{subfigure}[t]{.45 \textwidth}
  \centering
			\includegraphics[trim=0cm 0cm 0cm 0cm,width=0.8\textwidth]{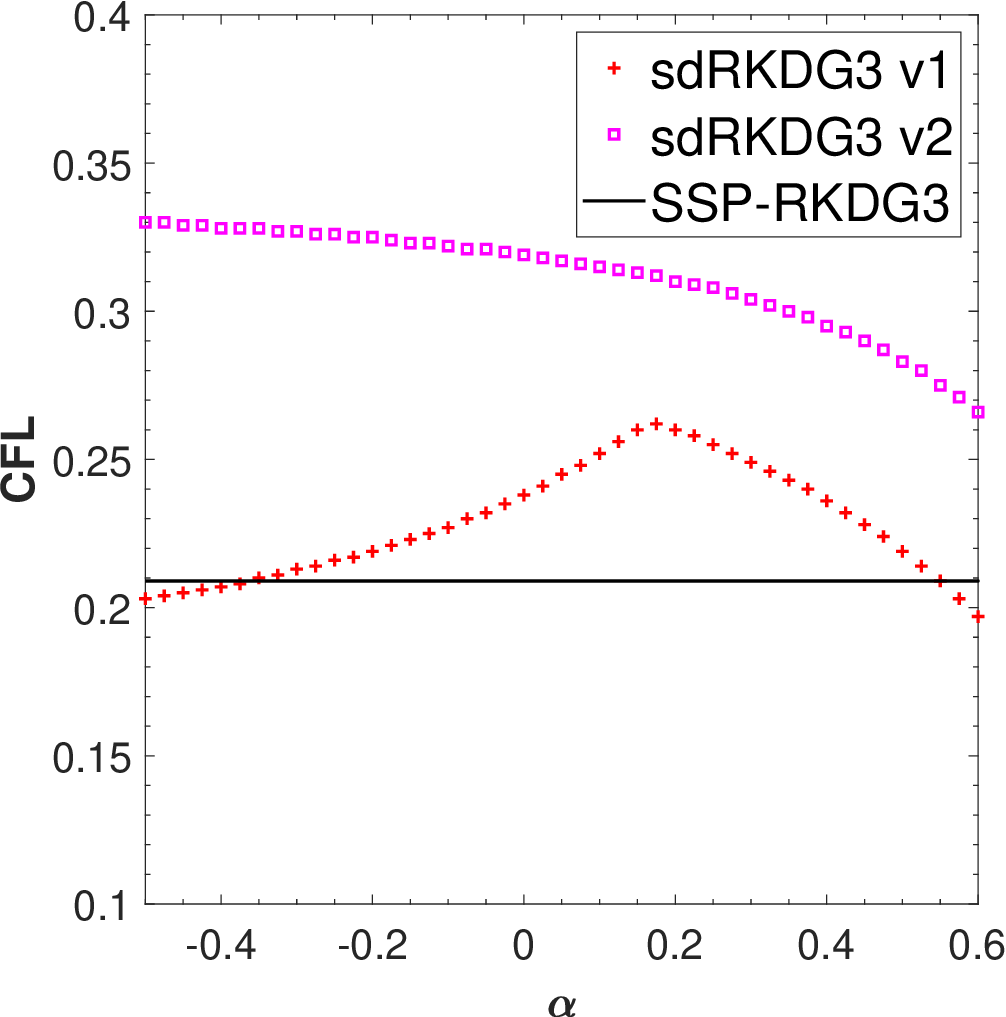}
		\end{subfigure} 
  \caption{CFL number of second-order (left) and  third-order (right) sdRKDG schemes.  }\label{fig:CFL}\vspace{-0.75cm}
      \end{figure}

\emph{Other methods.} Here we also provide the  CFL numbers for a few other sdRKDG schemes we have discussed.

For the third-order Heun's method, the sdRKDG scheme \eqref{eq:heun} (Class A) has the  CFL number $0.191$, and that of the corresponding standard RKDG method is $0.209$. 

For the SSP-RK3 method, the sdRKDG scheme \eqref{eq:SSPRK3} (Class B) has the  CFL number $0.275$, and that of the corresponding standard RKDG method is $0.209$. 

For the classic RK4 method, the sdRKDG scheme \eqref{eq:rk4} (Class B) has the  CFL number $0.213$, and that of the corresponding standard RKDG method is $0.145$. 
 
\subsubsection{Fully discrete error estimate}\label{sec:full discrete}

Here we take the generic second-order method v1 in \eqref{genericRk2} as an example. By a   Taylor expansion through Mathematica, the error between the DG solution by scheme v1  and the exact solution at time $t$ is 
\begin{equation}\label{eq:vareps1/4}
\begin{aligned}
&\varepsilon_{\frac{1}{4}}=  -\frac{1}{12}\xi^2 \mathrm{i} e^{-\mathrm{i} \omt} ((1-3 \cfl+2 \cfl^2) \omt+\alpha (-1+3 \cfl) (\mathrm{i}+2 \omt))+\mathcal{O}(\xi^3)+
 \\&
\left(1-\frac{3\cfl}{\alpha}\right)^{\frac{{t}}{\cfl h}}e^{\frac{3 \mathrm{i} \omt}{2\alpha-6 \cfl}+\mathcal{O}(\xi)}\left(\frac{\alpha}{12}\xi^2  (1-3\cfl) -\mathrm{i} \xi^3(\frac{1}{64}+\frac{1}{24}\alpha(\cfl-1)+\frac{1}{12}\alpha^2(3\cfl-1)) \right).
\end{aligned}
\end{equation} 
where $\xi=\omega h$ and $\omt =\omega t$. 

Assuming $0<\cfl/\alpha<{2}/{3}$, in this case, the term with the coefficient $\left(1-3 \cfl/\alpha\right)^{\frac{  t}{\cfl h}}$ goes to $0$ as $h \rightarrow 0$. After dropping this term in \eqref{eq:vareps1/4}, we have
\begin{equation*}
\left|\varepsilon_{\frac{1}{4}}\right| = \frac{\sqrt{\alpha^2(1 - 3 \cfl)^2 + (1 -2\alpha- 3 \cfl+6\alpha\cfl + 2 \cfl^2)^2 \omt^2}}{12} \xi^2  + \mathcal{O}(\xi^3),
\end{equation*}
Similarly, we can compute $ \left|\varepsilon_{-{1}/{4}}\right|$. Therefore   when $\xi$ is small and $0<\cfl/\alpha<{2}/{3}$, we have an estimate of the numerical error of the scheme \eqref{eq:SSPRK2} for \eqref{eq:linear wave for fourier}:
\begin{equation}
\begin{aligned}
\varepsilon_{\star,1}= \max \left\{\left|\varepsilon_{-\frac{1}{4}}\right|, \left|\varepsilon_{\frac{1}{4}}\right|\right\}\simeq& \frac{\sqrt{\alpha^2(1 - 3 \cfl)^2 + (1 -2\alpha- 3 \cfl+6\alpha\cfl + 2 \cfl^2)^2 \omt^2}}{12} \xi^2. 
\end{aligned}
\label{eq:error_rk2}
\end{equation}
It can be seen that $\varepsilon_{\star,1} = \mathcal{O}(\xi^2)$ and the method is second-order accurate. Here we only keep the leading second-order term 
in $\xi$ for clarity. But in the numerical examples, we also include the $\xi^3$ term when computing the predicted numerical error as a reference. 

To better visualize the analytical results, in Figure \ref{pic:error rk2},  we plot  the  leading coefficients in \eqref{eq:error_rk2} versus the CFL number with respect to different $\alpha$ and  $\omt$. 

\begin{figure}[h!]
            \centering
		\begin{subfigure}[t]{.31\textwidth}
			\centering
			\includegraphics[trim=0cm 1cm 0cm 1cm, width=1. \linewidth]{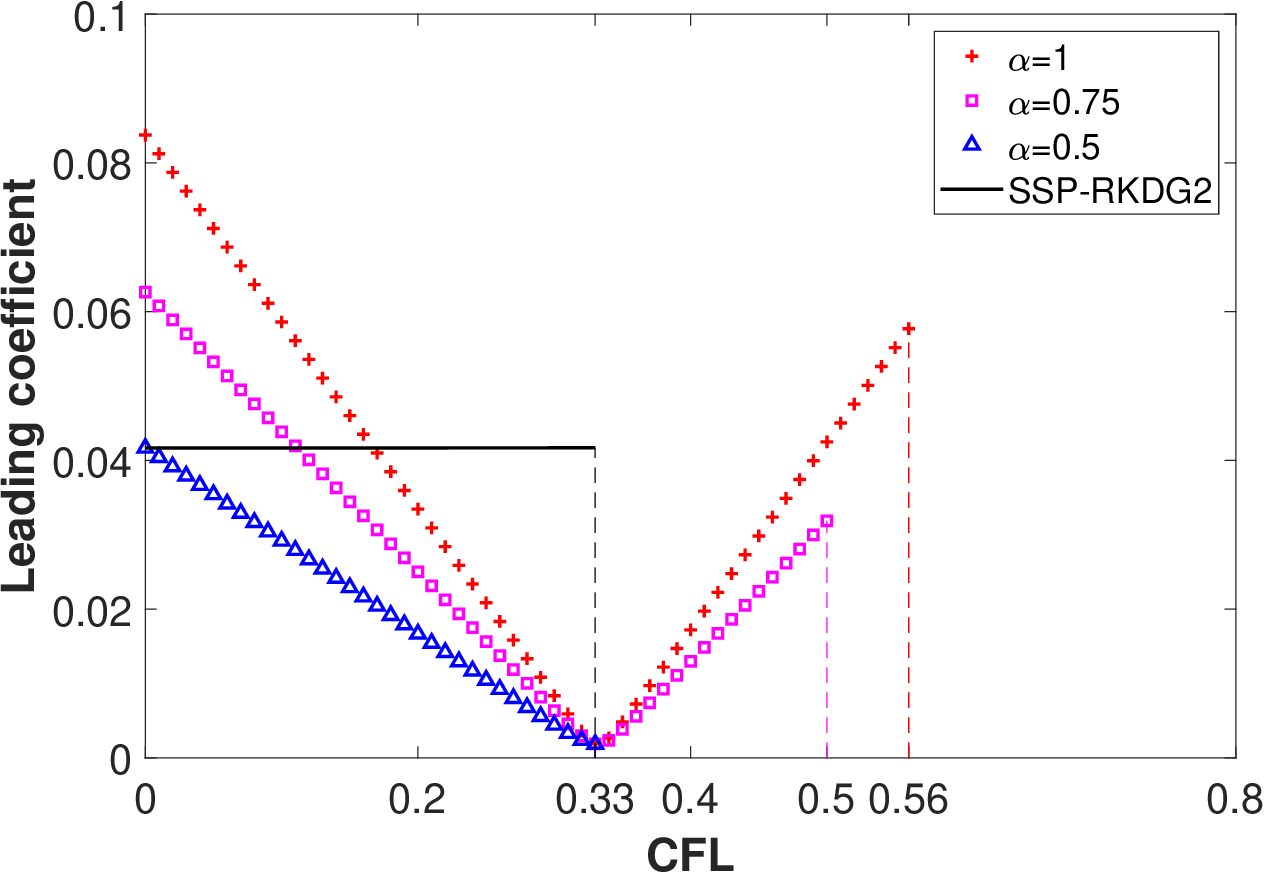}
			\caption{$\omt=0.1$ }
		\end{subfigure}%
		\hspace{0mm}
		\begin{subfigure}[t]{.31 \textwidth}
			\centering
			\includegraphics[trim=0cm 1cm 0cm 1cm, width=1. \linewidth]{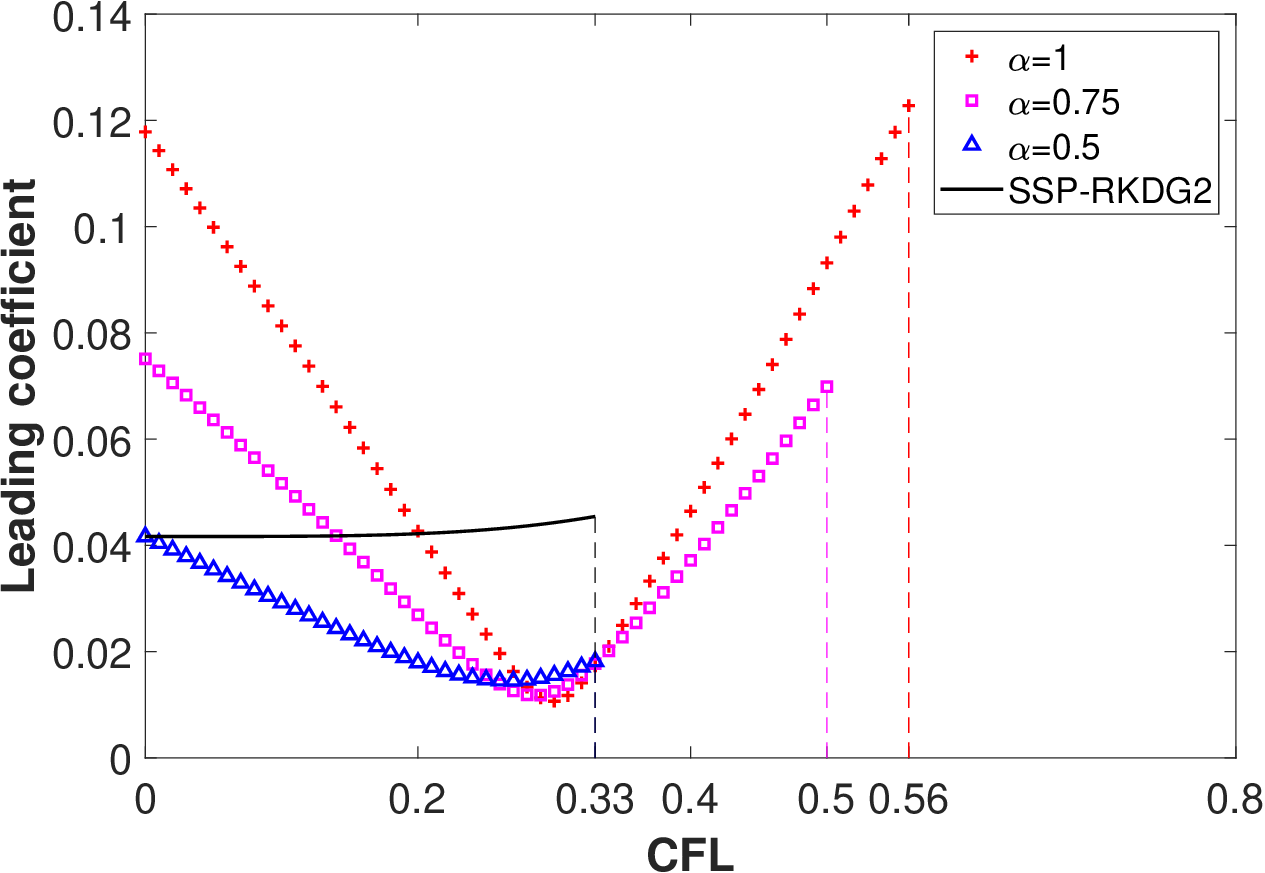}
			\caption{$\omt=1$}
			\label{fig:sfig2}
		\end{subfigure} 
  \hspace{0mm}
		\begin{subfigure}[t]{.31\textwidth}
			\centering
			\includegraphics[trim=0cm 1cm 0cm 1cm, width=1. \linewidth]{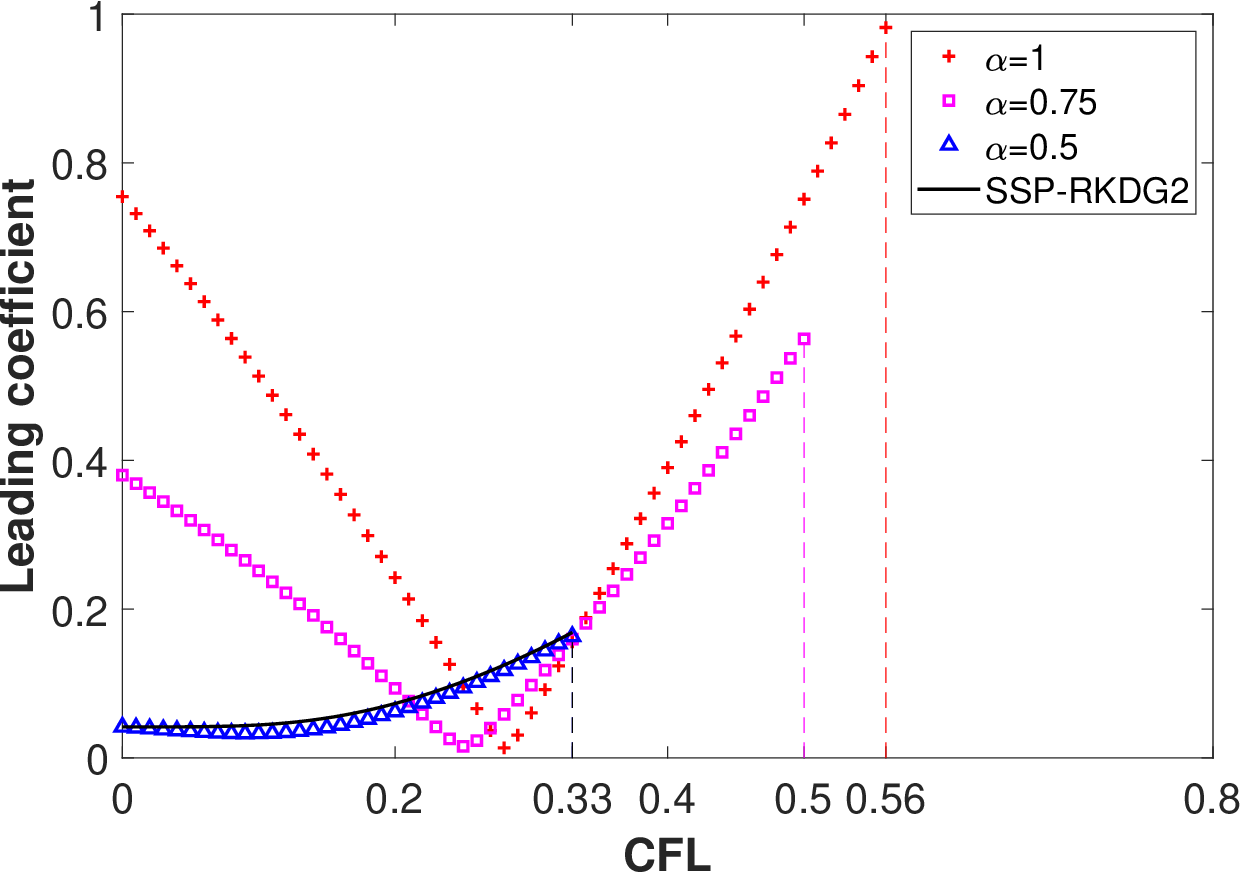}
			\caption{$\omt=9$}
		\end{subfigure} 
		\caption{The leading coefficient of the predicted error \eqref{eq:error_rk2} at different final times.}
		\label{pic:error rk2}\vspace{-0.75cm}
	\end{figure}

For $k=2$, the analysis is similar. We take the third-order scheme \eqref{eq:SSPRK3} as an example. Under the assumption that $\lambda<0.275$ and $\xi$ is sufficiently small, we have

\begin{subequations}\label{genericRK3_error}
	\begin{align}
	\left| \varepsilon_{-\frac{1}{3}}\right| =&\sqrt{\frac{\cfl^6 \omt^2}{576}+\frac{\left(5400 \cfl^4-1200 \cfl^3+44 \cfl^2-39 \cfl+4\right)^2}{41990400 \left(-12 \cfl^3+2 \cfl^2-3 \cfl+1\right)^2}}\xi^3+O(\xi^4)
   ,\\
 \left| \varepsilon_0 \right|=& \sqrt{\frac{\cfl^6 \omt^2}{576}+\frac{\left(600 \cfl^4-280 \cfl^3+56 \cfl^2-33 \cfl+6\right)^2}{1440000 \left(-12 \cfl^3+2 \cfl^2-3 \cfl+1\right)^2}}\xi^3+O(\xi^4),\\
 \left| \varepsilon_{\frac{1}{3}}\right|=&\sqrt{\frac{\cfl^6 \omt^2}{576}+\frac{\left(16200\lambda^4-960\lambda^3-788\lambda^2+399\lambda-88\right)^2}{1049760000 \left(-12 \cfl^3+2 \cfl^2-3 \cfl+1\right)^2}}\xi^3+O(\xi^4).
	\end{align}
\end{subequations}
We use $\varepsilon_{\star,2} = \max_{0\leq j\leq N-1}\left\{|\varepsilon_{-{1}/{3}}|, |\varepsilon_{0}|, |\varepsilon_{{1}/{3}}|\right\}$ for the predicted numerical error. It can be seen that $\varepsilon_{\star,2} = \mathcal{O}(\xi^3)$ and the method is third-order accurate. 

\begin{remark}
   For comparison, we list the error of the standard RKDG scheme by von Neumann analysis \cite{ZHONG20112814}, with 
   \begin{alignat*}{2}
\varepsilon_{\star, 1,\mathrm{RKDG}} \simeq \frac{\sqrt{1+16 \cfl^4 \bar{t}^2}}{24}  \xi^2 \  \text{ for } k = 1 \ \text{and} \
\varepsilon_{\star, 2,\mathrm{RKDG}} \simeq \frac{\sqrt{1+100 \cfl^6 \bar{t}^2}}{240}  \xi^3\ \text{ for } k = 2. 
\end{alignat*}
\end{remark}
\begin{remark}\label{rmk:expsin}
    Although in $\varepsilon_{\star,1}$, we assumed the initial condition to be $u(x,0) =\exp(\mathrm{i}\omega x)$, \eqref{eq:globerr} reveals that $\varepsilon_{j\pm1/4}$ essentially covers all possible arguments of a complex number due to the presence of $\exp{(\mathrm{i}\omega x_j)}$. Hence $\varepsilon_{\star,1}$ will also be a precise prediction for other initial inputs with a different argument, such as $\sin(x)$ and $\cos(x)$, as has been used in \cref{ex:adv}. Similar comments apply to $\varepsilon_{\star,2}$.
\end{remark}
\begin{remark}
    The order degeneracy of the sdRKDG schemes in Class B is observed in the presence of sonic points; however, for the linear advection equation with a constant wave speed, as shown above, the methods retain optimal convergence rates.
\end{remark}

\section{Numerical results}
\label{sec:Numerical results}
In this section, we present numerical solutions of our sdRKDG schemes and compare them with those of the standard RKDG schemes. We mainly present results for schemes in Class B --- \eqref{eq:SSPRK2},  \eqref{eq:SSPRK3}  and \eqref{eq:rk4} --- as they are more efficient thanks to the improved CFL conditions. 

\subsection{Accuracy tests} 
In this section, we test the accuracy of sdRKDG schemes in different settings and validate the improved CFL conditions by setting CFL numbers as those obtained from the von Neumann analysis. For the sdRKDG methods  \eqref{eq:SSPRK2}, \eqref{eq:SSPRK3}, and \eqref{eq:rk4}, the CFL numbers are set as $\lambda = 0.565$, $0.275$, and $0.213$, respectively. For the standard RKDG methods, the CFL numbers are set as $\lambda = 0.333$, $0.209$, and $0.145$, respectively.
\subsubsection{One-dimensional tests}
\begin{exmp}[linear advection equation]\label{ex:adv}
In this example, we solve the linear equation $\partial_tu+\partial_xu=0$ on domain $(- \pi,  \pi)$ up to time $t=1$. The initial condition is set as $u(x,0) = \sin(x)$ and the exact solution is given by $u(x, t)=\sin (x-t)$.  To compare the numerical errors with the predicted results from von Neumann analysis (see \cref{rmk:expsin}), we use the upwind flux and periodic boundary condition on  uniform meshes.  We test with several different CFL numbers and calculate the $\varepsilon_{\star,k}$ error as defined in \Cref{sec:full discrete}. Numerical errors and convergence rates are listed in \cref{table:predicted-p1} for the second-order scheme \eqref{eq:SSPRK2} and in \cref{table:predicted-p2} for the third-order scheme \eqref{eq:SSPRK3}.  It can be seen that the numerical errors agree well with the predicted errors for both schemes. In the second-order case, when both schemes adopt their corresponding maximum CFL numbers, the sdRKDG scheme, as it uses a larger time step size, admits a larger numerical error compared to the standard RKDG2 scheme; when both schemes adopt the same CFL number $0.333$, the sdRKDG scheme \eqref{eq:SSPRK2} admits a smaller error than that of the standard RKDG2 scheme; when both schemes adopt the CFL number $0.565$, the sdRKDG scheme remains stable, but the standard RKDG2 scheme blows up. These observations are consistent with \cref{fig:sfig2}. Similar results are observed for the third-order schemes. 
 
 \begin{table}[h!]
 	\centering
 	\resizebox{0.95\textwidth}{!}
  {
 		\begin{tabular}{ c|c | cc | cc|cc | cc }

 		\hline  
 		\multirow{2}{*}{CFL} &\multirow{2}{*}{$N$} &\multicolumn{4}{c|}{RKDG2}&\multicolumn{4}{c}{sdRKDG \eqref{eq:SSPRK2}}\\
 		\cline{3-10}
 		& &$\varepsilon_{1,\star}$ error &Order&Predicted   &Order&$\varepsilon_{1,\star}$ error &Order&Predicted   &Order\\
 		
 \hline
 		\multirow{6}{*}{\rotatebox[origin=c]{90}{{\centering $\lambda=0.001$}}}
 		&20   &  4.46E-03&   -  & 4.54E-03&   -   & 1.07E-02&   -  &  1.08E-02&   - \\ 
 		&40   &  1.08E-03& 2.05 & 1.08E-03& 2.07  & 2.79E-03& 1.94 &  2.80E-03& 1.95\\
 		&80   &  2.63E-04& 2.03 & 2.64E-04& 2.03  & 7.12E-04& 1.97 &  7.12E-04& 1.97\\
 		&160  &  6.51E-05& 2.02 & 6.51E-05& 2.02  & 1.80E-04& 1.99 &  1.80E-04& 1.99\\
 		&320  &  1.62E-05& 2.01 & 1.62E-05& 2.01  & 4.51E-05& 1.99 &  4.51E-05& 1.99\\
 		&640  &  4.03E-06& 2.00 & 4.03E-06& 2.00  & 1.13E-05& 2.00 &  1.13E-05& 2.00\\
 
 		\hline  
 		\multirow{6}{*}{\rotatebox[origin=c]{90}{{\centering $\lambda=0.333$}}}
 		
 		&20   & 6.61E-03&   -  & 4.87E-03&   -  &2.28E-03&   - &2.20E-03&   - \\ 
 		&40   & 1.65E-03& 2.00 & 1.17E-03& 2.06 &5.17E-04& 2.14&5.02E-04& 1.77\\
 		&80   & 3.74E-04& 2.14 & 2.86E-04& 2.03 &1.20E-04& 2.11&1.20E-04& 1.90\\
 		&160  & 8.12E-05& 2.20 & 7.10E-05& 2.01 &2.90E-05& 2.05&2.91E-05& 1.95\\
 		&320  & 1.82E-05& 2.16 & 1.77E-05& 2.01 &7.15E-06& 2.02&7.19E-06& 1.98\\
 		&640  & 4.40E-06& 2.05 & 4.40E-06& 2.00 &1.78E-06& 2.01&1.79E-06& 1.99\\
 
 		\hline
 		\multirow{6}{*}{\rotatebox[origin=c]{90}{{\centering $\lambda=0.565$}}} 		
 		&20  &1.64E+00&   - &---&- & 1.09E-02&   -   & 1.23E-02&   -  \\ 
 		&40  &6.07E+02&-8.35&---&- & 3.01E-03& 1.86  & 3.08E-03& 2.00 \\
 		&80  &9.38E+07&-&---&- &     7.57E-04& 1.99  & 7.72E-04& 2.00 \\
 		&160 &2.93E+19&-&---&- &     1.93E-04& 1.97  & 1.93E-04& 2.00 \\
 		&320 &3.38E+42&-&---&- &     4.82E-05& 2.00  & 4.83E-05& 2.00 \\
 		&640 &1.79E+89&-&---&- &     1.21E-05& 2.00  & 1.21E-05& 2.00 \\
 		\hline	
 	\end{tabular}}
 	\caption{ Numerical errors and predicted results using \eqref{eq:error_rk2} of RKDG and sdRKDG methods for the linear advection equation using $\mathcal{P}^1$ elements in \cref{ex:adv}.}
 	\label{table:predicted-p1}\vspace{-0.2cm}
 \end{table}

 \begin{table}[h!]
 	\centering
 	\resizebox{0.95\textwidth}{!}
  {
 		\begin{tabular}{c|c|cc|cc|cc|cc}

 		\hline  
 		\multirow{2}{*}{CFL} &\multirow{2}{*}{$N$} &\multicolumn{4}{c|}{RKDG3}&\multicolumn{4}{c}{sdRKDG \eqref{eq:SSPRK3}}\\
 		\cline{3-10}
 		& &$\varepsilon_{2,\star}$ error &Order&Predicted   &Order&$\varepsilon_{2,\star}$ error &Order&Predicted   &Order\\
 		\hline
 		\multirow{6}{*}{\rotatebox[origin=c]{90}{{\centering $\lambda=0.001$}}}
 				&20   &  1.27E-04&   -  & 1.29E-04&   - & 1.54E-04   &   -  &   1.55E-04   &   - \\ 
 		&40   &  1.61E-05& 2.98 & 1.61E-05& 3.00& 1.94E-05   & 2.99 &   1.93E-05   & 2.99\\
 		&80   &  2.02E-06& 3.00 & 2.02E-06& 3.00& 2.43E-06   & 3.00 &   2.42E-06   & 3.00\\
 		&160  &  2.52E-07& 3.00 & 2.52E-07& 3.00& 3.03E-07   & 3.00 &   3.02E-07   & 3.00\\
 		&320  &  3.15E-08& 3.00 & 3.15E-08& 3.00& 3.78E-08   & 3.00 &   3.78E-08   & 3.00\\
 		&640  &  3.94E-09& 3.00 & 3.94E-09& 3.00& 4.72E-09   & 3.00 &   4.72E-09   & 3.00\\
 		\hline  
 		\multirow{6}{*}{\rotatebox[origin=c]{90}{{\centering $\lambda=0.209$}}}
 		
 		&20   & 1.28E-04   &   -  & 1.30E-04    &   -  &6.27E-05    &   - &4.75E-05   &   - \\ 
 		&40   & 1.62E-05   & 2.99 & 1.62E-05    & 3.00 &6.22E-06    & 3.33&5.94E-06   & 3.00\\
 		&80   & 2.03E-06   & 3.00 & 2.03E-06    & 3.00 &7.30E-07    & 3.09&7.42E-07   & 3.00\\
 		&160  & 2.53E-07   & 3.00 & 2.53E-07    & 3.00 &9.11E-08    & 3.00&9.29E-08   & 3.00\\
 		&320  & 3.17E-08   & 3.00 & 3.17E-08    & 3.00 &1.15E-08    & 2.99&1.16E-08   & 3.00\\
 		&640  & 3.96E-09   & 3.00 & 3.96E-09    & 3.00 &1.44E-09    & 2.99&1.45E-09   & 3.00\\
 		\hline
 		\multirow{6}{*}{\rotatebox[origin=c]{90}{{\centering $\lambda=0.275$}}} 	
 		&20  &1.28E-04& -&---&- &    2.88E-04&   -  &  4.35E-04  &   -  \\ 
 		&40  &1.62E-05& 2.99&---&- & 4.24E-05& 2.76 &  5.44E-05  & 3.00 \\
 		&80  &2.22E-06& 2.86&---&- & 6.17E-06& 2.78 &  6.80E-06  & 3.00 \\
 		&160 &2.52E+03&- &---&- &    8.21E-07& 2.91 &  8.50E-07  & 3.00 \\
 		&320 &3.46E+22&-&---&- &     1.05E-07& 2.97 &  1.06E-07  & 3.00 \\
 		&640 &3.58E+60&-&---&- &     1.32E-08& 2.99 &  1.33E-08  & 3.00 \\
 		\hline	
 	\end{tabular}}
 	\caption{ Numerical errors and predicted results using  \eqref{genericRK3_error} of RKDG and sdRKDG methods for the linear advection equation using $\mathcal{P}^2$ elements in \cref{ex:adv}. }
 	\label{table:predicted-p2}
 \end{table}

\end{exmp}

\begin{exmp}[Burgers equation]\label{ex:1dburgers}
We solve the nonlinear Burgers equation in one dimension,
	$\partial_{t} u+\partial_{x}\left({u^{2}}/{2}\right)=0$,  on $x \in(-\pi,  \pi)$ with periodic boundary conditions. We test two different initial conditions: $u(x, 0)=\sin (x)+2$ and $u(x, 0)=\sin (x)+0.5$. The latter is  known to have sonic points while the former does not.
 We use the Godunov flux and compute to $t=0.2$ on both uniform and nonuniform meshes for spatial polynomial degrees $k = 1,2,3$. For the nonuniform meshes, we generate the meshes by randomly perturbing the nodes up to $15\%h$.  For the initial condition $u(x, 0)=\sin (x)+2$, where there are no sonic points, both schemes achieve the optimal order of accuracy on the same meshes, as given in \cref{table_burgers1d}. 
 
 For the initial condition $u(x, 0)=\sin (x)+0.5$, sonic points may cause the sdRKDG schemes in Class B to be suboptimal as shown in \cref{table:burgers_sonic}. We observe that their $L^1$, $L^2$, and $L^\infty$ accuracy orders are dropped by $0.25$, $0.5$, and $1$, correspondingly. While the sdRKDG schemes in Class A remain to be optimal.  We have also tested with the linear advection equation with degenerate variable coefficients in \cite{li2019analysis}, and similar order reduction is observed for the sdRKDG methods of both classes. Detailed numerical results are omitted to save space. The investigation into the cause and remedy of such order degeneracy will be left to our future work.

\begin{table}[h!]
		\centering
  {
			\begin{tabular}{   c|c|c|c  c| c c|cc  }
				\hline
		   	&&\multirow{2}{*}{$N$} &\multicolumn{2}{c|}{$k=1$}&\multicolumn{2}{c|}{$k=2$}&\multicolumn{2}{c}{$k=3$} \\
			\cline{4-9}
			 	&&&$L^2$ error &Order&$L^2$ error &Order&$L^2$ error &Order \\
				\hline
\multirow{8}{*}{\rotatebox[origin=c]{90}{{\centering Uniform  }}}&
\multirow{4}{*}{\rotatebox[origin=c]{90}{{\centering RKDG  }}}
&40 & 2.54E-03&   - &4.09E-05&   - & 6.65E-07&   -  \\
&&80& 6.58E-04& 1.95&5.16E-06& 2.99& 4.19E-08& 3.99 \\
&&160&1.70E-04& 1.95&6.52E-07& 2.98& 2.67E-09& 3.97 \\
&&320&4.32E-05& 1.98&8.18E-08& 3.00& 1.67E-10& 4.00 \\
\cline{2-9}				
 
&\multirow{4}{*}{\rotatebox[origin=c]{90}{{\centering sdRKDG}}}
&40&  3.35E-03&   - &4.63E-05& 0.00&4.43E-06&   -  \\
&&80& 9.03E-04& 1.89&6.33E-06& 2.87&3.08E-07& 3.85 \\
&&160&2.34E-04& 1.95&8.43E-07& 2.91&2.03E-08& 3.92 \\
&&320&5.89E-05& 1.99&1.09E-07& 2.95&1.31E-09& 3.96 \\
\hline

\multirow{8}{*}{\rotatebox[origin=c]{90}{{\centering Nonuniform  }}}&
\multirow{4}{*}{\rotatebox[origin=c]{90}{{\centering RKDG }}}
&40&  2.99E-03&   - &5.05E-05&   - &9.37E-07&   -  \\
&&80& 7.30E-04& 2.04&5.90E-06& 3.10&5.23E-08& 4.16 \\
&&160&1.84E-04& 1.99&7.62E-07& 2.95&3.40E-09& 3.94 \\
&&320&4.60E-05& 2.00&9.37E-08& 3.02&2.10E-10& 4.02 \\
\cline{2-9}				
&\multirow{4}{*}{\rotatebox[origin=c]{90}{{\centering sdRKDG}}}
&40&  3.30E-03&   - & 6.70E-05&   - &7.40E-06&   - 		 \\
&&80& 8.21E-04& 2.01& 8.41E-06& 2.99&4.40E-07& 4.09  \\
&&160&2.01E-04& 2.03& 1.11E-06& 2.92&2.91E-08& 3.90  \\
&&320&5.01E-05& 2.00& 1.55E-07& 2.84&1.84E-09& 3.99   \\
\hline

		\end{tabular}
        }
		\caption{$L^2$ error of RKDG and sdRKDG methods for the one-dimensional Burgers equation without sonic points on uniform and nonuniform meshes in \cref{ex:1dburgers}.  } \label{table_burgers1d}\vspace{-0.2cm}
	\end{table}

 \begin{table}[h!]
 \begin{center}
		\begin{tabular}{c | c|c| c | c |c|c  }
		\hline 		 
		$N$&$L^1$ error &Order&$L^2$ error &Order&$L^\infty$ error &Order \\
		\hline
		Class A&\multicolumn{6}{c}{$k=1$, sdRKDG \eqref{eq:midpoint}, $\lambda=0.333$}\\
		\hline
		40 &4.83E-03&   - &2.45E-03&   - &3.44E-03&   - \\
		80 &1.25E-03& 1.95&6.43E-04& 1.93&9.03E-04& 1.93\\
		160&3.18E-04& 1.98&1.64E-04& 1.97&2.26E-04& 2.00\\
		320&8.04E-05& 1.98&4.17E-05& 1.98&5.70E-05& 1.98\\
		640&2.02E-05& 1.99&1.05E-05& 1.99&1.43E-05& 1.99\\		
		\hline			
		Class A&\multicolumn{6}{c}{$k=2$, sdRKDG \eqref{eq:heun}, $\lambda=0.191$}\\
		\hline
		40 &7.44E-05&   - &3.98E-05&   - &8.56E-05&   - \\
		80 &9.37E-06& 2.99&5.19E-06& 2.94&1.30E-05& 2.72\\
		160&1.18E-06& 2.99&6.67E-07& 2.96&1.78E-06& 2.87\\
		320&1.48E-07& 2.99&8.52E-08& 2.97&2.39E-07& 2.90\\
		640&1.85E-08& 3.00&1.08E-08& 2.98&3.11E-08& 2.94\\
		
		\hline
		Class B&\multicolumn{6}{c}{$k = 1$, sdRKDG \eqref{eq:SSPRK2}, $\lambda=0.565$}\\
		\hline
		40 &6.30E-03&   - &3.37E-03&   - &4.03E-03&   - \\
		80 &2.10E-03& 1.58&1.29E-03& 1.38&2.41E-03& 0.74\\
		160&6.60E-04& 1.67&4.82E-04& 1.42&1.34E-03& 0.85\\
		320&2.03E-04& 1.70&1.81E-04& 1.41&6.85E-04& 0.96\\
		640&6.05E-05& 1.75&6.68E-05& 1.44&3.57E-04& 0.94\\		
		\hline			
		Class B&\multicolumn{6}{c}{$k = 2$, sdRKDG \eqref{eq:SSPRK3}, $\lambda=0.275$}\\
		\hline	
		40 &1.32E-04&   - &7.78E-05&   - &1.74E-04&   - \\
		80 &1.89E-05& 2.81&1.22E-05& 2.67&2.17E-05& 3.00\\
		160&2.67E-06& 2.83&1.97E-06& 2.63&4.43E-06& 2.29\\
		320&3.70E-07& 2.85&3.29E-07& 2.58&1.06E-06& 2.07\\
		640&5.07E-08& 2.87&5.63E-08& 2.55&2.57E-07& 2.04\\

		\hline
	\end{tabular} 
  \end{center}
		\caption{Numerical errors for the one-dimensional Burgers equation with sonic points in \cref{ex:1dburgers}.} \label{table:burgers_sonic}\vspace{-0.2cm}
	\end{table}

\end{exmp}

 \begin{exmp}[Euler equations]\label{ex:1deuler}
In this test, we solve the nonlinear system of one-dimensional Euler equations $\partial_t{u}+\partial_x{f}({u})=0$ on $(0,1)$, 
where ${u}=(\rho, \rho w, E)^\intercal$,  ${f}({u})=\left(\rho w, \rho w^{2}+p, w(E+p)\right)^\intercal$, $E={p}/{(\gamma-1)}+\rho w^{2}/2$ with $\gamma=1.4$. We impose the initial condition
$\rho(x, 0) = 1 + 0.2\sin(2\pi x)$, $w(x, 0) = 1$, $p(x, 0) = 1$ with periodic boundary conditions, and the exact solution is $\rho(x, t) = 1 + 0.2\sin(2\pi(x- t))$, $w(x, t) = 1$, $p(x, t) = 1$. We use the local Lax--Friedrichs flux to compute to $t = 10$.  

As demonstrated in \cref{table:euler accuracy}, the sdRKDG schemes remain stable and achieve optimal orders of accuracy, with errors comparable to RKDG schemes, for this nonlinear system. We have also conducted tests on randomly perturbed meshes or with different final times, observing similar convergence rates (details omitted).

\begin{table}[h!]
		\centering
		\resizebox{1.\textwidth}{!}
  {
		\begin{tabular}{  c | cc | cc|cc | cc }
			\hline  \multirow{2}{*}{$N$} &\multicolumn{4}{c|}{ $k=1$}&\multicolumn{4}{c}{ $k=2$}\\
                   \cline{2-9}
			  &$L^2$ error &Order&$L^\infty$ error &Order&$L^2$ error &Order&$L^\infty$ error &Order\\
    \hline              
              &\multicolumn{4}{c|}{sdRKDG \eqref{eq:SSPRK2}, $\lambda=0.565$}&\multicolumn{4}{c}{sdRKDG \eqref{eq:SSPRK3}, $\lambda=0.275$}\\
            
\hline  
20  &   5.16e-02 &         -    &  7.84e-02 &         -    & 4.88e-05 &         -    &  1.01e-04 &         -   \\
40  &   1.33e-02 &         1.96 &  1.98e-02 &         1.98 & 5.43e-06 &         3.17 &  1.47e-05 &         2.79 \\
80  &   3.32e-03 &         2.00 &  4.82e-03 &         2.04 & 6.51e-07 &         3.06 &  2.04e-06 &         2.85 \\
160 &   8.30e-04 &         2.00 &  1.20e-03 &         2.01 & 8.07e-08 &         3.01 &  2.66e-07 &         2.94 \\

\hline  
&\multicolumn{4}{c|}{sdRKDG \eqref{eq:SSPRK2},   $\lambda=0.333$}&\multicolumn{4}{c}{sdRKDG \eqref{eq:SSPRK3}, $\lambda=0.209$}\\\hline 
20  &   6.22e-03 &         -    &  1.19e-02 &         -    & 3.69e-05  &        -     &  1.29e-04 &        -     \\
40  &   1.17e-03 &         2.41 &  2.34e-03 &         2.34 & 4.73e-06  &         2.96 &  1.75e-05 &         2.89 \\
80  &   2.62e-04 &         2.16 &  4.97e-04 &         2.24 & 5.98e-07  &         2.98 &  2.24e-06 &         2.97 \\
160 &   6.34e-05 &         2.05 &  1.14e-04 &         2.13 & 7.52e-08  &         2.99 &  2.84e-07 &         2.98 \\

\hline
             &\multicolumn{4}{c|}{RKDG,  $\lambda=0.333$ }&\multicolumn{4}{c}{RKDG, $\lambda=0.209$}\\\hline 
20  &   3.23e-03 &         -    &  5.03e-03 &         -    & 3.81e-05 &        -     &  1.23e-04 &        -    \\
40  &   7.76e-04 &         2.06 &  1.17e-03 &         2.10 & 4.68e-06 &         3.03 &  1.62e-05 &         2.92 \\
80  &   1.92e-04 &         2.01 &  2.83e-04 &         2.06 & 5.84e-07 &         3.00 &  2.07e-06 &         2.97 \\
160 &   4.79e-05 &         2.00 &  7.24e-05 &         1.96 & 7.29e-08 &         3.00 &  2.61e-07 &         2.99 \\
\hline
		\end{tabular}
  }
		\caption{Numerical errors for one-dimensional Euler equations in \cref{ex:1deuler}.  } \label{table:euler accuracy}\vspace{-0.2cm}
	\end{table}  

\end{exmp}

\subsubsection{Two-dimensional tests}
\begin{exmp}[Euler equations in two dimensions] \label{ex:2deuler}
	We consider the nonlinear Euler equations in two dimensions : 
	$\partial_t{u}+\partial_x{f}({u})+\partial_y{g}({u})=0$, where $u = (\rho, \rho w, \rho v, E)^\intercal$, $f(u) = (\rho w, \rho w^2 + p, \rho w v, w(E+p))^\intercal$, $g(u) = (\rho v, \rho wv, \rho v^2 + p, v(E+p))^\intercal$, $E={p}/{(\gamma-1)}+\rho (w^2+v^{2})/2$ with $\gamma=1.4$
 with the periodic boundary conditions.  The initial condition is set as $\rho(x, y, 0)=1+0.2 \sin (2\pi(x+y))$, $w(x, y, 0)=0.7$, $v(x, y, 0)=0.3$, $p(x, y, 0)=1$, and the exact solution is $\rho(x, y, t)=1+0.2 \sin (2\pi(x+y-(w+v) t))$, $w=0.7$, $v=0.3$, $p=1$. We use the local Lax--Friedrichs flux and compute the solution up to $t=0.5$. The numerical results are listed in \cref{table:2deuler-periodic-rectangle}.  We can observe that both RKDG and sdRKDG schemes achieve their expected order of optimal accuracy with comparable numerical errors on rectangular meshes.

	\begin{table}[h!]
 \centering
		\resizebox{1.\textwidth}{!}
  {
  \begin{tabular}{  c | cc | cc|cc | cc }
			
               \hline
               \multirow{2}{*}{$N$} &\multicolumn{4}{c|}{ $k=1$}&\multicolumn{4}{c}{ $k=2$}\\
\cline{2-9}
              &$L^2$ error &Order&$L^\infty$ error &Order&$L^2$ error &Order&$L^\infty$ error &Order\\
            \hline
              &\multicolumn{4}{c|}{sdRKDG \eqref{eq:SSPRK2},  $\lambda=0.565$}&\multicolumn{4}{c}{sdRKDG \eqref{eq:SSPRK3}, $\lambda=0.275$}\\
            
   \hline  
20  &  2.83e-03 &          -   &  6.42e-03 &         -     & 1.28e-04 &          -   &  8.38e-04 &          -       \\
40  &  5.53e-04 &         2.36 &  1.82e-03 &         1.82  & 1.77e-05 &         2.85 &  1.26e-04 &         2.73  \\
80  &  1.26e-04 &         2.14 &  4.91e-04 &         1.89  & 2.32e-06 &         2.94 &  1.69e-05 &         2.90  \\
160 &  3.06e-05 &         2.04 &  1.25e-04 &         1.97  & 2.94e-07 &         2.98 &  2.16e-06 &         2.97 \\
 
\hline  
&\multicolumn{4}{c|}{sdRKDG \eqref{eq:SSPRK2},  $\lambda=0.333$}&\multicolumn{4}{c}{sdRKDG \eqref{eq:SSPRK3},  $\lambda=0.209$}\\\hline 
20  &   4.52e-03 &            - &  9.26e-03 &            -  &  1.29e-04 &            - &  8.42e-04 &          - \\
40  &   1.05e-03 &         2.11 &  2.45e-03 &         1.92  &  1.78e-05 &         2.85 &  1.26e-04 &         2.74 \\
80  &   2.57e-04 &         2.03 &  6.31e-04 &         1.96  &  2.33e-06 &         2.94 &  1.70e-05 &         2.90 \\
160 &   6.40e-05 &         2.01 &  1.61e-04 &         1.97  &  2.95e-07 &         2.98 &  2.16e-06 &         2.97 \\
   \hline
             &\multicolumn{4}{c|}{RKDG,  $\lambda=0.333$}&\multicolumn{4}{c}{RKDG,  $\lambda=0.209$}\\\hline 
20  &   2.43e-03 &          -   &  5.68e-03 &            -  & 1.11e-04 &           -  &  7.35e-04 &           -  \\
40  &   4.27e-04 &         2.51 &  1.74e-03 &         1.71  & 1.39e-05 &         3.00 &  1.05e-04 &         2.80  \\
80  &   9.07e-05 &         2.24 &  4.74e-04 &         1.87  & 1.73e-06 &         3.00 &  1.39e-05 &         2.93  \\
160 &   2.15e-05 &         2.08 &  1.23e-04 &         1.94  & 2.16e-07 &         3.00 &  1.76e-06 &         2.98 \\
   \hline
		\end{tabular}
  }
		\caption{Numerical errors for two-dimensional Euler equations on rectangular meshes in \cref{ex:2deuler}.}
		\label{table:2deuler-periodic-rectangle}\vspace{-0.2cm}
	\end{table}
 \end{exmp}
\subsection{Tests with discontinuous solutions} 
We now test the sdRKDG method for Euler equations with discontinuous solutions. We always take $\gamma=1.4$ and plot cell averages in this subsection. We apply the standard TVB minmod limiters in \cite{rkdg3,rkdg5} to identify troubled cells and to reconstruct polynomials. The TVB constant $M$ is to be specified for each problem. In our numerical tests, the CFL numbers are taken as $0.56$ and $0.27$ for $\mathcal{P}^1$  and $\mathcal{P}^2$ sdRKDG methods (\eqref{eq:SSPRK2} and \eqref{eq:SSPRK3}), and as $0.3$ and $0.18$ for $\mathcal{P}^1$ and $\mathcal{P}^2$ RKDG methods, respectively.

\subsubsection{One-dimensional tests}
 
\begin{exmp}[Sod problem]\label{ex:sod}
In this test, we solve  the one-dimensional Euler equations given in \cref{ex:1deuler} with a discontinuous initial condition  given by	
\begin{equation*}
		(\rho, w, p)= \begin{cases}(1,0,1), & x \leq 0.5, \\ (0.125,0,0.1), & x> 0.5.\end{cases}
	\end{equation*}
	We compute to $t = 0.2$ with $N=100$ elements. We use the local Lax--Friedrichs flux and the TVB limiter with $M=1$. As presented in \cref{fig:sod}, the sdRKDG method performs well near the discontinuity, and its solution stays close to both the RKDG solution and the exact solution. 
	
	\begin{figure}[h!]
            \centering
		\begin{subfigure}[t]{.3\textwidth}
			\centering
			\includegraphics[trim=0cm 1cm 0cm 1cm, width=1. \linewidth]{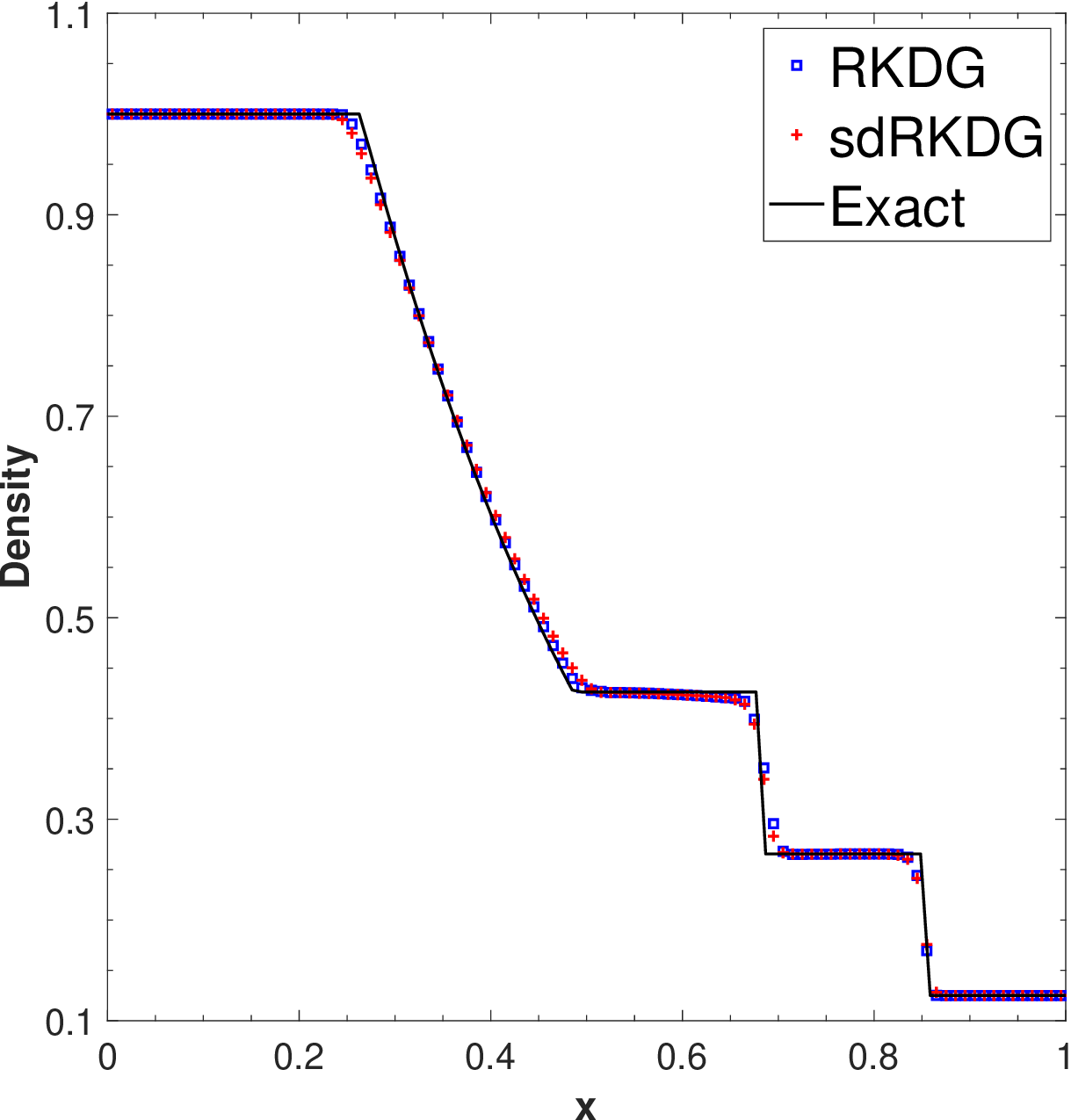}
			\caption{$k=1$ }
		\end{subfigure}%
		\hspace{10mm}
		\begin{subfigure}[t]{.3 \textwidth}
			\centering
			\includegraphics[trim=0cm 1cm 0cm 1cm, width=1. \linewidth]{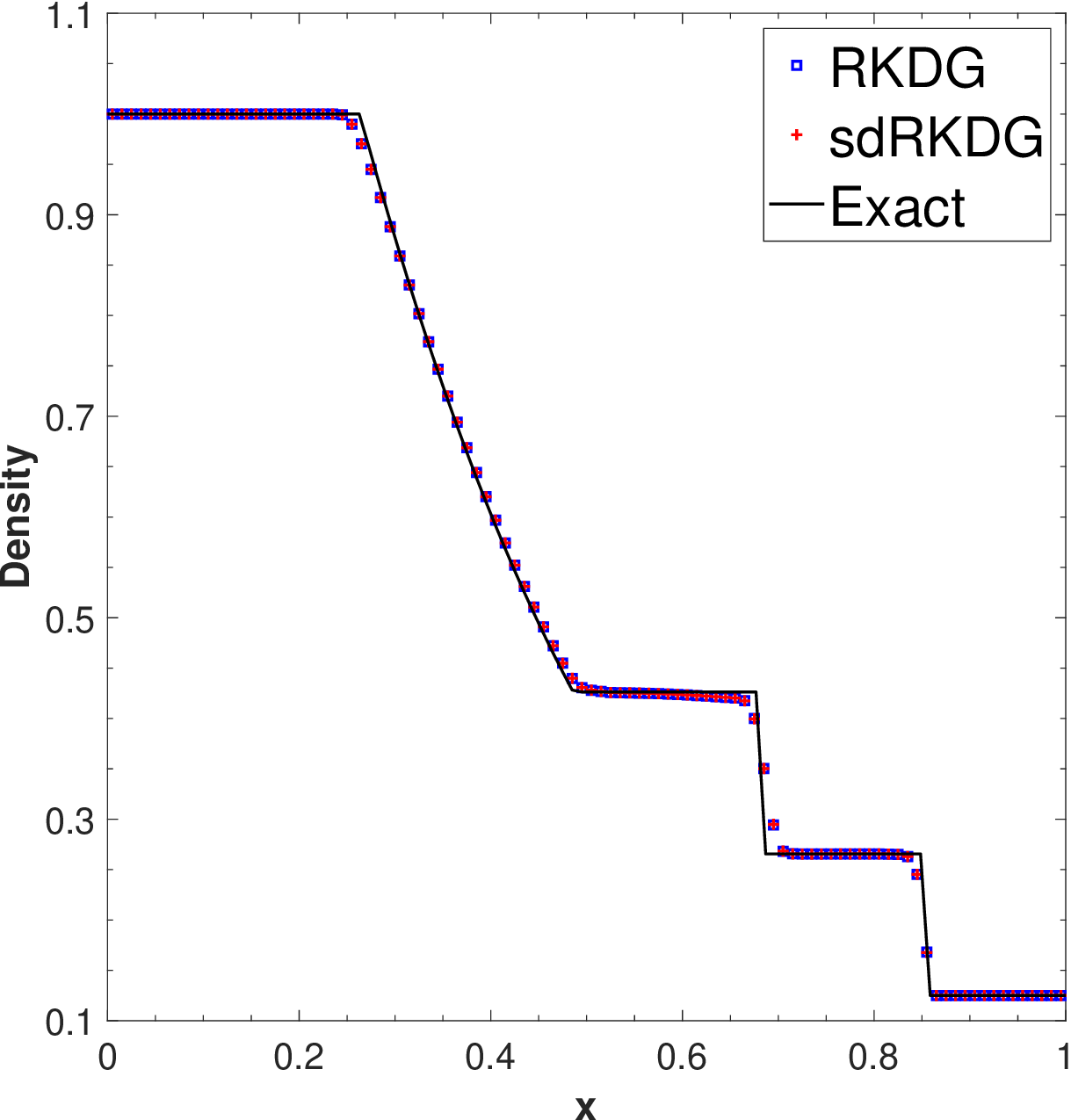}
			\caption{$k=2$}
		\end{subfigure} 
		\caption{Solution profiles for the Sod problem in \cref{ex:sod}. $N=100$ and $M=10$.}\vspace{-0.2cm}
		\label{fig:sod}
	\end{figure}

\end{exmp}
\begin{exmp}[Interacting blast waves]\label{ex:blast}
	We consider Euler equations with the following initial condition, 
	\begin{equation*}
		(\rho, w, p)= \begin{cases}(1,0,1000), & x \leq 0.1, \\ (1,0,0.01), & 0.1<x \leq 0.9, \\ (1,0,100), & x>0.9.\end{cases}
	\end{equation*}
Reflective boundary conditions are used both at $x=0$ and $x=1$. We use the local Lax--Friedrichs flux and the TVB
limiter with $M=200$, and compute up to $t = 0.038$ on meshes with $N = 200$ and $N = 400$ cells. The numerical density $\rho$ is compared with the reference solution in \cref{fig:blastwave}, computed using a fifth-order finite difference weighted essentially non-oscillatory (WENO) scheme \cite{shu2009high} on a more refined mesh, showing good agreement.
\begin{figure}[h!]
	\centering
\begin{subfigure}[t]{.24  \textwidth}
			\centering			
   \includegraphics[trim=1.5cm 1cm 1cm 2cm,width=1.0\textwidth]  {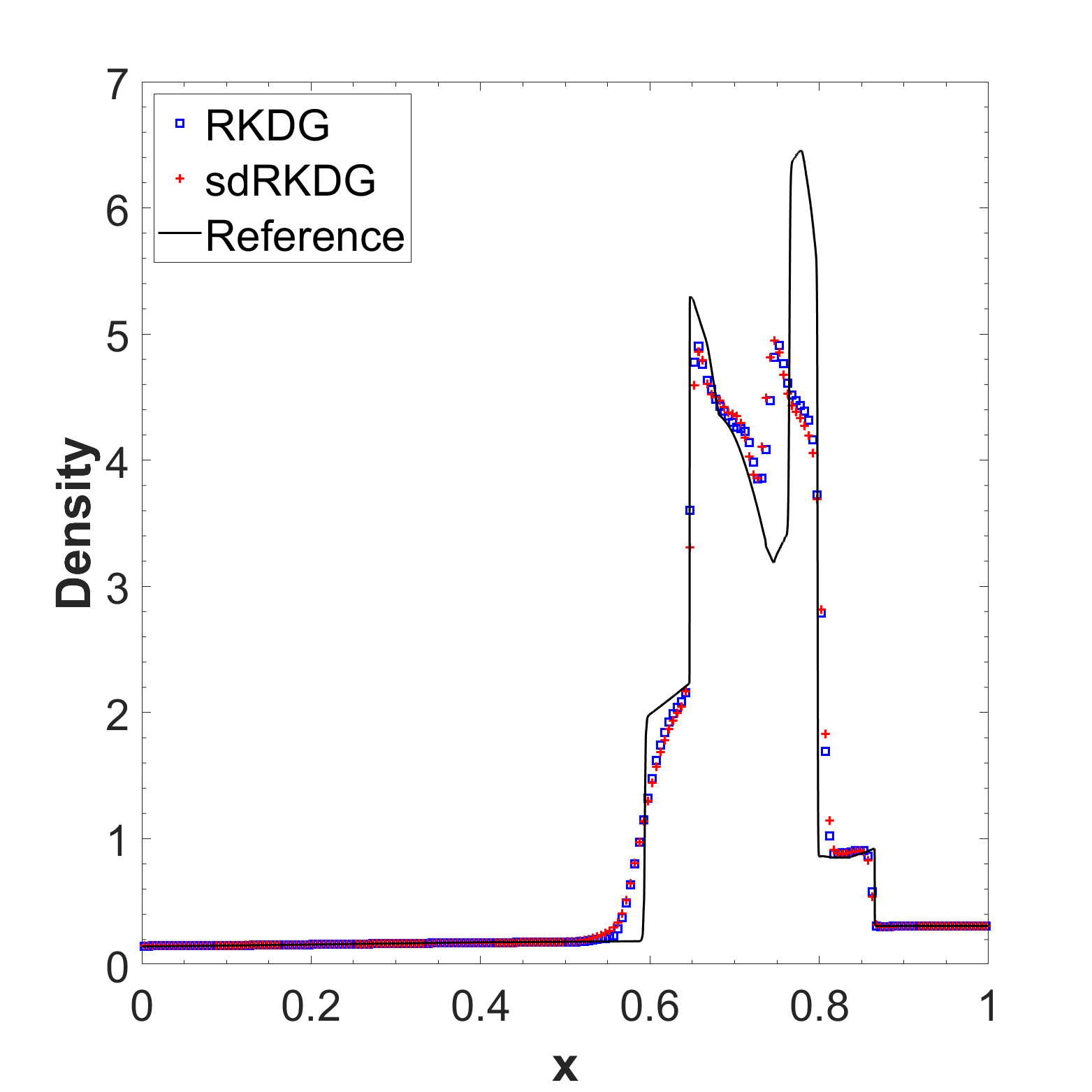}
			\caption{$k=1$, $N = 200$}		 
      \end{subfigure}		 
		\begin{subfigure}[t]{.24  \textwidth}
			\centering			\includegraphics[trim=1.5cm 1cm 1cm 2cm,width=1.0 \linewidth]            {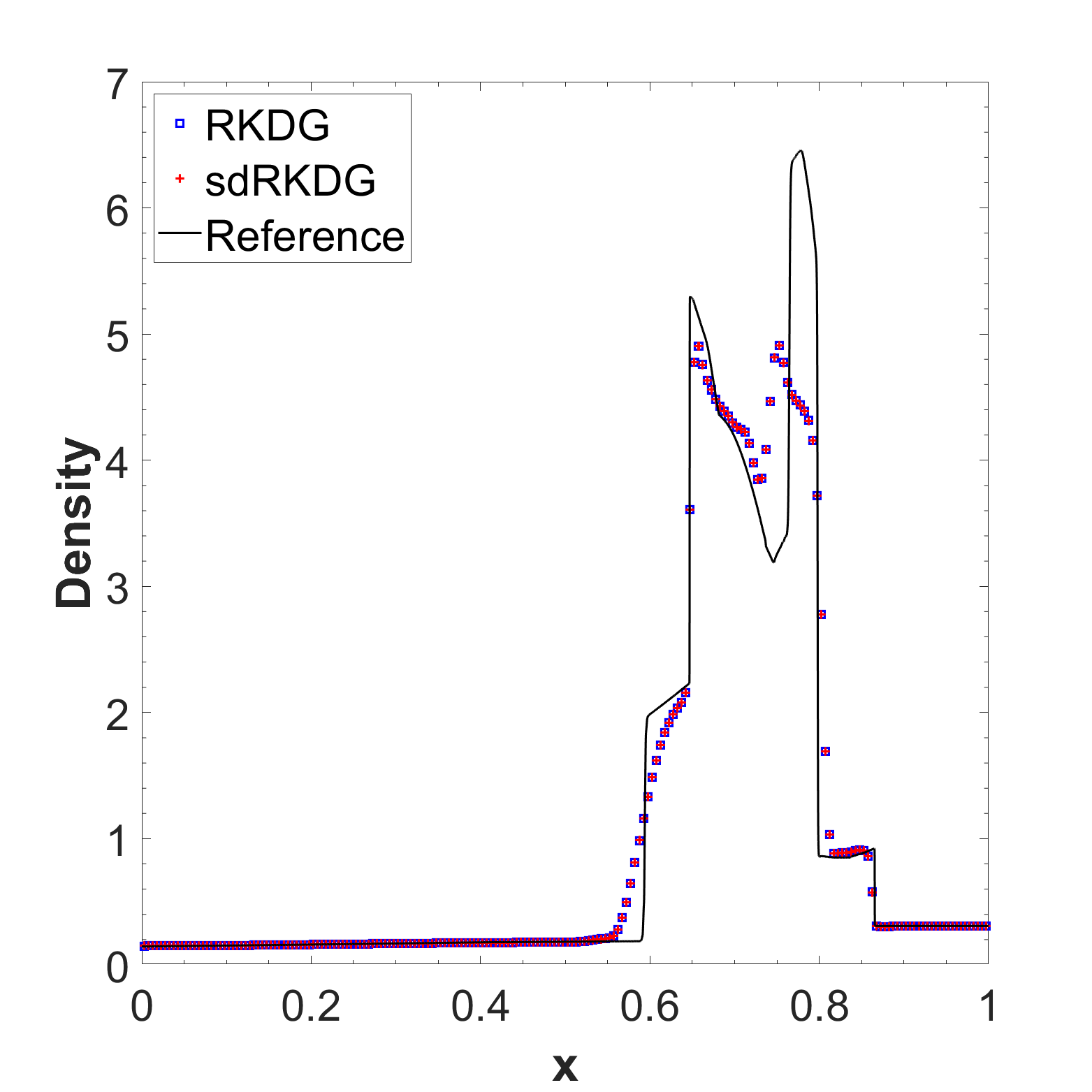}
			\caption{$k=2$, $N = 200$}
		\end{subfigure}
		\begin{subfigure}[t]{.24 \textwidth}
			\centering			\includegraphics[trim=1.5cm 1cm 1cm 2cm,width=1.0 \linewidth]            {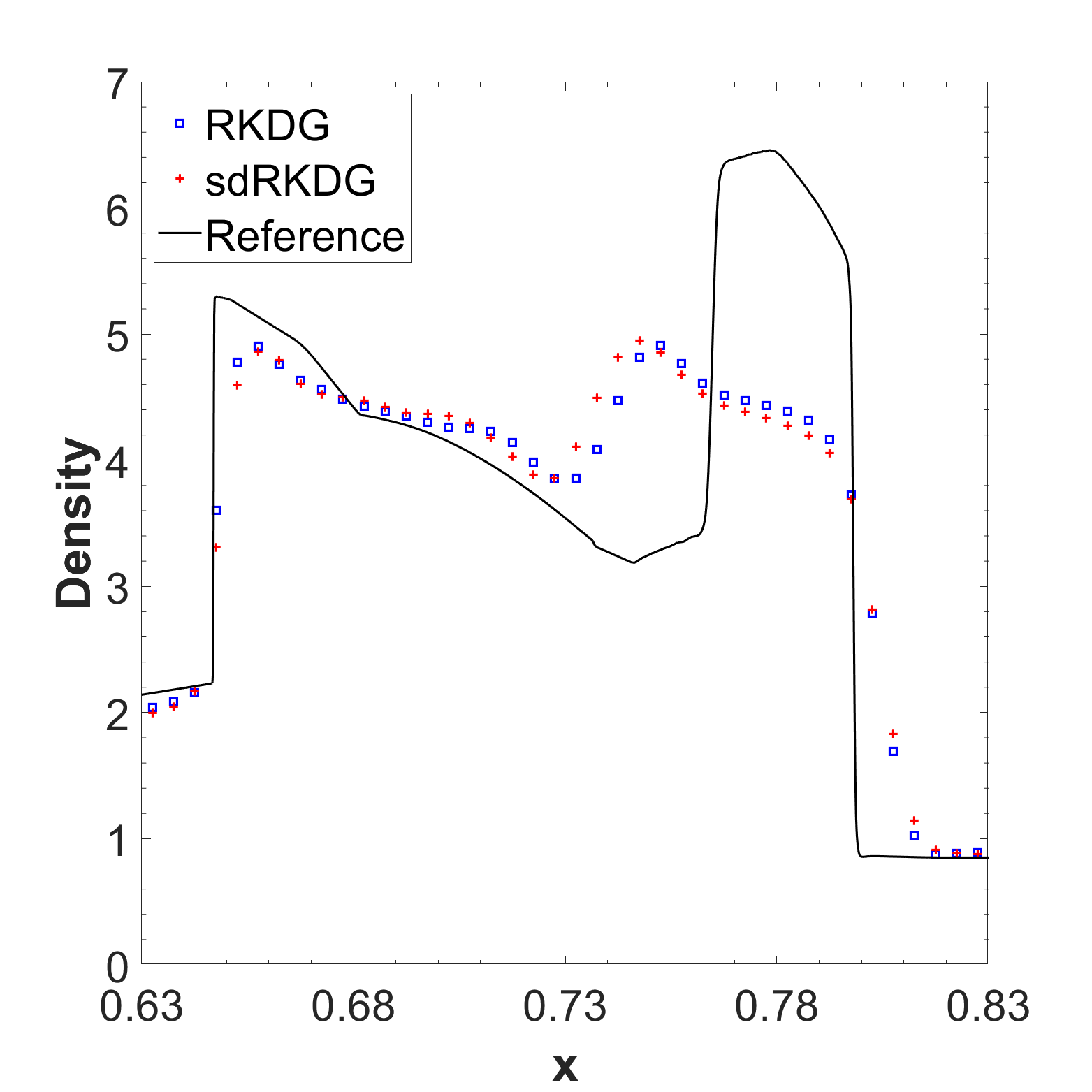}
			\caption{$k=1$ (Zoomed-in)}
		\end{subfigure}%
		\begin{subfigure}[t]{.24 \textwidth}
			\centering			\includegraphics[trim=1.5cm 1cm 1cm 2cm,width=1.0 \linewidth]   {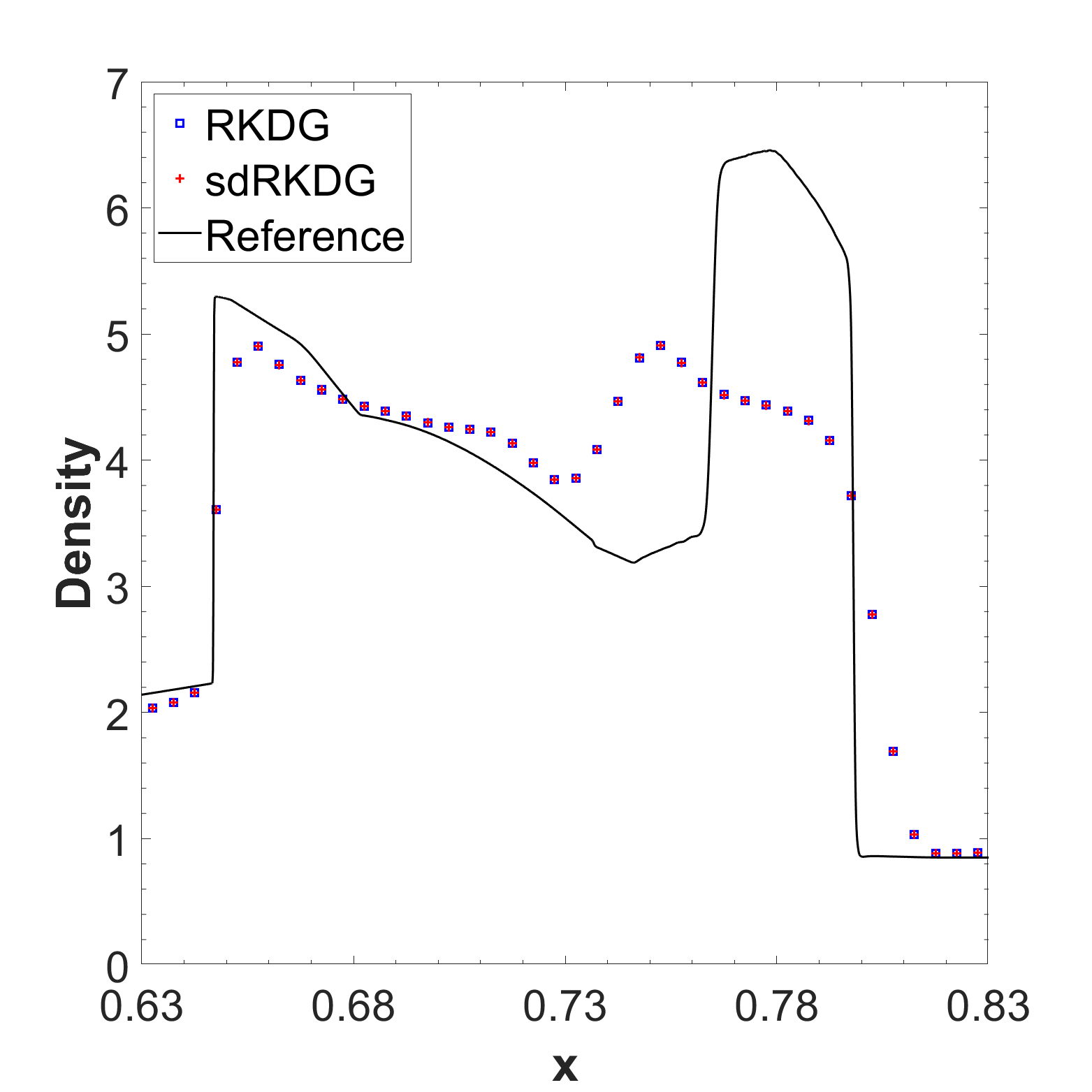}
			\caption{$k=2$ (Zoomed-in)}
		\end{subfigure}
		\label{fig:blastwave-200}
\\
\begin{subfigure}[t]{.24  \textwidth}
			\centering			
   \includegraphics[trim=1.5cm 1cm 1cm 0cm,width=1.0\textwidth]  {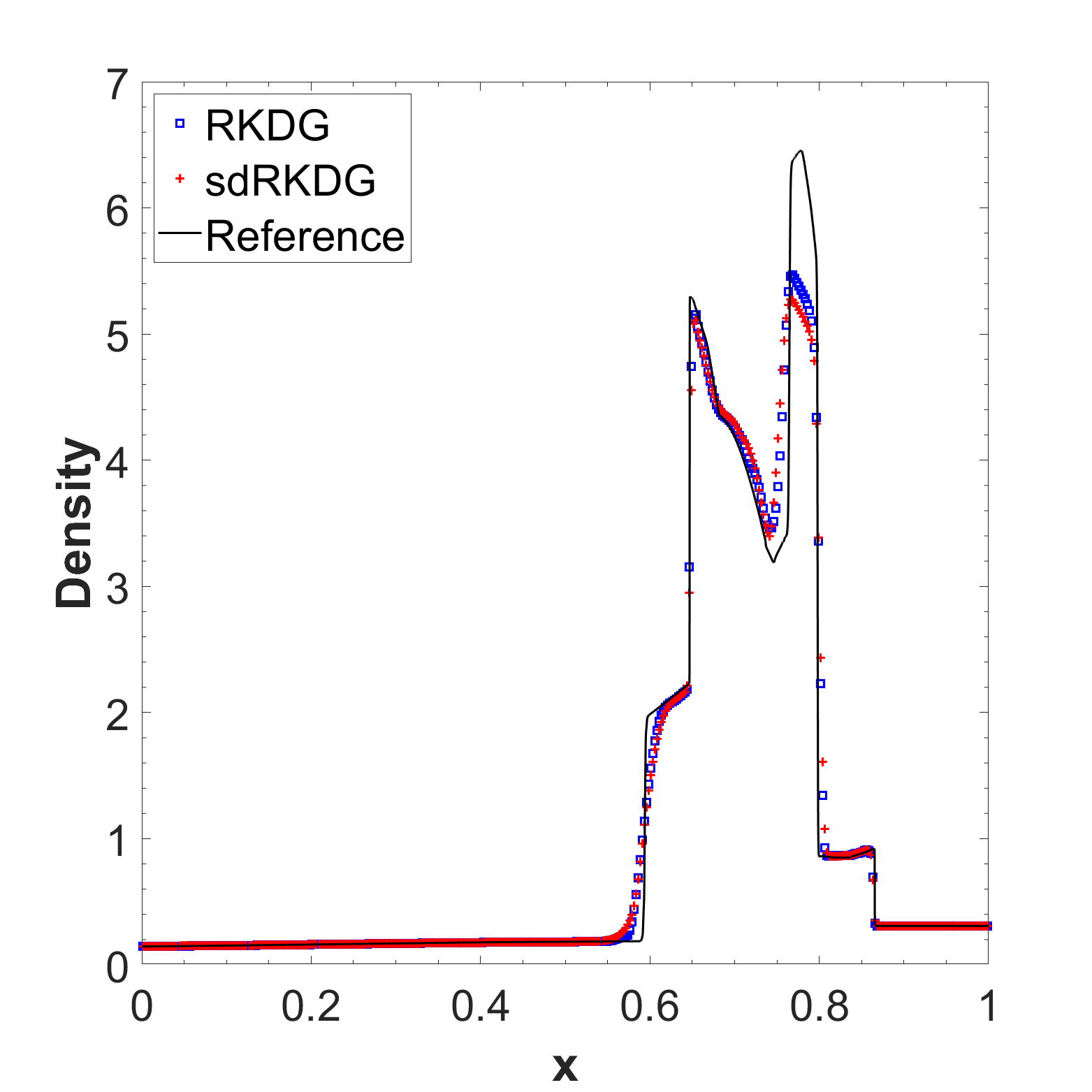}
			\caption{$k=1$, $N = 400$}		 
      \end{subfigure}		 
		\begin{subfigure}[t]{.24  \textwidth}
			\centering			\includegraphics[trim=1.5cm 1cm 1cm 0cm,width=1.0 \linewidth]            {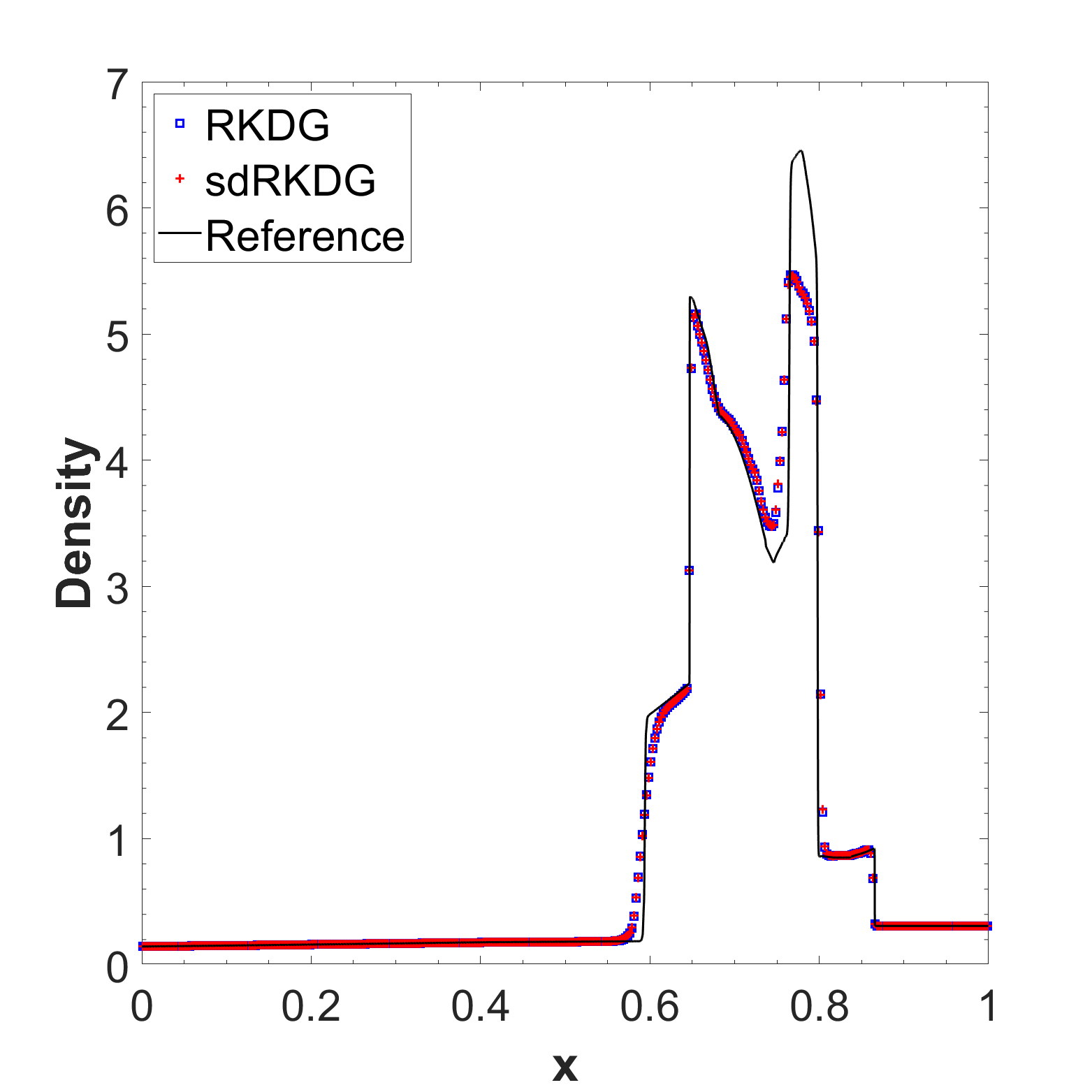}
			\caption{$k=2$, $N = 400$}
		\end{subfigure}
		\begin{subfigure}[t]{.24 \textwidth}
			\centering			\includegraphics[trim=1.5cm 1cm 1cm 0cm,width=1.0 \linewidth]            {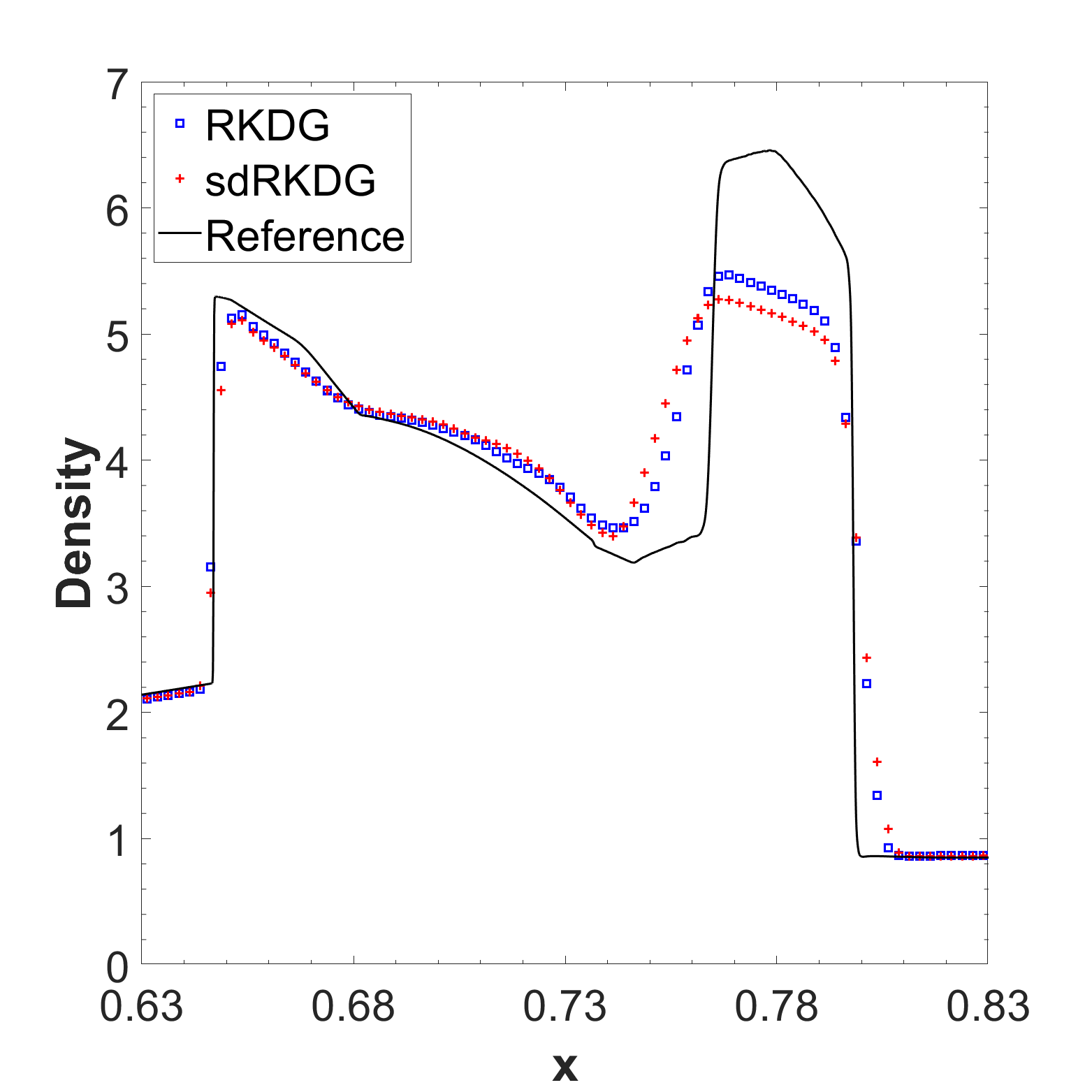}
			\caption{$k=1$ (Zoomed-in)}
		\end{subfigure}%
		\begin{subfigure}[t]{.24 \textwidth}
			\centering			\includegraphics[trim=1.5cm 1cm 1cm 0cm,width=1.0 \linewidth]   {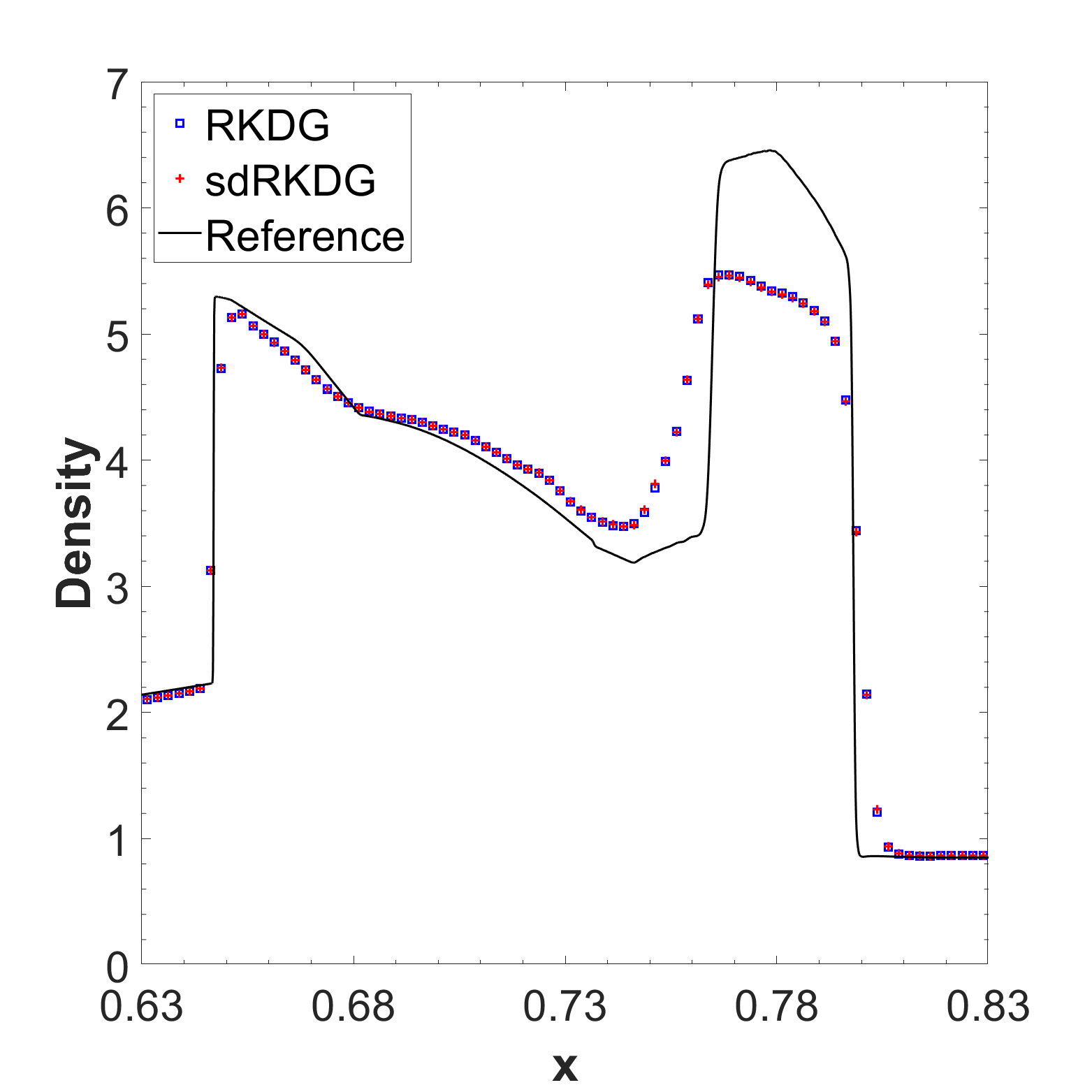}
			\caption{$k=2$ (Zoomed-in)}
		\end{subfigure}
		\caption{Solution profiles for the blast wave problem in \cref{ex:blast}. $M=200$. In the first row, $N = 200$; in the second row, $N = 400$.}
		\label{fig:blastwave}
	\end{figure}
\end{exmp}
\begin{exmp}[Shu--Osher problem]\label{ex:shuosher}
    We then consider the Shu--Osher problem, a benchmark test with complicated structures involving both strong and weak shock waves and highly oscillatory smooth waves  \cite{SHU198932}. The initial condition is set as
	\begin{equation*}
		(\rho, w, p)= \begin{cases}(3.857143,2.629369,10.333333), & x \leq -4, \\ (1+0.2 \sin (5 x), 0,1), & x>-4 .\end{cases}
	\end{equation*}
	The numerical density $\rho$ is plotted at $t=1.8$ against the reference solution computed by the fifth-order finite difference WENO scheme on a fine mesh. In \cref{fig:shu-osher}, we plot the densities by sdRKDG and RKDG methods with the local Lax--Friedrichs flux and the TVB
limiter with $M=300$. In \cref{fig:shuosher-zoom1} and \cref{fig:shuosher-zoom2}, we  present a zoomed-in view of the solution for $x\in[0.5,2.5]$. For the second-order schemes, it seems that the sdRKDG scheme is slightly more dissipative than the standard RKDG method, especially under a larger CFL condition. For the third-order schemes, the sdRKDG solutions are in good agreement with those of the standard RKDG solutions, regardless of the time step size.

	\begin{figure}[h!]
	\centering
		\begin{subfigure}[t]{.24\textwidth}
			\centering
			\includegraphics[trim=0cm 1cm 0cm 1cm,width=1. \linewidth,height=1.\linewidth]{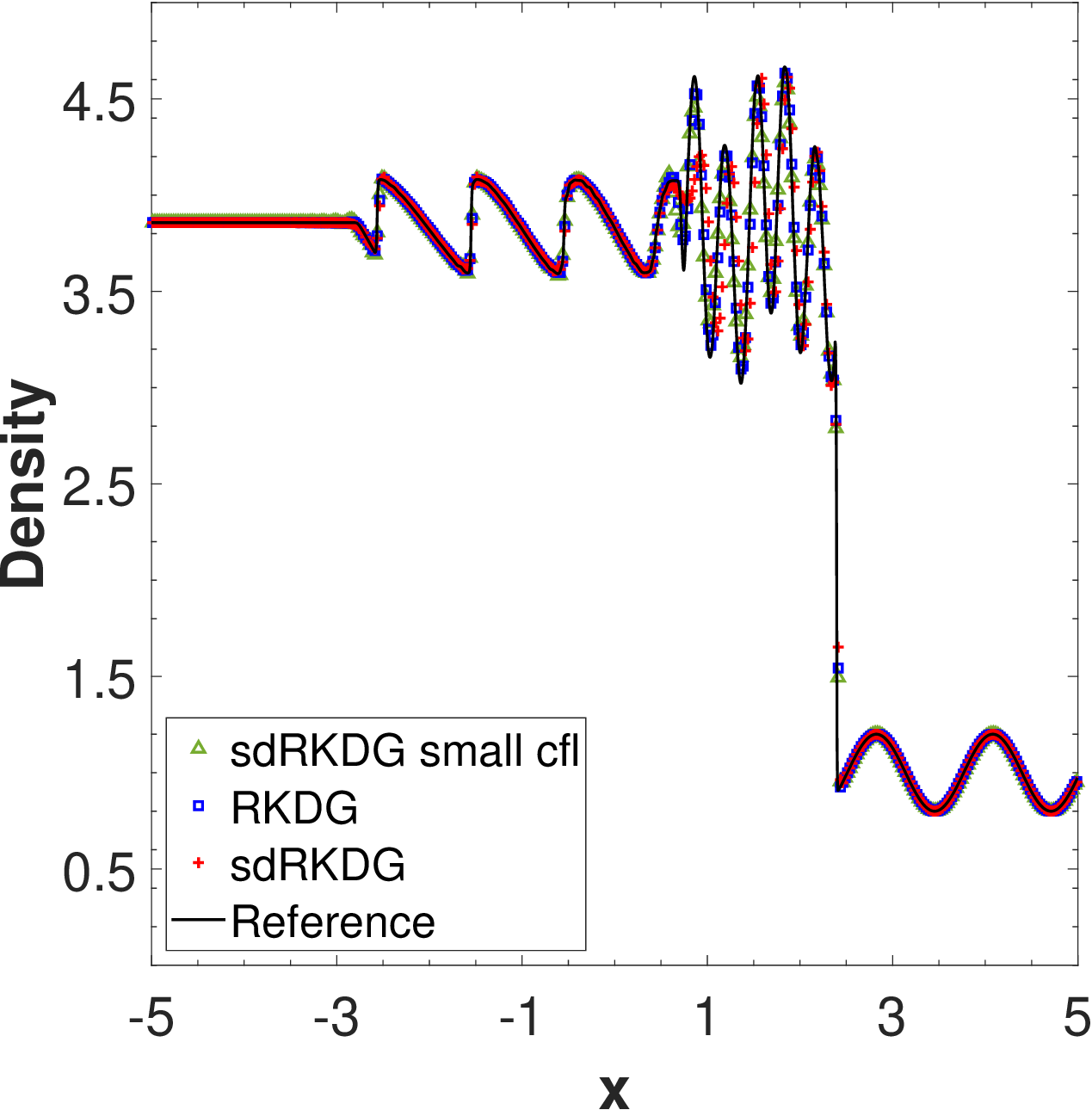}
			\caption{$k=1$}
		\end{subfigure}
		\begin{subfigure}[t]{.24\textwidth}
			\centering
			\includegraphics[trim=0cm 1cm 0cm 1cm,width=1. \linewidth,height=1.\linewidth]{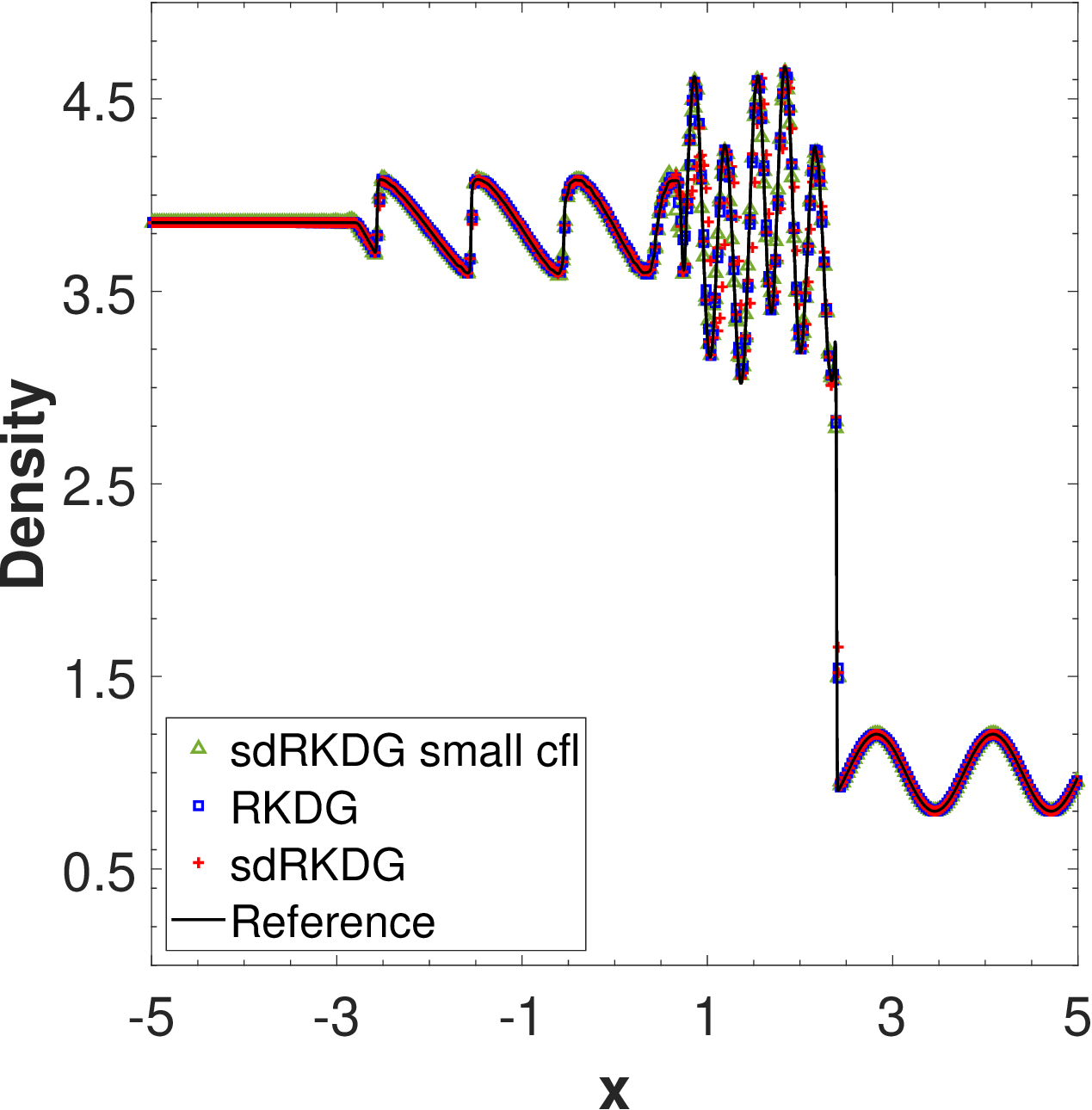}
			\caption{$k=2$}
		\end{subfigure} 
		\begin{subfigure}[t]{.24\textwidth}
			\centering
			\includegraphics[trim=0cm 1cm 0cm 1cm,width=1. \linewidth,height=1.\linewidth]{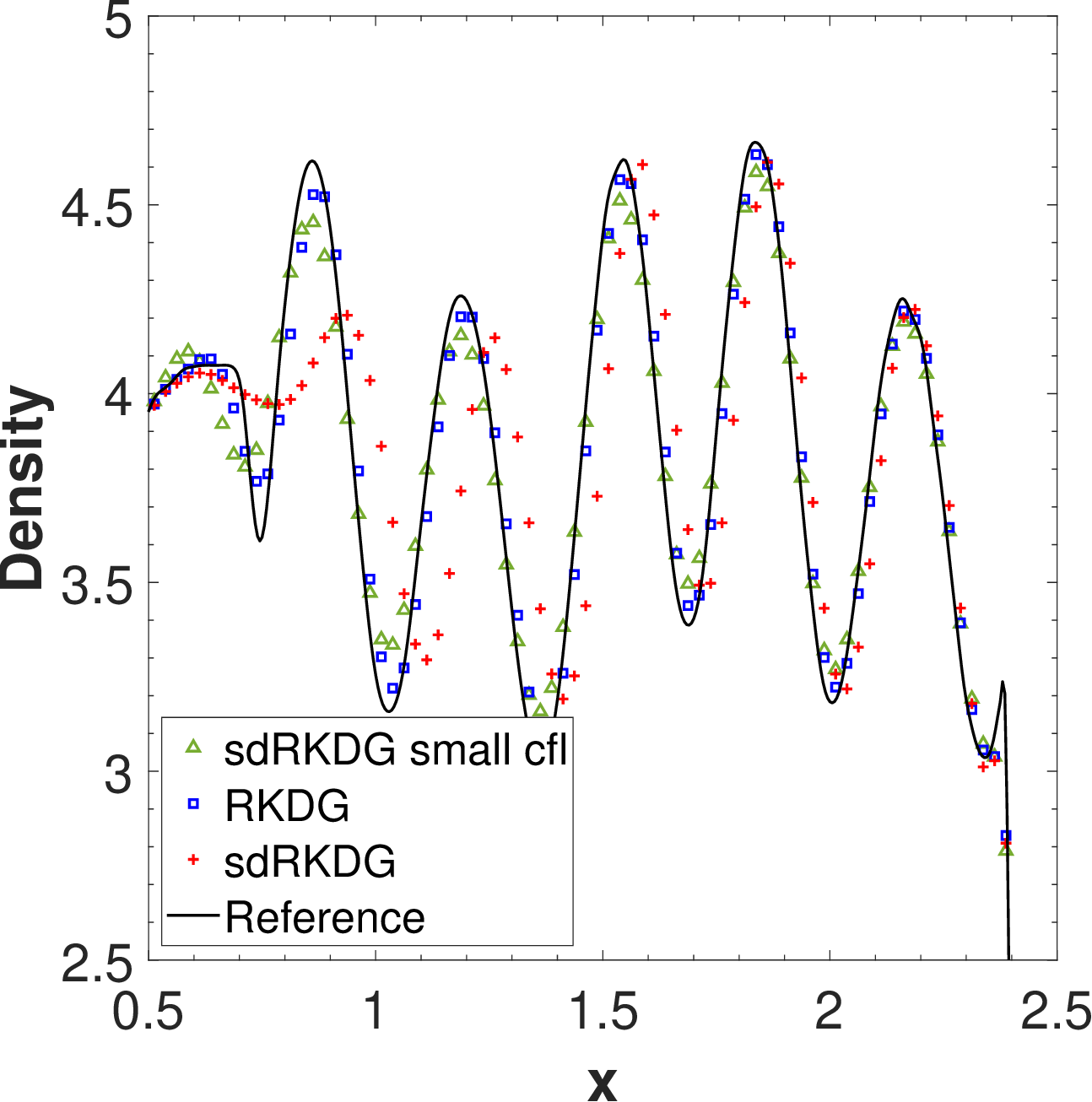}
			\caption{$k=1$ (Zoomed-in)}\label{fig:shuosher-zoom1}
		\end{subfigure}
		\begin{subfigure}[t]{.24\textwidth}
			\centering
			\includegraphics[trim=0cm 1cm 0cm 1cm,width=1. \linewidth,height=1.\linewidth]{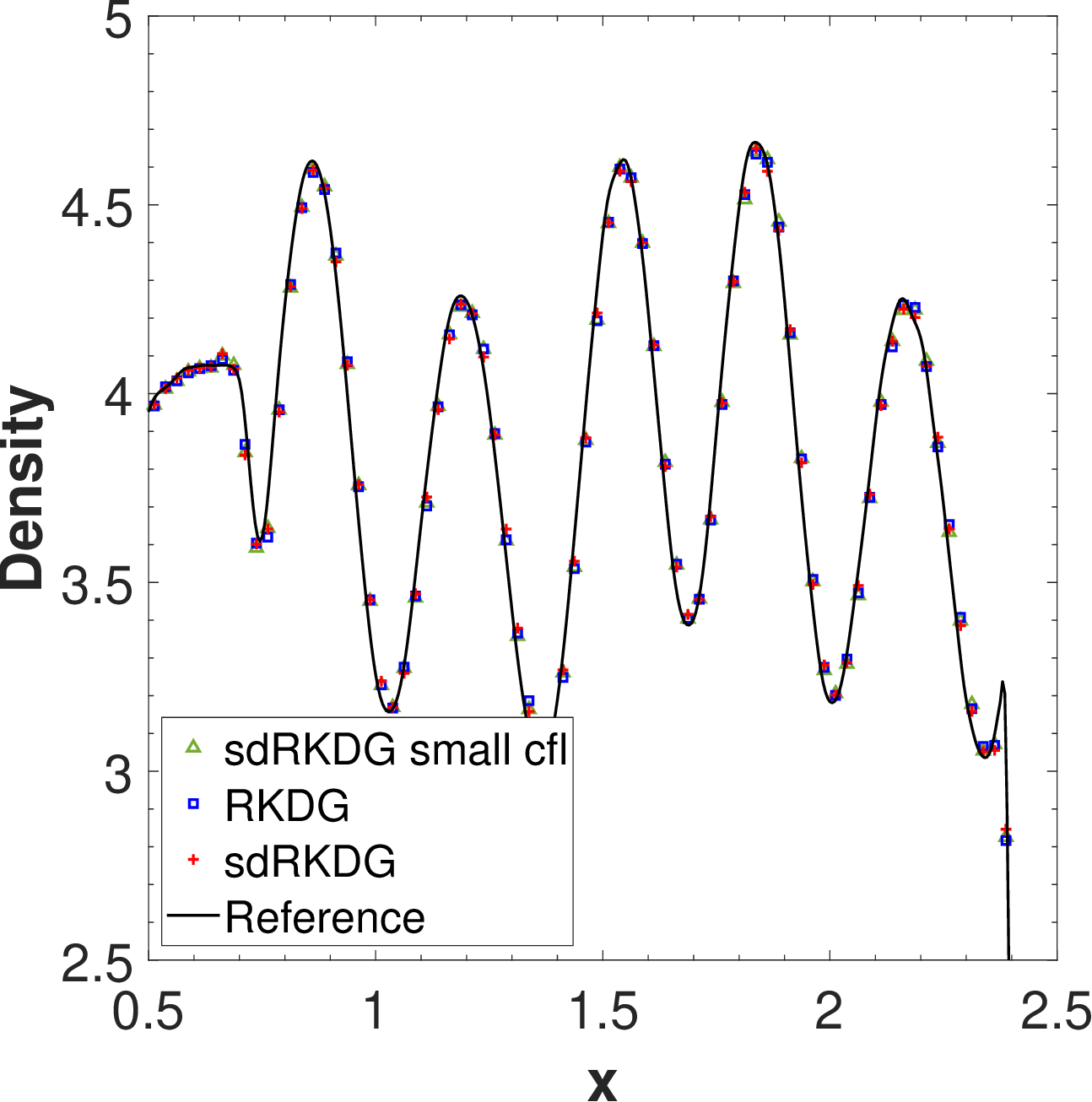}
			\caption{$k=2$ (Zoomed-in)}\label{fig:shuosher-zoom2}
		\end{subfigure} 
		\caption{Solution profiles for the Shu--Osher problem in \cref{ex:shuosher} at $t=1.8$. $M=300$ and $N=400$.  Here “sdRKDG” solutions are computed with the CFL numbers $0.56$ for
$k = 1$ and $0.27$ for $k = 2$; the “sdRKDG small CFL” and “RKDG” solutions are computed with the CFL number $0.3$ for $k = 1$ and $0.18$ for $k = 2$.
}
		\label{fig:shu-osher}\vspace{-0.3cm}
	\end{figure}	
\end{exmp}
\subsubsection{Two-dimensional tests}
 \begin{exmp}[Double Mach reflection]\label{ex:doublemach}
 The Double Mach problem is a benchmark test initially proposed
 in \cite{woodward1984numerical}.
  The computational domain is set as $[0,4] \times[0,1]$. At $t=0$, a strong planar shock wave of Mach 10 is positioned at $x={1}/{6}$, $y=0$ and inclined at a $60^{\circ}$ angle with respect to the $x$-axis. At the top boundary, the exact motion of the Mach 10 shock is imposed for the flow values. At the left, bottom, and right boundaries of the domain, we use inflow, reflexive, and outflow boundary conditions, respectively. We compute the solution up to $t = 0.2$ and use the TVB limiter with $M = 50$. Limited by space, we only present the simulation results with 
  $1960\times480$ mesh cells for $k = 1,2$ in \cref{fig:doublemach-big}. We also present corresponding zoomed-in figures for the double Mach stem in \cref{fig:doublemach-small}. For this problem, with the same meshes and orders of accuracy, the resolutions of sdRKDG and RKDG methods are comparable. 
	\begin{figure}[h!]
	\centering
		\begin{subfigure}[t]{0.49\textwidth}
			\centering
			\includegraphics[trim=1cm 0cm 3cm 0cm, width=\textwidth]{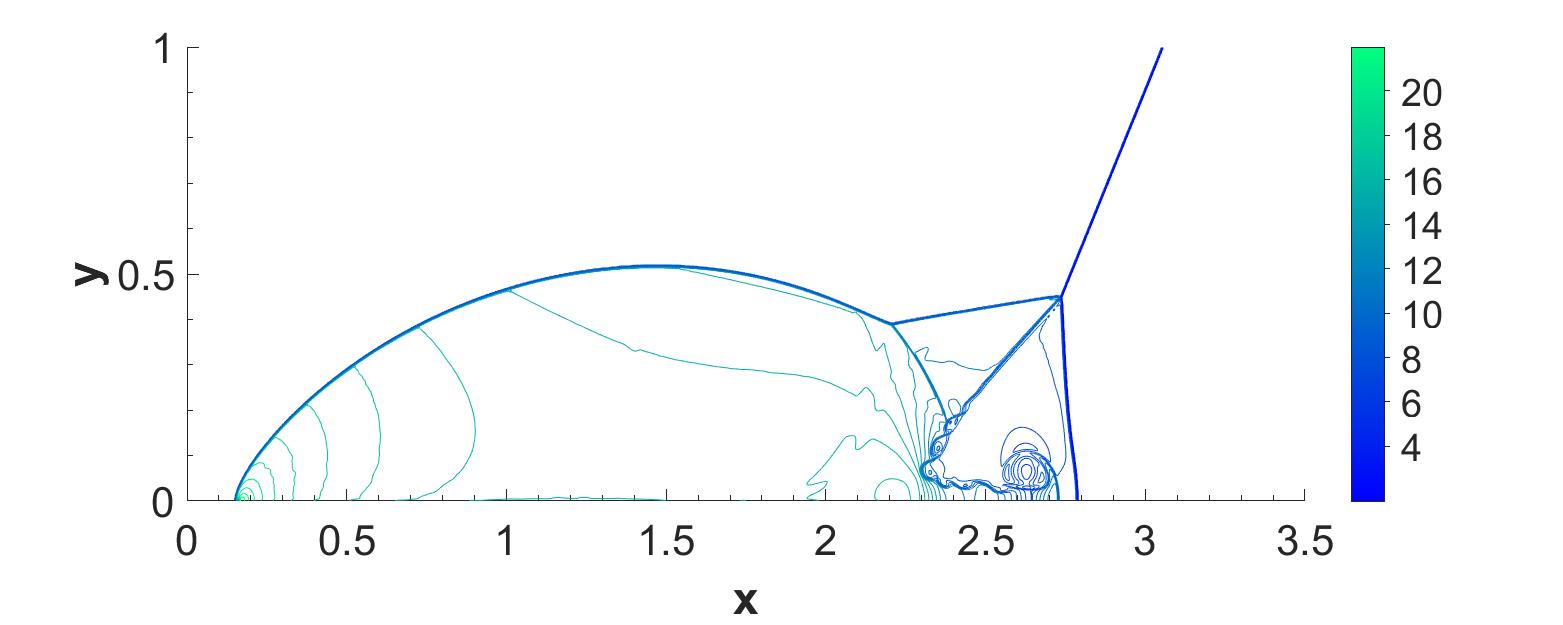}
			\caption{$k=1$, sdRKDG, $1920\times 480$ mesh}
		\end{subfigure}		
		\begin{subfigure}[t]{0.49\textwidth}
			\centering		\includegraphics[trim=1cm 0cm 3cm 0cm, width=\textwidth]{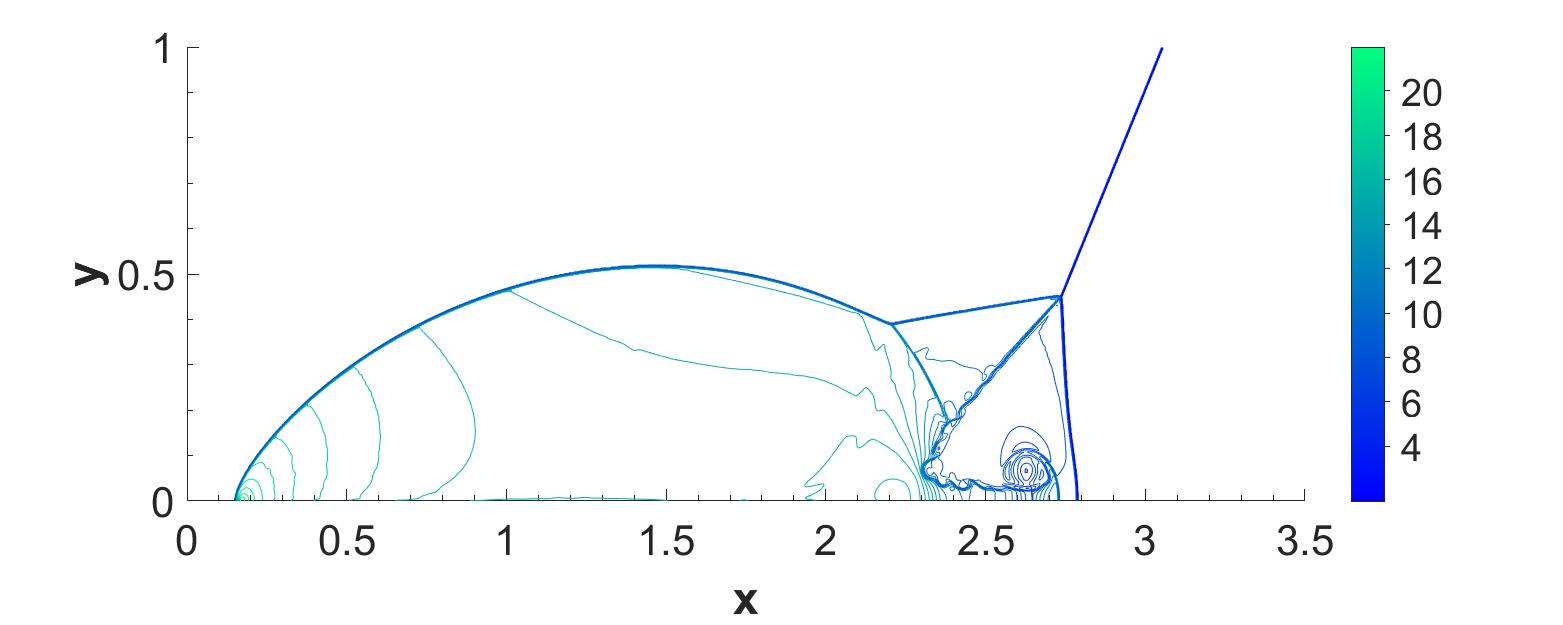}
			\caption{$k=1$, RKDG, $1920\times 480$ mesh}
		\end{subfigure}
  \\
  \begin{subfigure}[t]{0.49\textwidth}
			\centering
			\includegraphics[trim=1cm 0cm 3cm 0cm, width=\textwidth] {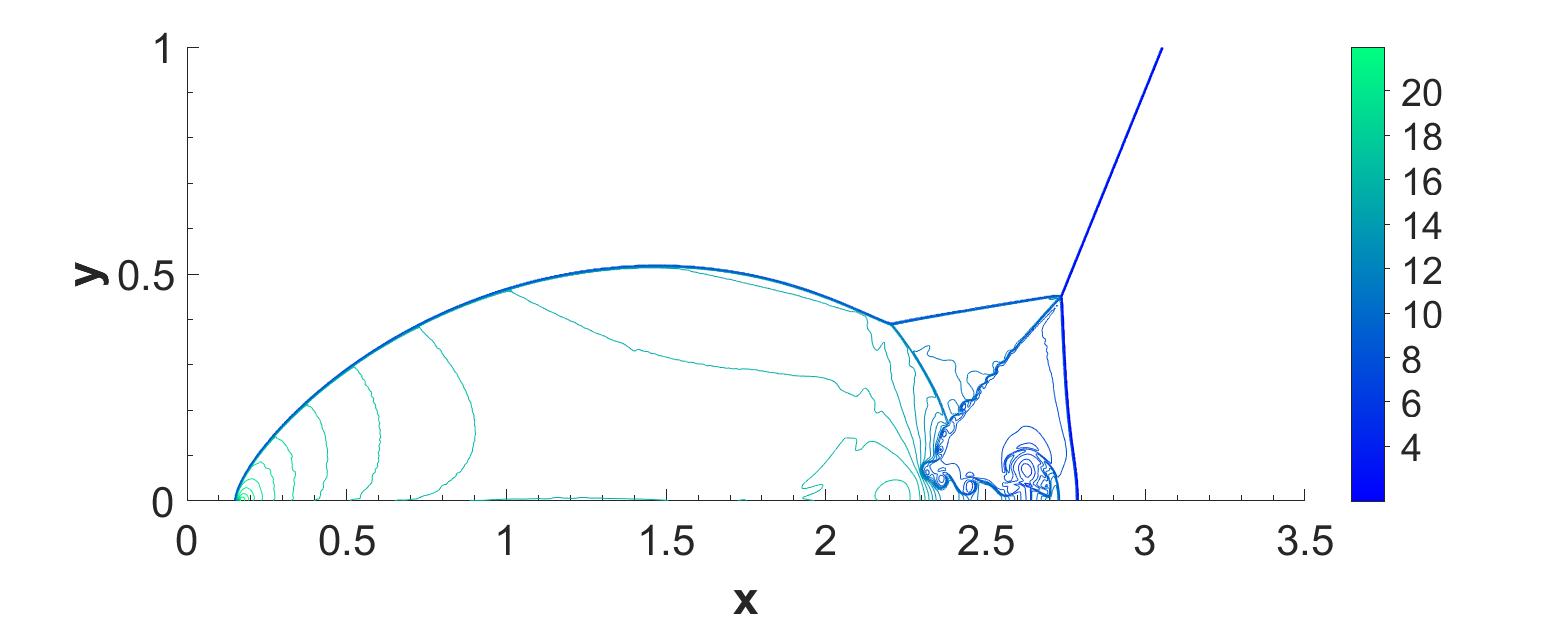}
			\caption{$k=2$, sdRKDG, $1920\times 480$ mesh}
		\end{subfigure} 
		\begin{subfigure}[t]{0.49\textwidth}
			\centering
			\includegraphics[trim=1cm 0cm 3cm 0cm, width=\textwidth]{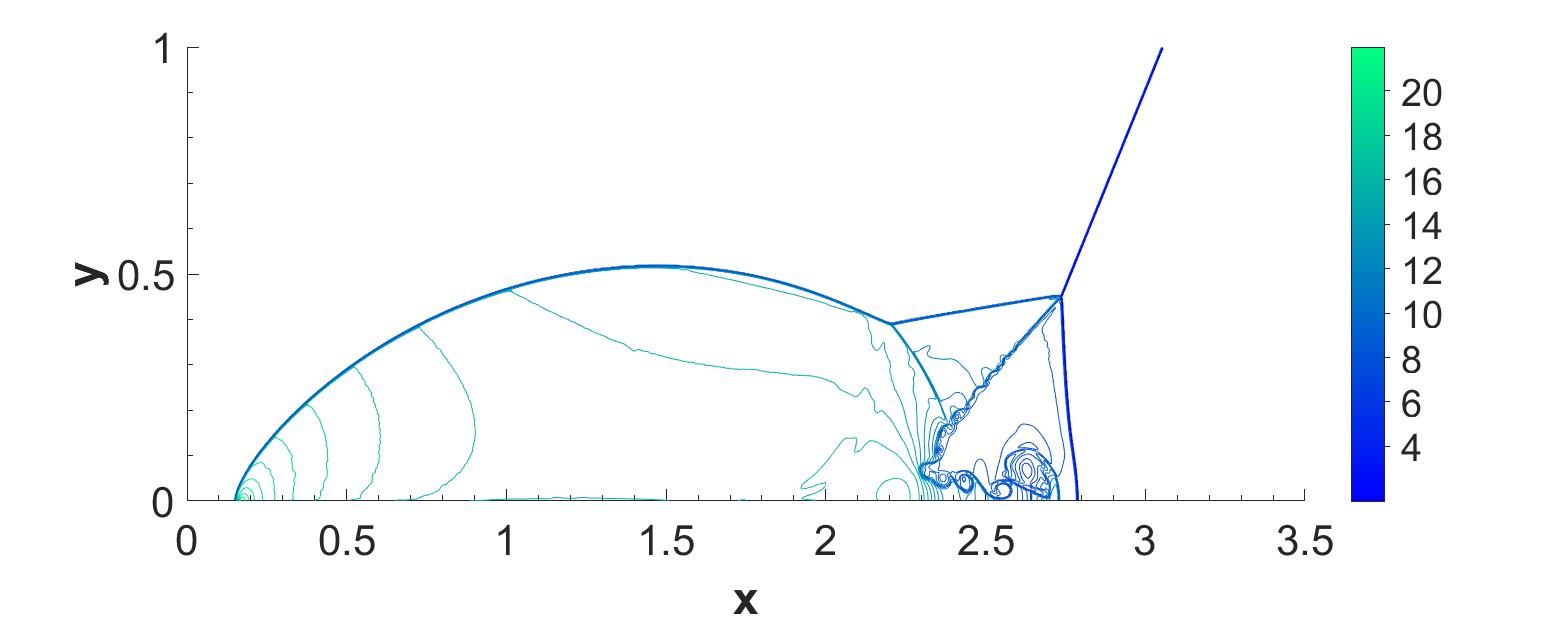}
			\caption{$k=2$, RKDG, $1920\times 480$ mesh}
		\end{subfigure} 
		\caption{Solution profiles for the double Mach problem in \cref{ex:doublemach} at $t=0.2$ with $M=50$. $30$ equally spaced density contours from $1.5$ to $22.7$ are displaced.}
		\label{fig:doublemach-big}\vspace{-0.2cm}
	\end{figure}
	\begin{figure}[h!]
	\centering
		\begin{subfigure}[t]{.4\textwidth}			\includegraphics[width=\textwidth,height=0.7\linewidth]{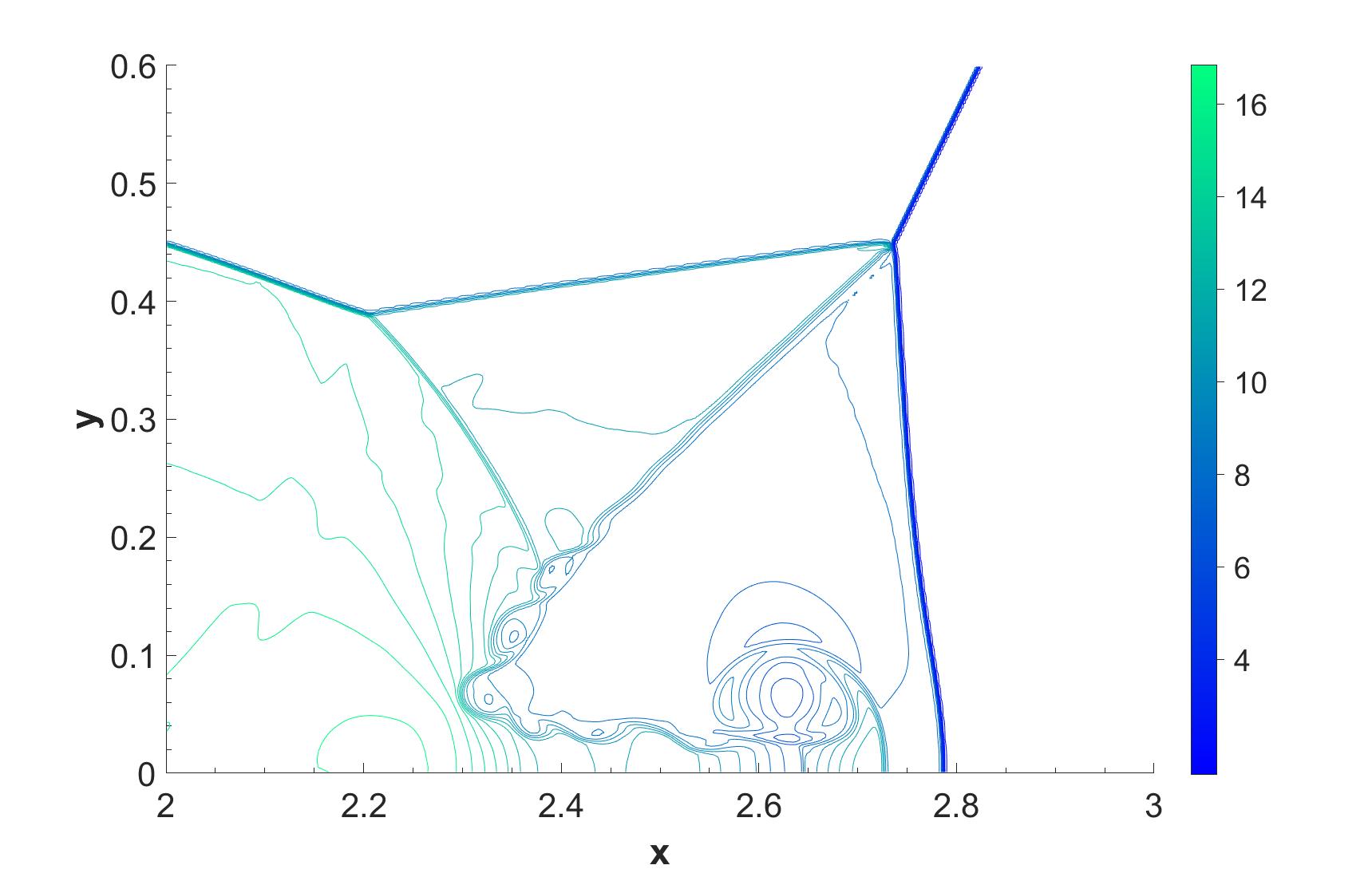}
			\caption{$k=1$, sdRKDG, $1920\times 480$ mesh}
		\end{subfigure}
  \begin{subfigure}[t]{.4\textwidth}
			\includegraphics[width=\textwidth,height=0.7\linewidth]{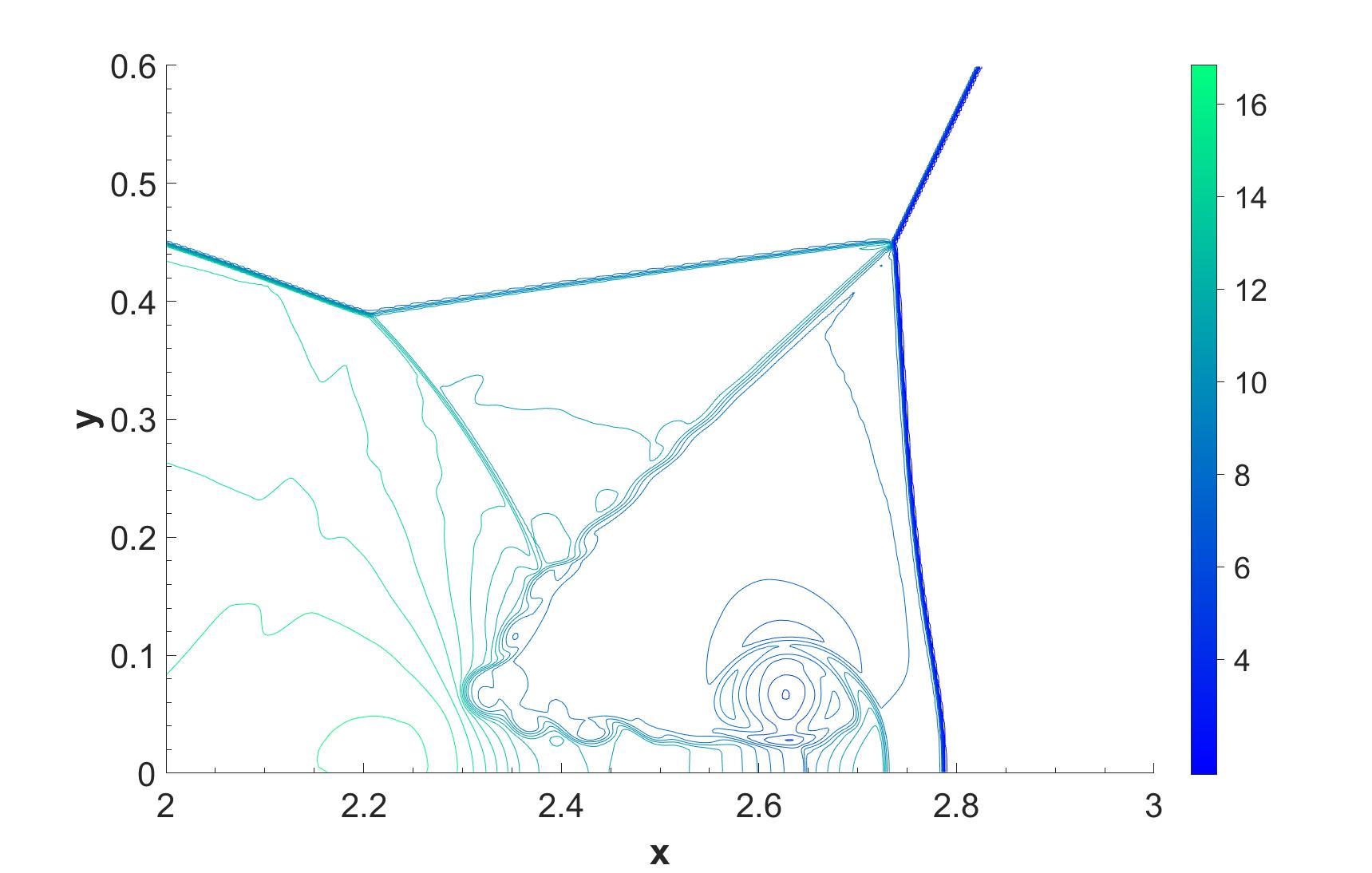}
			\caption{$k=1$, RKDG, $1920\times 480$ mesh}
		\end{subfigure}  
		\begin{subfigure}[t]{.4\textwidth}
			\includegraphics[width=\textwidth,height=0.7\linewidth]{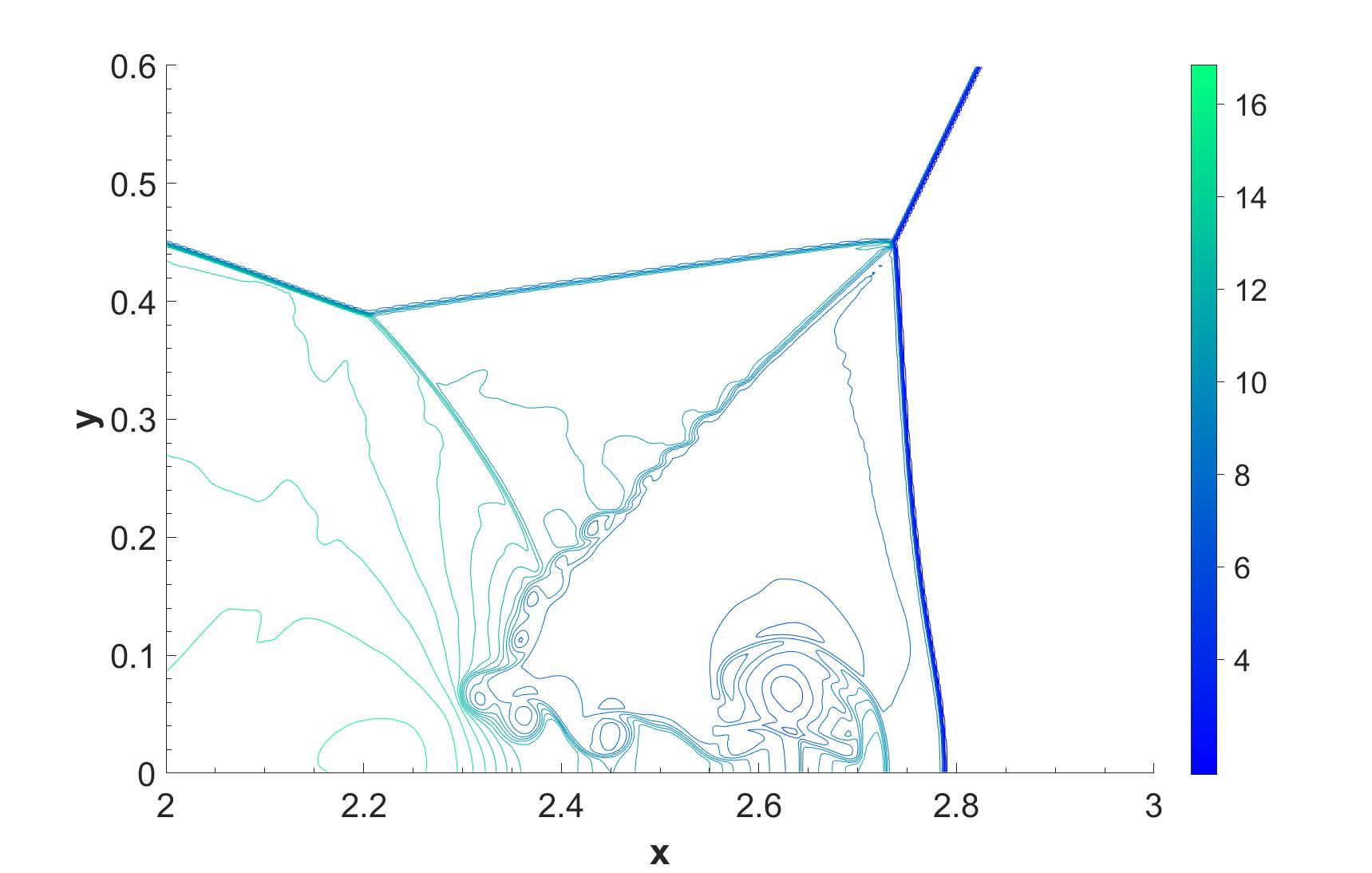}
			\caption{$k=2$, sdRKDG, $1920\times 480$ mesh}
		\end{subfigure}
			\begin{subfigure}[t]{.4\textwidth}
			\includegraphics[width=\textwidth,height=0.7\linewidth]{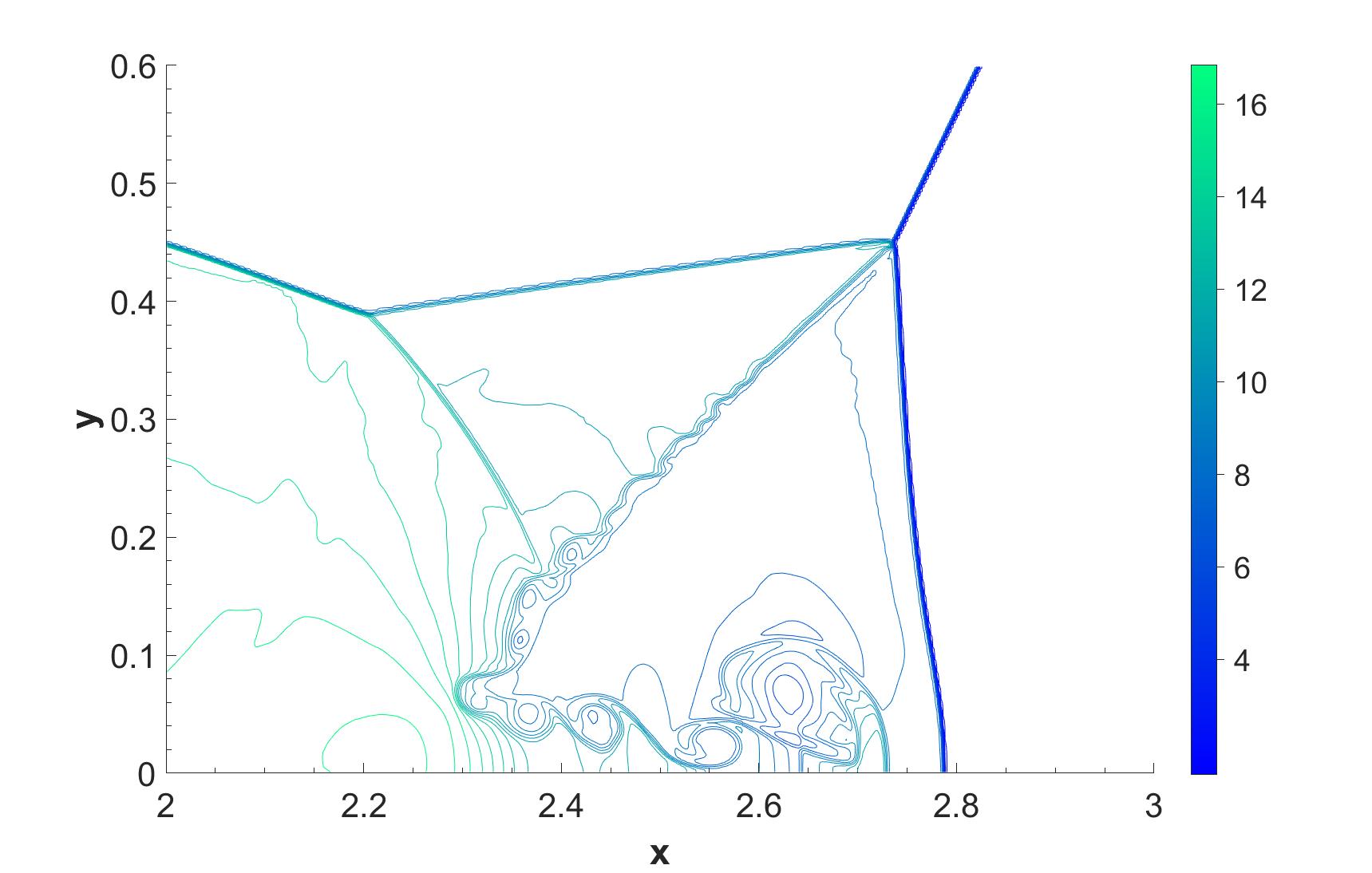}
			\caption{$k=2$, RKDG, $1920\times 480$ mesh}
		\end{subfigure}
		\caption{Zoomed-in double Mach solutions in \cref{fig:doublemach-big} for \cref{ex:doublemach}. }
		\label{fig:doublemach-small}\vspace{-0.3cm}
	\end{figure}
\end{exmp}
\begin{exmp}[Forward facing step]\label{ex:forward-step}
This is another classical test studied in \cite{woodward1984numerical}. In this test, a Mach 3 right-going wave enters a rectangular wind tunnel. The tunnel is 1 length unit wide and 3 length units long, with a step of 0.2 length units high located 0.6 length units from the inlet. Reflective boundary conditions are imposed along the wall of the tunnel, while inflow/outflow boundary conditions are applied at the entrance/exit.  There is a singular point at the corner of the step, but unlike in \cite{woodward1984numerical}, we do not modify our schemes or refine the mesh near the corner in order to evaluate the performance of our schemes in shock wave interactions. We compute the solution up to $t = 4$ and utilize the TVB limiter with a TVB constant $M = 50$. We present the simulation results with $240\times 80$ mesh cells for $k = 1$ and $960\times320$ mesh cells for $k = 1,2$ in \cref{fig:forward-step}. For this problem, the resolutions of sdRKDG and RKDG methods are comparable for the same order of accuracy and mesh.

\begin{figure}[h!]
	\centering
 \begin{subfigure}[t]{.49\textwidth}
			\includegraphics[trim=1cm 0cm 1cm 0cm, width=\textwidth]
                {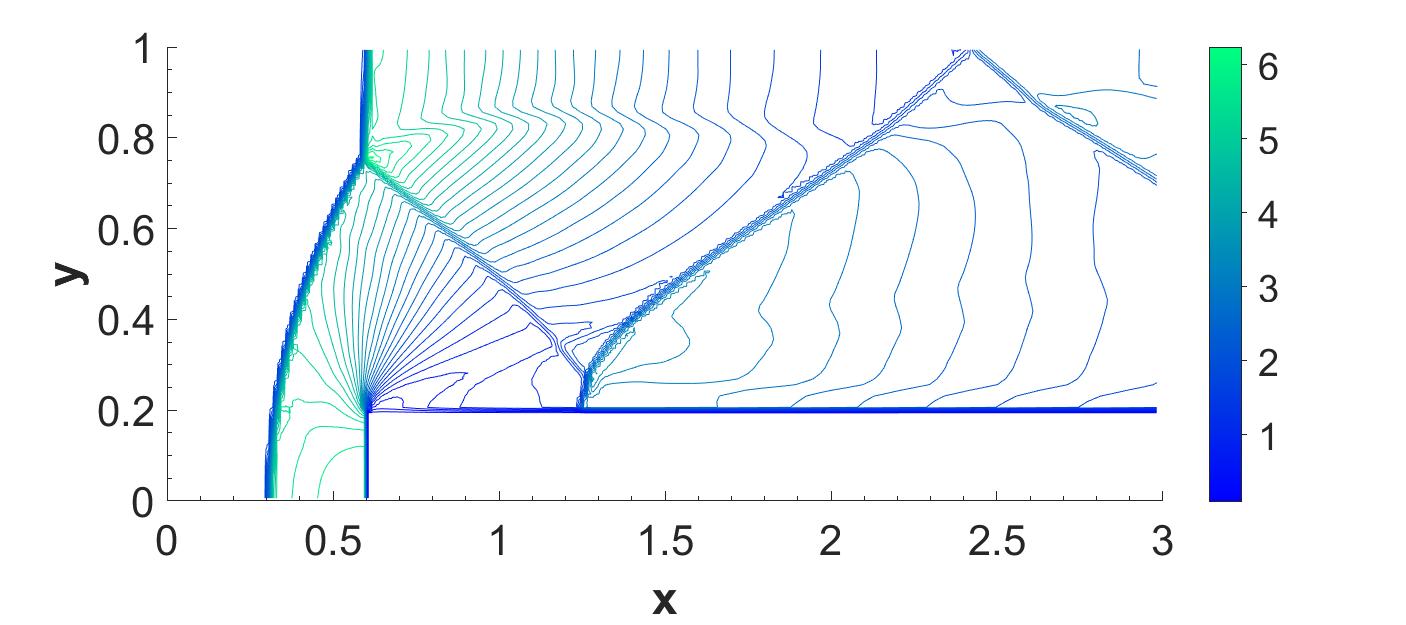}
			\caption{$k=1$, sdRKDG, $240\times 80$ mesh}
		\end{subfigure}
		\begin{subfigure}[t]{.49\textwidth}
			\includegraphics[trim=1cm 0cm 1cm 0cm, width=\textwidth]
                {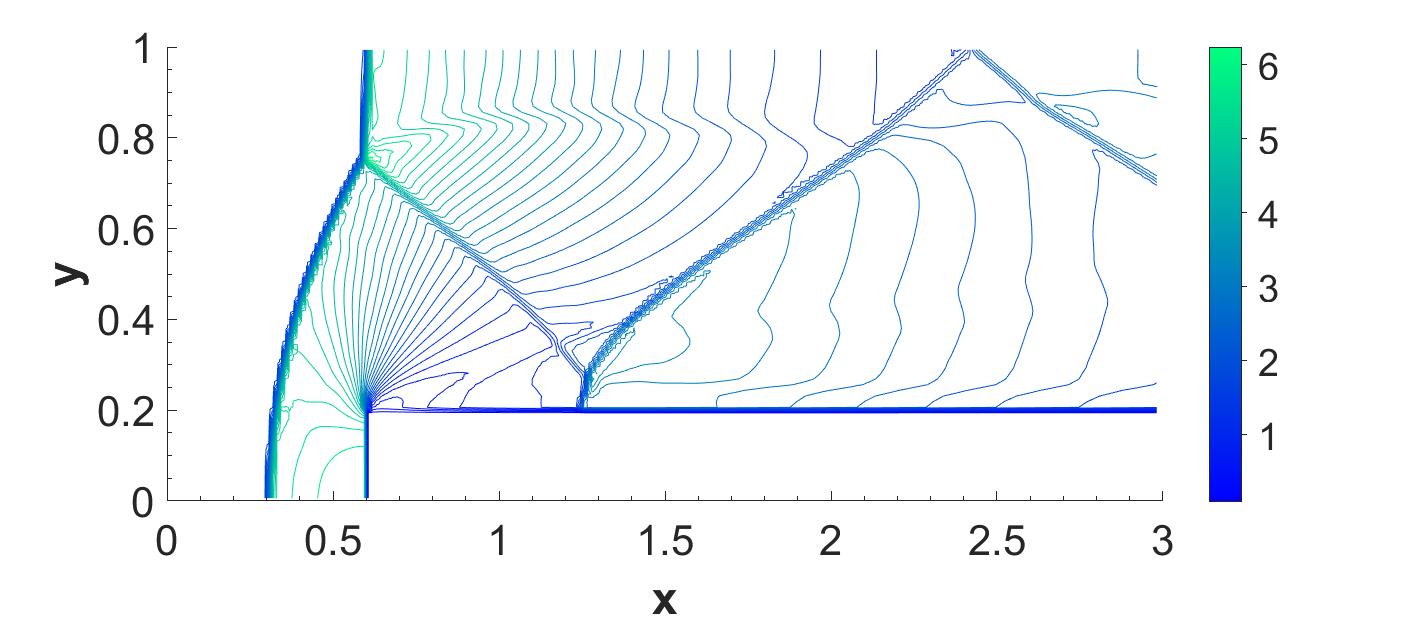}
			\caption{$k=1$, RKDG, $240\times 80$ mesh}
		\end{subfigure}

		\begin{subfigure}[t]{.49\textwidth}
			\includegraphics[trim=1cm 0cm 1cm 0cm, width=\textwidth]
                {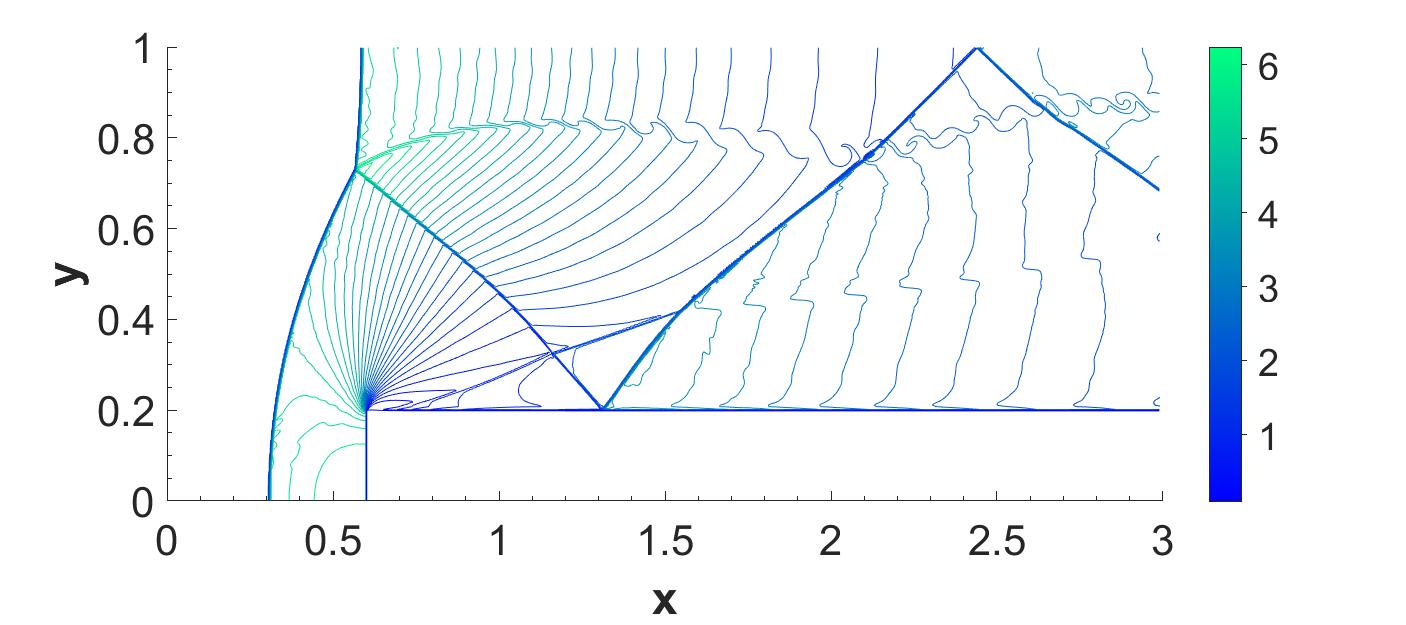}
			\caption{$k=1$, sdRKDG, $960\times 320$ mesh}
		\end{subfigure}
		\begin{subfigure}[t]{.49\textwidth}
			\includegraphics[trim=1cm 0cm 1cm 0cm, width=\textwidth]
                {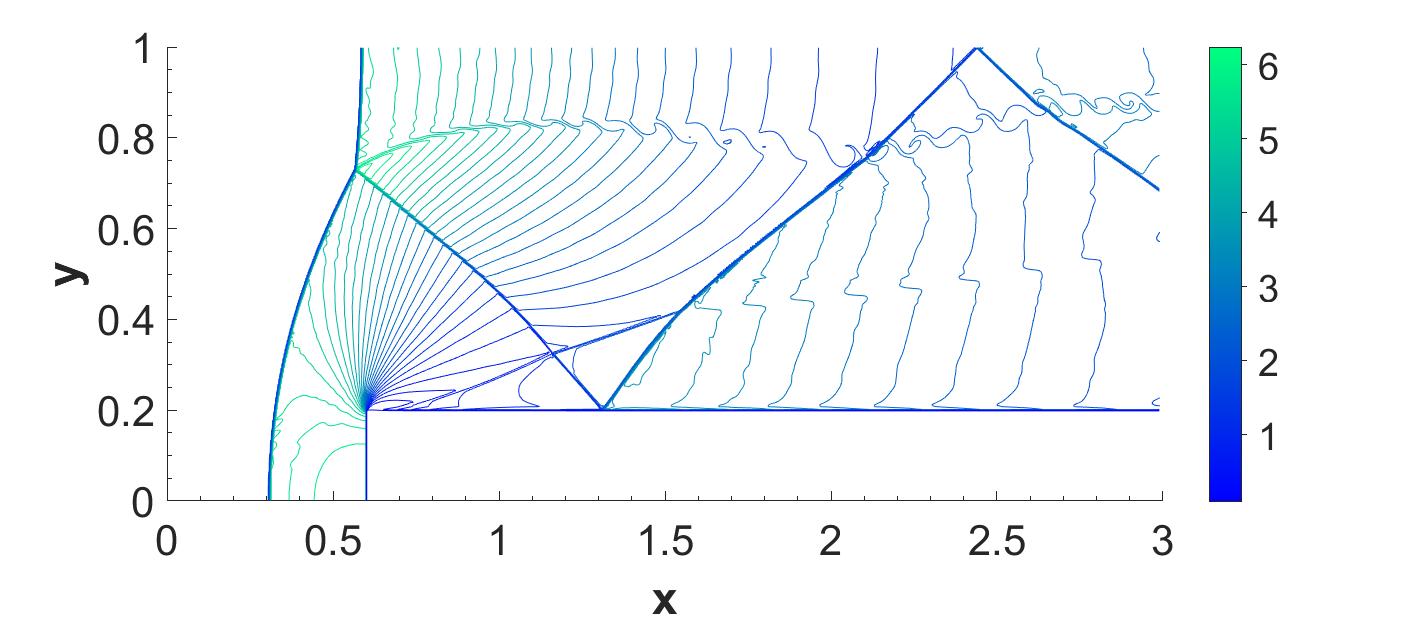}
			\caption{$k=1$, RKDG, $960\times 320$ mesh}
		\end{subfigure}
      \begin{subfigure}[t]{.49\textwidth}
			\includegraphics[trim=1cm 0cm 1cm 0cm, width=\textwidth]
                {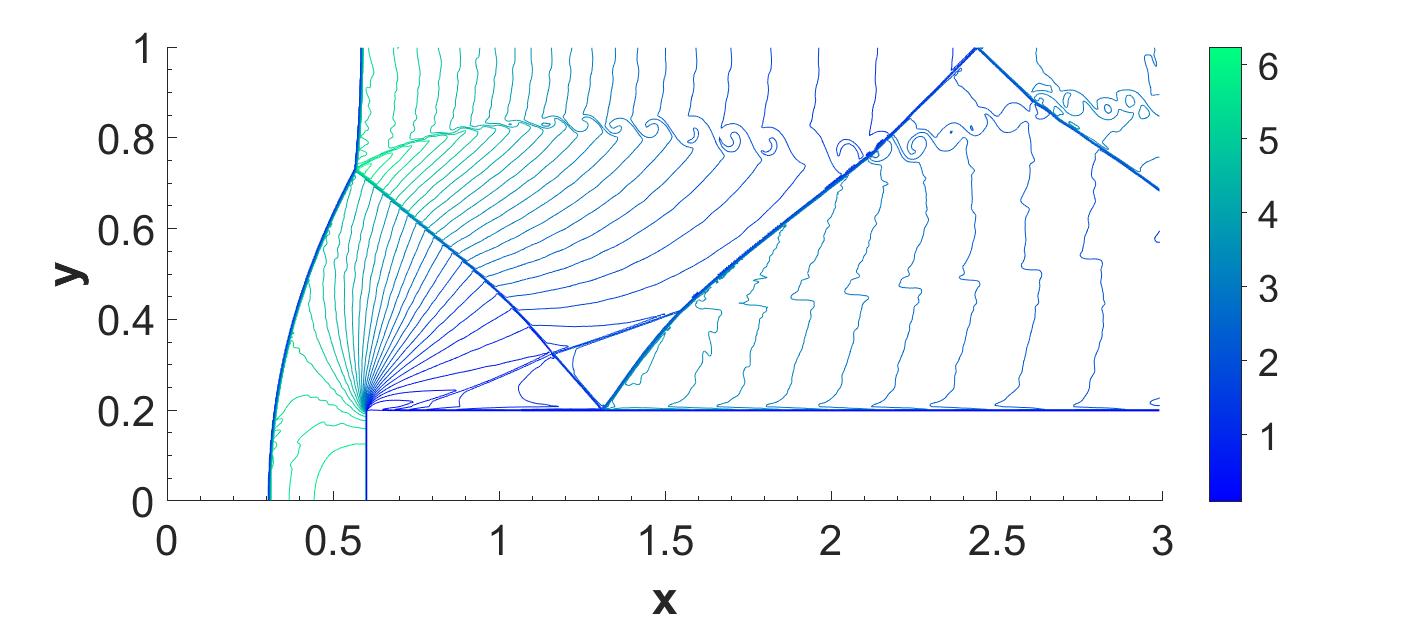}
			\caption{$k=2$, sdRKDG, $960\times 320$ mesh}
		\end{subfigure}
		\begin{subfigure}[t]{.49\textwidth}
			\includegraphics[trim=1cm 0cm 1cm 0cm, width=\textwidth]
                {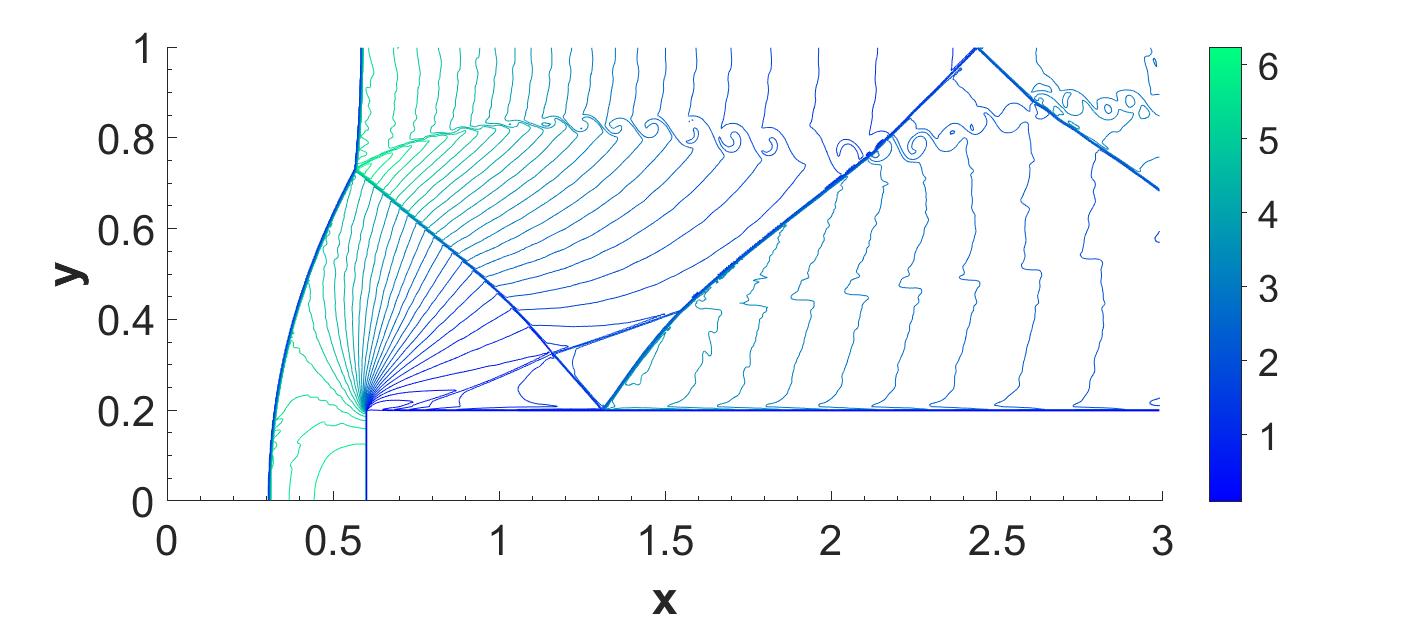}
			\caption{$k=2$, RKDG, $960\times 320$ mesh}
		\end{subfigure}
		\caption{Solution profiles for the forward step problem in \cref{ex:forward-step} at $t=4$ with $M=50$. $30$ equally spaced density contours from $0.090388$ to $6.2365$ are displaced. }
		\label{fig:forward-step}\vspace{-0.2cm}
	\end{figure}
\end{exmp}

\section{Conclusions}\label{sec:conclusions}

In this paper, we propose a class of RKDG schemes for hyperbolic conservation laws with stage-dependent polynomial spaces. The main idea is to blend two different DG operators, with $\mathcal{P}^{k}$ and $\mathcal{P}^{k-1}$ piecewise polynomials, at different stages of the RK scheme of $(k+1)$th order. The new sdRKDG method features a reduced number of floating-point operations and inherits the local conservation and TVD/TVB properties of the standard RKDG method. The von Neumann analysis reveals that the method may remain stable under a more relaxed CFL condition that is 30\% to 70\% larger compared to the standard RKDG method, and attains optimal convergence rate for the linear advection equation with constant coefficients. For problems with sonic points, our numerical tests indicate that the sdRKDG method may maintain optimal accuracy only when all spatial operators in the final stages use $\mathcal{P}^k$ piecewise polynomials. Additional numerical tests are presented to exhibit the effectiveness of the sdRKDG method for problems with discontinuous solutions. 

Although in this paper, we only focus on the hybridization of the $\mathcal{P}^k$ and $\mathcal{P}^{k-1}$ DG operators, the sdRKDG method exemplifies a general approach for discretizing the PDEs beyond MOL and with different stage operators. With potentially another type of spatial discretizations, the variety of selection of the stage operators allows additional flexibility and creates opportunities to improve the standard MOL schemes based on the multi-stage RK discretization. This framework may bring in a new direction with refreshing perspectives that could contribute to the further development and understanding of the existing method, not only in the RKDG framework but also with other types of discretization.

\appendix
\section{Proof of \cref{thm:lwthm}}\label{sec:conservation-proof}
\begin{lemma}\label{lem:lip}
For any $u_h, v_h \in \mathcal{V}_h^k$ with $\nm{u_h}_{L^\infty},\nm{v_h}_{L^\infty}\leq C$, we have
\begin{equation*}
\nm{
    \nabla^\DG_{\ell}\cdot f\left(u_h\right) - \nabla^\DG_{\ell}\cdot f\left(v_h\right) 
    }_{L^\infty(K)}\leq Ch^{-1}\nm{u_h-v_h}_{L^\infty(\widehat{B}_{K})}, \quad \ell = k-1, k.
\end{equation*}
\end{lemma}
\begin{proof}[Proof of  \cref{lem:lip}]
By the norm equivalence, we have 
\begin{equation}\label{eq:lip1}
    \nm{
    \nabla^\DG_{\ell}\cdot f\left(u_h\right) - \nabla^\DG_{\ell}\cdot f\left(v_h\right) 
    }_{L^\infty(K)}\leq Ch^{-\frac{d}{2}}     \nm{\nabla^\DG_{\ell}\cdot f\left(u_h\right) - \nabla^\DG_{\ell}\cdot f\left(v_h\right) 
    }_{L^2(K)}.
\end{equation}
Therefore, to prove \cref{lem:lip},  it suffices to show that 
\begin{equation}\label{eq:estnablaf}
    \nm{\nabla^\DG_{\ell}\cdot f\left(u_h\right) - \nabla^\DG_{\ell}\cdot f\left(v_h\right) 
    }_{L^2(K)}  \leq Ch^{\frac{d}{2}-1}\nm{u_h-v_h}_{L^\infty(\widehat{B}_{K})}.
\end{equation}
Indeed, for any test function $w_h \in \mathcal{V}_h^\ell$, we can apply the definition of $\nabla_\ell^{\mathrm{DG}}\cdot f$, the Lipschitz continuity of $f$ and $\widehat{f\cdot \nu_{e,K}}$, and the norm equivalence to get 
\begin{equation*}
\begin{aligned}
    &\int_{K}\left(\nabla^\DG_k\cdot f\left(u_h\right) - \nabla^\DG_k\cdot f\left(v_h\right)\right)w_h \dd x\\
    = &
    -\int_{K} \left(f\left(u_{h}\right)-f\left(v_{h}\right)\right)\cdot \nabla w_h \dd x+ \sum_{e\in \partial K}\int_{e} \left(\widehat{f\cdot \nu_{e,K}}(u_h) - \widehat{f\cdot \nu_{e,K}}(v_h)\right) w_h \dd l\\
    \leq &C\nm{u_{h} - v_{h}}_{L^\infty(\widehat{B}_K)}
    \left(\int_{K} | \nabla w_h| \dd x + \sum_{e\in \partial K}\int_{e} |w_h| \dd l\right)\\
    \leq &Ch^{\frac{d}{2}}\nm{u_{h} - v_{h}}_{L^\infty(\widehat{B}_K)}\left(
      \nm{\nabla w_h}_{L^2(K)} + h^{-\frac{1}{2}} \nm{w_h}_{L^2(\partial K)}\right).
\end{aligned}
\end{equation*}
Applying the inverse estimates
$\nm{\nabla w_h}_{L^2(K)}\leq Ch^{-1}\nm{w_h}_{L^2(K)}$ and $\nm{w_h}_{L^2(\partial K)}\leq Ch^{-1/2}\nm{w_h}_{L^2(K)}$, it gives
\begin{equation*}
\int_{K}\left(\nabla^\DG_\ell\cdot f\left(u_h\right) - \nabla^\DG_\ell\cdot f\left(v_h\right)\right)w_h \dd x
    \leq Ch^{\frac{d}{2}-1}\nm{u_h-v_h}_{L^\infty(\widehat{B}_K)}\nm{w_h}_{L^2(K)}.
\end{equation*}
We can take $w_h = \nabla^\DG_\ell\cdot f\left(u_h\right) - \nabla^\DG_\ell\cdot f\left(v_h\right)$ to prove \eqref{eq:estnablaf}, which together with \eqref{eq:lip1} implies \cref{lem:lip}. 
\end{proof}
\begin{proof}[A sketch of the proof of \cref{thm:lwthm}]
The proof of the theorem is similar to \cite[Appendix A]{chen2023runge}, but one needs to pay attention to the difference between $\widehat{B}_K$ (the neighbor of $K$) and $B_K$ (the stencil of sdRKDG scheme). 

Testing \eqref{eq:sdRKDG-last} with $1_K$, the characteristic function on $K$, we have 
\begin{equation*}
\bar{u}_h^{n+1} = \bar{u}_h^{n} -\frac{\Delta t}{|K|} \sum_{e\in \partial K} g_{e,K}(u_h^n), \text{ with }  g_{e,K}(u_h^n) = \int_{e} \left(\sum_{i = 1}^s b_i \widehat{f\cdot\nu_{e,K}}\left(u_h^{(i)}\right)\right) \dd l, 
\end{equation*}
and the relationship between $u_h^{(i)}$ and $u_h^n$ is defined inductively in \eqref{eq:sdRKDG}. According to \cite{shi2018local}, or \cite[Theorem A.1]{chen2023runge}, to prove the conservative property of the scheme, one only needs to verify that $g_{e,K}(u_h^n)$ is consistent, bounded, and anti-symmetric. The proofs of consistency and anti-symmetry are trivial, similar to those in \cite{chen2023runge}. The major complication is to show the boundedness: 
\begin{equation}\label{eq:bounded}
    |g_{e,K}(u_h^n) - g_{e,K}(v_h^n)|\leq C\nm{u_h^n - v_h^n}_{L^{\infty}(B_K)}h^{d-1}, \quad \forall u_h^n, v_h^n \in \mathcal{V}_h^k. 
\end{equation}
Here $B_K$ is the stencil of the sdRKDG scheme for acquiring $u_h^{n+1}$ from $u_h^n$ on $K$. 

Indeed, with the Lipschitz continuity of $\widehat{f\cdot\nu_{e,K}}$, we have 
\begin{equation}\label{eq:bounded-1}
\begin{aligned}
    |g_{e,K}(u_h^n) - g_{e,K}(v_h^n)|\leq &\,\sum_{i = 1}^s\int_e |b_i||\widehat{f\cdot\nu_{e,K}}(u_h^{(i)}) - \widehat{f\cdot\nu_{e,K}}(v_h^{(i)})|\dd l\\
    \leq&\,  C \sum_{i = 1}^s  |b_i||e|\nm{u_h^{(i)}-v_h^{(i)}}_{L^\infty(\widehat{B}_K)}.
\end{aligned}
\end{equation}
To estimate $\nm{u_h^{(i)}-v_h^{(i)}}_{L^\infty(\tilde{K})}$ on one mesh cell with ${\tilde{K}\in \widehat{B}_K}$, we use the definition in \eqref{eq:sdRKDG-in} and the triangle inequality to get
\begin{equation*}
    \begin{aligned}
    &\nm{u_h^{(i)}-v_h^{(i)}}_{L^\infty(\tilde{K})}\\
    \leq &\nm{u_h^n-v_h^n}_{L^\infty(\tilde{K})} + \Delta t\sum_{j = 1}^{i-1}|a_{ij}|\nm{
    \nabla^\DG_{d_{ij}}\cdot f\left(u_h^{(j)}\right) - \nabla^\DG_{d_{ij}}\cdot f\left(v_h^{(j)}\right) 
    }_{L^\infty(\tilde{K})}. 
    \end{aligned}
\end{equation*}
With \cref{lem:lip} and the CFL condition $\Delta t/h \leq \lambda_0$, we can inductively show that $\nm{u_h^{(i)}-v_h^{(i)}}_{L^\infty(\tilde{K})} \leq C\nm{u_h^{n}-v_h^{n}}_{L^\infty({B}_{K})}$. Then using \eqref{eq:bounded-1}, \eqref{eq:bounded} is proved. 
\end{proof}

\bibliography{reference_abbr}  
\bibliographystyle{siamplain}

\end{document}

%% file: ex_shared.tex
\usepackage{lipsum}
\usepackage{amsfonts}
\usepackage{graphicx}
\usepackage{epstopdf}
\usepackage{algorithmic}
\ifpdf
  \DeclareGraphicsExtensions{.eps,.pdf,.png,.jpg}
\else
  \DeclareGraphicsExtensions{.eps}
\fi

\headers{The sdRKDG method for hyperbolic conservation laws}{Q. Chen, Z. Sun and Y. Xing}

\title{The Runge--Kutta discontinuous Galerkin method with stage-dependent polynomial spaces for hyperbolic conservation laws}

\author{Qifan Chen\thanks{Department of Mathematics, The Ohio State University, Columbus, OH 43210, USA. (\email{chen.11010@osu.edu}) }
  \and Zheng Sun\thanks{Department of Mathematics, The University of Alabama,
		Tuscaloosa, AL 35487, USA. (\email{zsun30@ua.edu}) The work of this author is partially supported by the NSF grant DMS-2208391.} \and Yulong Xing\thanks{Department of Mathematics, The Ohio State University,
		Columbus, OH 43210, USA. (\email{xing.205@osu.edu}) The work of this author is partially supported by the NSF grant DMS-1753581 and DMS-2309590.}}

\usepackage{amsopn}
\DeclareMathOperator{\diag}{diag}
